\documentclass[10pt]{article}
\usepackage{graphicx}
\usepackage{amsmath, amsthm, amssymb, tikz}
\usepackage{amsfonts}
\usepackage{fullpage}
\usepackage{ytableau} 
\usepackage{hyperref}  
\usepackage{setspace} 
\usepackage{algorithm} 
\usepackage{algpseudocode} 

\algrenewcommand\algorithmicrequire{\textbf{Input:}}
\algrenewcommand\algorithmicensure{\textbf{Output:}}
\newcommand{\JdT}{\textsc{JdT}}

\newtheorem{theorem}{Theorem}[section]
\newtheorem{proposition}[theorem]{Proposition}
\newtheorem{lemma}[theorem]{Lemma}
\newtheorem{corollary}[theorem]{Corollary}
\newtheorem{definition}[theorem]{Definition}
\newtheorem{claim}[theorem]{Claim}
\newtheorem{remark}[theorem]{Remark}

\newtheorem{question}[theorem]{Question}

\newtheorem{observation}[theorem]{Observation}
\newtheorem{example}[theorem]{Example}

\newcommand{\dem}{\bf} 
\newcommand{\standardboxsize}{1.07em}

\newcommand{\la}{\lambda}
\newcommand{\hD}{{\hat D}}

\newcommand{\then}{\Longrightarrow}
\newcommand\Sc{{\mathcal S}}

\newcommand{\mone}{{\text{--}1}}
\newcommand{\mtwo}{{\text{--}2}}
\newcommand{\ts}[1]{\text{\small #1}}
\newcommand{\bnum}[1]{{\mathbf{#1}}}

\DeclareMathOperator\maj{maj} \DeclareMathOperator\SYT{SYT}
\DeclareMathOperator\YT{YT} \DeclareMathOperator\PT{PT}
\DeclareMathOperator\SPT{SPT} \DeclareMathOperator\BS{BS}
 
 \DeclareMathOperator\sh{sh}
\DeclareMathOperator\zigzag{zigzag} \DeclareMathOperator\row{row}
\DeclareMathOperator\col{col} \DeclareMathOperator\vol{vol}
 \DeclareMathOperator\gf{gf}

\newcommand{\bbc}{\mathbb{C}}
\newcommand{\bbn}{\mathbb{N}}
\newcommand{\bbz}{\mathbb{Z}}

\newcommand{\bbr}{\mathbb{R}}

\newcommand{\tC}{{\widetilde{C}}}

\newcommand{\Des}{{\rm{Des}}}
\newcommand{\Inv}{{\rm{Inv}}}
\newcommand{\inv}{{\rm{inv}}}
\newcommand{\Winv}{{\rm{Winv}}}
\newcommand{\winv}{{\rm{winv}}}
\newcommand{\des}{{\rm{des}}}

\newcommand{\sgn}{{\rm{sign}}}
\newcommand{\height}{{\rm{height}}}
\newcommand{\Avoid}{{\rm{Avoid}}}
\newcommand{\Prob}{{\rm{Prob}}}

\begin{document}

\title{Enumeration of Standard Young Tableaux} 


\author{Ron M. Adin   \and   Yuval Roichman}

\date{August 31, 2014}

\maketitle


\section{Introduction}\label{AR_s:introduction}

\subsection{Appetizer}\label{AR_s:appetizer}

Consider throwing balls labeled $1, 2, \ldots, n$ into
a V-shaped bin with perpendicular sides.
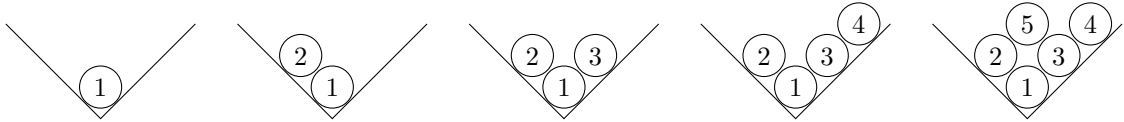
\begin{center}
\[
\begin{tikzpicture}[scale=0.14]

\draw (-97,9) -- (-88,0) -- (-79,9); \draw (-88,3) circle (2);
\draw(-88,1.25) node[above]{$1$};

\draw (-75,9) -- (-66,0) -- (-57,9); \draw (-66,3) circle (2);
\draw(-66,1.25) node[above]{$1$}; \draw (-69,6) circle (2);
\draw(-69,4.25) node[above]{$2$};

\draw (-53,9) -- (-44,0) -- (-35,9); \draw (-44,3) circle (2);
\draw(-44,1.25) node[above]{$1$}; \draw (-47,6) circle (2);
\draw(-47,4.25) node[above]{$2$}; \draw (-41,6) circle (2);
\draw(-41,4.25) node[above]{$3$};

\draw (-31,9) -- (-22,0) -- (-13,9); \draw (-22,3) circle (2);
\draw(-22,1.25) node[above]{$1$}; \draw (-25,6) circle (2);
\draw(-25,4.25) node[above]{$2$}; \draw (-19,6) circle (2);
\draw(-19,4.25) node[above]{$3$}; \draw (-16,9) circle (2);
\draw(-16,7.25) node[above]{$4$};

\draw (-9,9) -- (0,0) -- (9,9); \draw (0,3) circle (2);
\draw(0,1.25) node[above]{$1$}; \draw (-3,6) circle (2);
\draw(-3,4.25) node[above]{$2$}; \draw (3,6) circle (2);
\draw(3,4.25) node[above]{$3$}; \draw (6,9) circle (2);
\draw(6,7.25) node[above]{$4$}; \draw (0,9) circle (2);
\draw(0,7.25) node[above]{$5$};

\end{tikzpicture}
\]
\end{center}

\begin{question}\label{AR_q:appet}
What is the total number of resulting configurations?
How many configurations are there of any particular shape?
\end{question}


In order to answer these questions, at least partially, recall the
symmetric group $\Sc_n$ of all permutations of the numbers $1,
\ldots, n$. An {\dem involution} is a permutation $\pi \in \Sc_n$
such that $\pi^2$ is the identity permutation.


\begin{theorem}\label{AR_t:appet1}
The total number of configurations of $n$ balls is equal to the
number of involutions in the symmetric group $\Sc_n$.
\end{theorem}

Theorem~\ref{AR_t:appet1} may be traced back to Frobenius and Schur.
A combinatorial proof will be outlined in Section~\ref{AR_s:JdT} 
(see Corollary~\ref{AR_t:RSK_cor1}).

\begin{example}
There are four configurations on three balls.
Indeed,
\[
\{\pi\in \Sc_3 \,:\, \pi^2=1\}=\{123,132,213,321\}.
\]
\end{example}

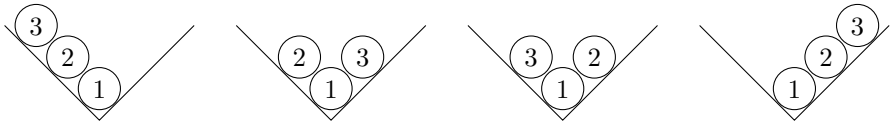
\begin{center}
\[
\begin{tikzpicture}[scale=0.14]



\draw (-75,9) -- (-66,0) -- (-57,9); \draw (-66,3) circle (2);
\draw(-66,1.25) node[above]{$1$}; \draw (-69,6) circle (2);
\draw(-69,4.25) node[above]{$2$}; \draw (-72,9) circle (2);
\draw(-72,7.25) node[above]{$3$};

\draw (-53,9) -- (-44,0) -- (-35,9); \draw (-44,3) circle (2);
\draw(-44,1.25) node[above]{$1$}; \draw (-47,6) circle (2);
\draw(-47,4.25) node[above]{$2$}; \draw (-41,6) circle (2);
\draw(-41,4.25) node[above]{$3$};

\draw (-31,9) -- (-22,0) -- (-13,9); \draw (-22,3) circle (2);
\draw(-22,1.25) node[above]{$1$}; \draw (-25,6) circle (2);
\draw(-25,4.25) node[above]{$3$}; \draw (-19,6) circle (2);
\draw(-19,4.25) node[above]{$2$};

\draw (-9,9) -- (0,0) -- (9,9); \draw (0,3) circle (2);
\draw(0,1.25) node[above]{$1$}; \draw (3,6) circle (2);
\draw(3,4.25) node[above]{$2$}; \draw (6,9) circle (2);
\draw(6,7.25) node[above]{$3$};

\end{tikzpicture}
\]
\end{center}

\bigskip

The {\dem inversion number} of a permutation $\pi$ is defined by
\[
\inv(\pi):=\#\{i<j \,:\, \pi(i)>\pi(j)\}.
\]
The {\dem left weak order} on $\Sc_n$ is defined by
\[
\pi\le \sigma \Longleftrightarrow
\inv(\pi)+\inv(\sigma\pi^{-1})=\inv(\sigma).
\]

The following surprising result was first proved by
Stanley~\cite{Stanley-words}.

\begin{theorem}\label{AR_t:staircase_balls_thm}
The number of configurations of ${n\choose 2}$ balls
which completely fill $n-1$ levels in the bin is equal to the
number of maximal chains in the weak order on $\Sc_n$.
\end{theorem}


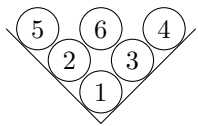
\begin{center}
\[
\begin{tikzpicture}[scale=0.14]

\draw (-13,9) -- (-4,0) -- (5,9); \draw (-4,3) circle (2);
\draw(-4,1.25) node[above]{$1$}; \draw (-7,6) circle (2);
\draw(-7,4.25) node[above]{$2$}; \draw (-1,6) circle (2);
\draw(-1,4.25) node[above]{$3$}; \draw (2,9) circle (2);
\draw(2,7.25) node[above]{$4$}; \draw (-10,9) circle (2);
\draw(-10,7.25) node[above]{$5$}; \draw (-4,9) circle (2);
\draw(-4,7.25) node[above]{$6$};

\end{tikzpicture}
\]
\end{center}



The configurations of balls in a bin are called {\dem standard
Young tableaux}. We shall survey in this chapter results related
to Question~\ref{AR_q:appet} and its refinements. Variants and
extensions of Theorem~\ref{AR_t:appet1} will be described in
Section~\ref{AR_s:JdT}. Variants and extensions of
Theorem~\ref{AR_t:staircase_balls_thm} will be described in
Section~\ref{AR_s:words}.



\subsection{General}

This chapter is devoted to the enumeration of standard Young
tableaux of various shapes, both classical and modern, and to
closely related topics. Of course, there is a limit as to how far
afield one can go. We chose to include here, for instance,
$r$-tableaux and $q$-enumeration, but many interesting related
topics were left out. Here are some of them, with a minimal list
of relevant references for the interested reader: Semi-standard
Young tableaux~\cite{Md}\cite{Stanley_EC2}, (reverse) plane
partitions~\cite{Stanley_EC2}, solid ($3$-dimensional) standard
Young
tableaux~\cite{Zeilberger_solid_SYT}, 
symplectic and orthogonal tableaux~\cite{King}\cite{DeConcini}%
\cite{Sundaram}\cite{Berele}\cite{Sundaram2}, 
oscillating tableaux~\cite{MacLarnan}\cite{Sagan90}%
\cite{Roby}\cite{DuluckSagan}\cite{PakPostnikov},
cylindric (and toric) tableaux~\cite{Postnikov}.




















\subsection{Acknowledgments}

Many people contributed comments and valuable information to this
chapter. We especially thank Christos Athanasiadis, Tomer Bauer,
Sergi Elizalde, Dominique Foata, Avital Frumkin, Curtis Greene,
Ira Gessel, Christian Krattenthaler, Igor Pak, Arun Ram, Amitai
Regev, Vic Reiner, Dan Romik, Bruce Sagan, Richard Stanley and
Doron Zeilberger.
We used Ryan Reich's 
package {\tt ytableau} 
for drawing diagrams and tableaux, and
the package {\tt algorithmicx} by J\'{a}nos Sz\'{a}sz for
typesetting algorithms in pseudocode.

\tableofcontents

\section{Preliminaries}\label{AR_s:prelim}

\subsection{Diagrams and tableaux}

\ytableausetup{centertableaux, boxsize = \standardboxsize}
\ytableausetup{nobaseline}

\begin{definition}
A {\dem diagram} 
is a finite subset $D$ of the two-dimensional integer lattice
$\bbz^2$.
A point $c = (i,j) \in D$ is also called the {\dem cell} in row
$i$ and column $j$ of $D$; write $\row(c) = i$ and $\col(c) = j$.
Cells are usually drawn as squares with axis-parallel sides of
length $1$, centered at the corresponding lattice points.
\end{definition}

Diagrams will be drawn here according to the ``English notation'',
by which $i$ enumerates rows and increases downwards, while $j$
enumerates columns and increases from left to right:
\[
\ytableausetup{boxsize = 2.5em}
\ytableaushort{{(1,1)}{(1,2)}{(1,3)},{(2,1)}{(2,2)}}
\ytableausetup{boxsize = \standardboxsize}
\]
For alternative conventions see
Subsection~\ref{AR_s:def_classical_shapes}.

\begin{definition}\label{AR_d:order_on_D}
Each diagram $D$ has a natural component-wise partial order,
inherited from $\bbz^2$:
\[
(i, j) \le_D (i', j') \iff i \le i' \hbox{\rm\ and } j \le j'.
\]
As usual, $c <_D c'$ means $c \le_D c'$ but $c \ne c'$.
\end{definition}


\begin{definition}\label{AR_d:SYT}
Let $n := |D|$, and consider the set $[n] := \{1, \ldots, n\}$
with its usual linear order.
A {\dem standard Young tableau (SYT) of shape $D$} is a map $T: D
\to [n]$ which is an order-preserving bijection, namely satisfies
\[
c \ne c' \then T(c) \ne T(c')
\]
as well as
\[
c \le_D c' \then T(c) \le T(c').
\]
\end{definition}
Geometrically, a standard Young tableau $T$ is a filling of the
$n$ cells of $D$ by the numbers $1, \ldots, n$ such that each
number appears once, and numbers increase in each row (as the
column index increases) and in each column (as the row index
increases). Write  $\sh(T) = D$. Examples will be given below.

Let $\SYT(D)$ be the set of all standard Young tableaux of shape
$D$, and denote its size by
\[
f^D := |\SYT(D)|.
\]

The evaluation of $f^D$ (and some of its refinements) for various
diagrams $D$ is the main concern of the current chapter.

\subsection{Connectedness and convexity}

We now introduce two distinct notions of connectedness for
diagrams, and one notion of convexity; for another notion of
convexity see Observation~\ref{AR_t:order_convex}.

\begin{definition}
Two distinct cells in $\bbz^2$ are {\dem adjacent} if they share a
horizontal or vertical side; the cells adjacent to $c = (i,j)$ are
$(i \pm 1, j)$ and $(i, j \pm 1)$.
A diagram $D$ is {\dem path-connected} if any two cells in it can
be connected by a {\dem path}, which is a finite sequence of cells
in $D$ such that any two consecutive cells are adjacent. The
maximal path-connected subsets of a nonempty diagram $D$ are its
{\dem path-connected components}.
\end{definition}

For example, the following diagram has two path-connected
components:
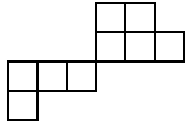
\[
\ydiagram{3 + 2, 3 + 3, 3, 1}
\]

\begin{definition}
The {\dem graph} of a diagram $D$ has all the cells of $D$ as
vertices, with two distinct cells $c, c' \in D$ connected by an
(undirected) edge if either $c <_D c'$ or $c' <_D c$. The diagram
$D$ is {\dem order-connected} if its graph is connected. In any
case, the {\dem order-connected components} of $D$ are the subsets
of $D$ forming connected components of its graph.
\end{definition}

For example, the following diagram (in English notation) is
order-connected:
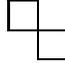
\[
\ydiagram{1, 1 + 1}
\]
while the following diagram has two order-connected components,
with cells marked $1$ and $2$, respectively:
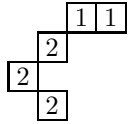
\[
\ytableaushort{\none\none11, \none2, 2, \none2}
\]

Of course, every path-connected diagram is also order-connected,
so that every order-connected component is a disjoint union of
path-connected components.

\begin{observation}\label{AR_t:obs1}
If $D_1, \ldots, D_k$ are the order-connected components of a 
diagram $D$, then
\[
f^D = {|D| \choose |D_1|, \ldots, |D_k|} \prod_{i=1}^{k} f^{D_i} =
|D|! \cdot \prod_{i=1}^{k} \frac{f^{D_i}}{|D_i|!}.
\]
\end{observation}

\begin{definition}
A diagram $D$ is {\dem line-convex} if its intersection with every
axis-parallel line is either empty or convex, namely if each of
its rows $\{j \in \bbz \,|\, (i, j) \in D\}$ (for $i \in \bbz$)
and columns $\{i \in \bbz \,|\, (i, j) \in D\}$ (for $j \in \bbz$)
is either empty or an interval $[p,q] = \{p, p+1, \ldots, q\}
\subseteq \bbz$.
\end{definition}

For example, the following diagram is path-connected but not
line-convex:
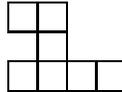
\[
\ydiagram{2, 1 + 1, 4}
\]


\subsection{Invariance under symmetry}

The number of SYT of shape $D$ is invariant under some of the
geometric operations (isometries of $\bbz^2$) which transform $D$.
It is clearly invariant under arbitrary translations $(i, j)
\mapsto (i + a, j + b)$. The group of isometries of $\bbz^2$ that
fix a point, say $(0,0)$, is the dihedral group of order $8$.
$f^D$ is invariant under a subgroup of order $4$.

\begin{observation}\label{AR_t:invariance}
$f^D$ is invariant under arbitrary translations of $\bbz^2$, as
well as under
\begin{itemize}
\item reflection in a diagonal line: $(i, j) \mapsto (j, i)$ or
$(i, j) \mapsto (-j, -i)$; and \item reflection in the origin
(rotation by $180^\circ$): $(i, j) \mapsto (-i, -j)$.
\end{itemize}
\end{observation}

Note that $f^D$ is not invariant, in general, under reflections in
a vertical  or horizontal line ($(i, j) \mapsto (i, -j)$ or $(i,
j) \mapsto (-i, j)$) or rotations by $90^\circ$ ($(i, j) \mapsto
(-j, i)$ or $(i, j) \mapsto (j, -i)$). Thus, for example, each of
the following diagrams, interpreted according to the English
convention (see Subsection~\ref{AR_s:def_classical_shapes}),
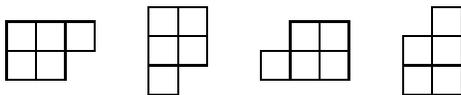
\[
\ydiagram{3, 2} \qquad \ydiagram{2, 2, 1} \qquad \ydiagram{1 + 2,
3} \qquad \ydiagram{1 + 1, 2, 2}
\]
has $f^D = 5$, whereas each of the following diagrams
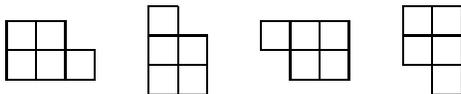
\[
\ydiagram{2, 3} \qquad \ydiagram{1, 2, 2} \qquad \ydiagram{3, 1 +
2} \qquad \ydiagram{2, 2, 1 + 1}
\]
has $f^D = 2$.

\subsection{Ordinary, skew and shifted shapes}\label{AR_s:def_classical_shapes}

The best known and most useful diagrams are, by far, the ordinary
ones. They correspond to partitions.

\begin{definition}
A {\dem partition} is a weakly decreasing sequence of positive
integers: $\la = (\la_1, \ldots, \la_t)$, where $t \ge 0$ and
$\la_1 \ge \ldots \ge \la_t >0$. We say that $\la$ is a partition
of {\dem size} $n = |\la| := \sum_{i=1}^t \la_i$ and {\dem length}
$\ell(\la) := t$, and write $\la \vdash n$. The empty partition
$\la = ()$ has size and length both equal to zero.
\end{definition}

\begin{definition}\label{AR_d:ordinary_diagram}
Let $\la = (\la_1, \ldots, \la_t)$ be a partition. The {\dem
ordinary} (or {\dem straight}, or {\dem left-justified}, or {\dem
Young}, or {\dem Ferrers}) {\dem diagram of shape $\la$} is the
set
\[
D = [\la] := \{(i,j) \,|\, 1\le i \le t,\, 1 \le j \le \la_i \}.
\]
We say that $[\la]$ is a diagram of {\dem height} $\ell(\la) = t$.
\end{definition}


We shall adopt here the ``English'' convention for drawing
diagrams, by which row indices increase from top to bottom and
column indices increase from left to right. For example, in this
notation the diagram of shape $\la = (4,3,1)$ is
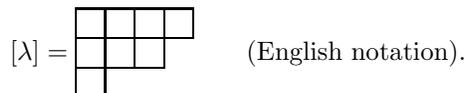
\[
[\la] =
\ydiagram{4,3,1} \qquad \hbox{(English notation).}
\]
An alternative convention is the ``French'' one, by which row
indices increase from bottom to top (and column indices increase
from left to right):
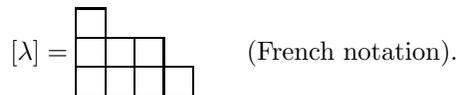
\[
[\la] =
\ydiagram{1,3,4} \qquad \hbox{(French notation).}
\]
Note that the term ``Young tableau'' itself mixes English and
French influences. There is also a ``Russian'' convention, rotated
$45^\circ$:

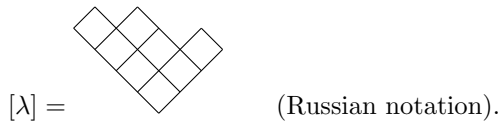
\[
[\la] =
\begin{tikzpicture}[scale=0.283]
\draw (0,-2) -- (-1,-1) -- (-2,0) -- (-3,1) -- (-4,2); \draw
(1,-1) -- (0,0) -- (-1,1) -- (-2,2) -- (-3,3); \draw (2,0) --
(1,1) -- (0,2) -- (-1,3); \draw (3,1) -- (2,2);

\draw (0,-2) -- (1,-1) -- (2,0) -- (3,1); \draw (-1,-1) -- (0,0)
-- (1,1) -- (2,2); \draw (-2,0) -- (-1,1) -- (0,2); \draw (-3,1)
-- (-2,2) -- (-1,3); \draw (-4,2) -- (-3,3);

\end{tikzpicture}
\qquad \hbox{(Russian notation).}
\]
This notation leads naturally to the ``gravitational'' setting
used to introduce SYT at the beginning of
Section~\ref{AR_s:introduction}.


A partition $\la$ may also be described as an infinite sequence,
by adding trailing zeros: $\la_i := 0$ for $i > t$. The partition
$\la'$ {\dem conjugate} to $\la$ is then defined by
\[
\la'_j := |\{i \,|\, \la_i \ge j\}| \qquad (\forall j \ge 1).
\]
The diagram $[\la']$ is obtained from the diagram $[\la]$ by
interchanging rows and columns. For the example above, $\la' =
(3,2,2,1)$ and
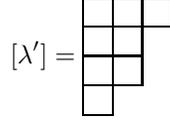
\[
[\la'] =
\ydiagram{3,2,2,1}
\]

An ordinary diagram is clearly path-connected and line-convex. If
$D = [\la]$ is an ordinary diagram of shape $\la$ we shall
sometimes write $\SYT(\la)$ instead of $\SYT(D)$ and $f^\la$
instead of $f^D$.
\begin{example}\label{AR_ex:ordinary_diagram}
\[
T =
\ytableaushort{1258, 346, 7} \,\in \SYT(4,3,1).
\]
\end{example}
Note that, by Observation~\ref{AR_t:invariance}, $f^\la =
f^{\la'}$.


\begin{definition}\label{AR_d:shifted_diagram}
If $\la$ and $\mu$ are partitions such that $[\mu] \subseteq
[\la]$, namely $\mu_i \le \la_i$ ($\forall i$), then the {\dem
skew diagram of shape $\la/\mu$} is the set difference
\[
D = [\la/\mu] := [\la] \setminus [\mu] = \{(i,j)\in [\la] \ :\,
\mu_i + 1 \le j \le \la_i\}
\]
of two ordinary shapes.
\end{definition}
For example,
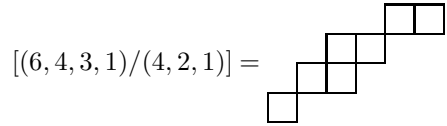
\[
[(6,4,3,1) / (4,2,1)] =
\ydiagram{4 + 2, 2 + 2, 1 + 2, 1}
\]

A skew diagram is line-convex, but not necessarily path-connected.
In fact, its path-connected components coincide with its
order-connected components. If $D = [\la/\mu]$ is a skew diagram
of shape $\la/\mu$ we shall sometimes write $\SYT(\la/\mu)$
instead of $\SYT(D)$ and $f^{\la/\mu}$ instead of $f^D$. For
example,
\[
T =
\ytableaushort{\none\none\none\none14, \none\none37, \none56, 2}
\,\in \SYT((6,4,3,1) / (4,2,1)).
\]
Skew diagrams have an intrinsic characterization.
\begin{observation}\label{AR_t:order_convex}
A diagram $D$ is skew if and only if it is {\dem order-convex},
namely:
\[
c,c'' \in D,\, c' \in \bbz^2,\, c \le c' \le c'' \then c' \in  D,
\]
where $\le$ is the natural partial order in $\bbz^2$, as in
Definition~\ref{AR_d:order_on_D}.
\end{observation}


Another important class is that of shifted shapes, corresponding
to strict partitions.

\begin{definition}
A partition $\la = (\la_1, \ldots, \la_t)$ $(t \ge 0)$ is {\dem
strict} if the part sizes $\la_i$ are strictly decreasing: $\la_1
> \ldots > \la_t >0$.
%
The {\dem shifted diagram of shape $\la$} is the set
\[
D = [\la^*] := \{(i,j) \,|\, 1\le i \le t,\, i \le j \le \la_i + i
-1\}.
\]
Note that $(\la_i + i - 1)_{i = 1}^{t}$ is a weakly decreasing
sequence of positive integers.
\end{definition}
For example, the shifted diagram of shape $\la = (4,3,1)$ is
\[
[\la^*] =
\ydiagram{4, 1 + 3, 2 + 1}
\]

A shifted diagram is always path-connected and line-convex. If $D
= [\la^*]$ is a shifted diagram of shape $\la$ we shall sometimes
write $\SYT(\la^*)$ instead of $\SYT(D)$ and $g^\la$ instead of
$f^D$. For example,
\[
T =
\ytableaushort{1246, \none358, \none\none7} \,\in \SYT((4,3,1)^*).
\]

\subsection{Interpretations}\label{AR_s:interpretations}

There are various interpretations of a standard Young tableau,
in addition to the interpretation (in Definition~\ref{AR_d:SYT})
as a linear extension of a partial order. Some of these
interpretations play a key role in enumeration.

\subsubsection{
The Young lattice}

A standard Young tableau of ordinary shape describes a {\dem
growth process} of diagrams of ordinary shapes, starting from the
empty shape. For example, the tableau $T$ in
Example~\ref{AR_ex:ordinary_diagram} corresponds to the process
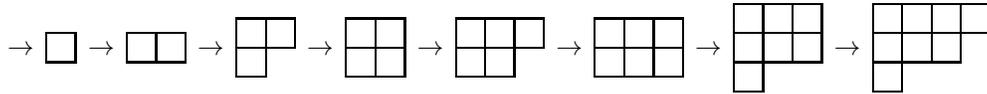
\[
\, \to \, \ydiagram{1} \, \to \, \ydiagram{2} \, \to \,
\ydiagram{2,1} \, \to \, \ydiagram{2,2} \, \to \, \ydiagram{3,2}
\, \to \, \ydiagram{3,3} \, \to \, \ydiagram{3,3,1} \, \to \,
\ydiagram{4,3,1}
\]

Consider the {\dem Young lattice} whose elements are all
partitions, ordered by inclusion (of the corresponding diagrams).
By the above, a SYT of ordinary shape $\la$ is a maximal chain, in
the Young lattice, from the empty partition to $\la$. The number
of such maximal chains is therefore $f^\la$. More generally, a SYT
of skew shape $\la/\mu$ is a maximal chain from $\mu$ to $\la$ in
the Young lattice.

A SYT of shifted shape can be similarly interpreted as a maximal
chain in the {\dem shifted Young lattice}, whose elements are
strict partitions ordered by inclusion.





\subsubsection{Ballot sequences and lattice paths}\label{AR_s:lattice_paths}

\begin{definition}\label{AR_d:ballot_seq}
A sequence $(a_1, \ldots, a_n)$ of positive integers is a {\dem
ballot sequence}, or {\dem lattice permutation}, if for any
integers $1 \le k \le n$ and $r \ge 1$,
\[
\# \{1 \le i \le k \,|\, a_i = r \} \ge \# \{1 \le i \le k \,|\,
a_i = r+1 \},
\]
namely: in any initial subsequence $(a_1, \ldots, a_k)$, the
number of entries equal to $r$ is not less than the number of
entries equal to $r+1$.
\end{definition}
A ballot sequence describes the sequence of votes in an election
process with several candidates (and one ballot), assuming that at
any given time candidate $1$ has at least as many votes as
candidate $2$, who has at least as many votes as candidate $3$,
etc. For example, $(1,1,2,3,2,1,4,2,3)$ is a ballot sequence for
an election process with $9$ voters and $4$ candidates.

For a partition $\la$ of $n$, denote by $\BS(\la)$ the set of
ballot sequences $(a_1, \ldots, a_n)$ with $\# \{i \,|\, a_i = r\}
= \la_r$ $(\forall r)$.

\begin{observation}
The map $\phi : \SYT(\la) \to \BS(\la)$ defined by
\[
\phi(T)_i := \row(T^{-1}(i)) \qquad (1 \le i \le n)
\]
is a bijection.
\end{observation}

For example, if $T$ is the SYT in Example~\ref{AR_ex:ordinary_diagram} 
then $\phi(T) = (1,1,2,2,1,2,3,1)$.

\medskip

Clearly, a ballot sequence $a = (a_1, \ldots, a_n) \in \BS(\la)$
with maximal entry $t$ corresponds to a {\dem lattice path} in
$\bbr^t$, from the origin $0$ to the point $\la$, where in step
$i$ of the path coordinate $a_i$ of the point increases by $1$. In
fact, $\BS(\la)$ is in bijection with the set of all lattice paths
from $0$ to $\la$ which lie entirely in the cone
\[
\{(x_1, \ldots, x_t) \in \bbr^t \,|\, x_1 \ge \ldots \ge x_t \ge 0
\}.
\]

A SYT of skew shape $\la/\mu$ corresponds to a lattice path in
this cone from $\mu$ to $\la$.

A SYT of shifted shape corresponds to
a {\dem strict} ballot sequence, describing a lattice path within
the cone
\[
\{(x_1, \ldots, x_t) \in \bbr^t \,|\, x_1 > \ldots > x_s > x_{s+1}
= \ldots = x_t = 0 \hbox{ for some } 0 \le s \le t \}.
\]

\subsubsection{The order polytope}\label{AR_s:order_polytope}

Using the partial order on a diagram $D$, as in
Definition~\ref{AR_d:order_on_D}, one can define the corresponding
{\dem order polytope}
\[
P(D) := \{ f: D \to [0,1] \,|\, c \le_D c' \then f(c) \le f(c')\,
(\forall c,\, c' \in D)\}.
\]
The order polytope is a closed convex subset of the unit cube
$[0,1]^D$. Denoting $n := |D|$, each linear extension $c_1 <
\ldots < c_n$ of the partial order $\le_D$ corresponds to a
simplex
\[
\{ f: D \to [0,1] \,|\, f(c_1) \le \ldots \le f(c_n) \}
\]
of volume $1/n!$ inside $P(D)$. The union of these simplices is
$P(D)$, and this union is almost disjoint: Any intersection of two
or more simplices is contained in a hyperplane, and thus has
volume $0$. Each simplex, or linear extension, corresponds to a
SYT of shape $D$. Therefore
\begin{observation}\label{AR_t:vol_order_polytope}
If $P(D)$ is the order polytope of a diagram $D$ then
\[
\vol P(D) = \frac{f^D}{|D|!}.
\]
\end{observation}

\subsubsection{Other interpretations}

The number of SYT of ordinary shape can be interpreted as a
coefficient in a power series, or as the constant term in a
Laurent series; see Remark~\ref{AR_r:DZ_coeff}.


SYT of certain ordinary, skew and shifted shapes may be
interpreted as {\dem reduced words} for suitable elements in the
symmetric group. This interpretation will be developed and
explained in Section~\ref{AR_s:words}.


SYT may be interpreted as {\dem permutations}; $f^\la$ then
measures the size of certain subsets of the symmetric group, such
as descent classes, sets of involutions, Knuth classes and pattern
avoidance
classes. Examples will be given in Sections~\ref{AR_s:JdT} and%
~\ref{AR_s:q}.

The Young lattice has a vast generalization to the concept of {\dem differential
poset}~\cite{Stanley_differential}\cite{Gessel}.

There are other deep algebraic and geometric interpretations. The
interested reader is encouraged to start with the beautiful
survey~\cite{Barcelo-Ram} and the excellent textbooks~\cite{JK,
Stanley_EC2, Sagan_book, Fulton, Md}. In the current survey we
will focus on combinatorial aspects and mostly ignore the
algebraic and geometric approaches.

\subsection{Miscellanea}

The concepts to be defined here are not directly related to
standard Young tableaux, but will be used later in this survey.



A {\dem composition} of a nonnegative integer $n$ is a sequence
$(\mu_1, \ldots, \mu_t)$ of positive integers such that $\mu_1 +
\ldots + \mu_t = n$. The components $\mu_i$ are {\em not} required
to be weakly decreasing; in fact, every composition may be
re-ordered to form a (unique) partition. $n = 0$ has a unique
(empty) composition.

A permutation $\sigma \in \Sc_n$ {\dem avoids a pattern} $\pi \in
\Sc_k$ if the sequence $(\sigma(1), \ldots, \sigma(n))$ does not
contain a subsequence $(\sigma(t_1), \ldots, \sigma(t_k))$ (with
$t_1 < \ldots < t_k$) which is order-isomorphic to $\pi$, namely:
$\sigma(t_i) < \sigma(t_j) \iff \pi(i) < \pi(j)$. For example,
$21354 \in \Sc_5$ is $312$-avoiding, but $52134$ is not (since
$523$ is order-isomorphic to $312$).


\section{Formulas for thin shapes}\label{AR_s:thin}

\subsection{Hook shapes}

A {\dem hook shape} is an ordinary shape which is the union of one
row and one column. For example,
\[
[(6,1^3)] =
\ydiagram{6, 1, 1, 1}
\]
One of the simplest  enumerative formulas is the following.

\begin{observation}\label{hooks} 
For every $n \ge 1$ and  $0 \le k \le n-1$,
\[
f^{(n-k, 1^k)} = {n-1 \choose k}.
\]
\end{observation}

\begin{proof}
The letter $1$ must be in the corner cell. The SYT is uniquely
determined by the choice of the other $k$ letters in the first
column.
\end{proof}



Note that, in a hook shape $(n+1-k,1^k)$, the letter $n+1$ must be
in the last cell of either the first row or the first column. Thus
\[
f^{(n+1-k,1^k)}=f^{(n-k,1^k)}+f^{(n+1-k,1^{k-1})}.
\]
By Observation~\ref{hooks}, this is equivalent to Pascal's
identity
\[
{n \choose k}={n-1 \choose k}+{n-1 \choose k-1} \qquad (1 \le k
\le n-1).
\]

\begin{observation}
The total number of hook shaped SYT of size $n$ is $2^{n-1}$.
\end{observation}

\begin{proof}
There is a bijection between hook shaped SYT of size $n$ and
subsets of $\{1, \ldots, n\}$ containing $1$: Assign to each SYT
the set of entries in its first row.

Alternatively, a hook shaped SYT of size $n \ge 2$ is uniquely
constructed by adding a cell containing $n$ at the end of either
the first row or the first column of a hook shaped SYT of size
$n-1$, thus recursively multiplying the number of SYT by $2$.

Of course, the claim also follows from Observation~\ref{hooks}.
\end{proof}

\subsection{Two-rowed shapes}

Consider now ordinary shapes with at most two rows.

\begin{proposition}\label{two-rows} 
For every $n \ge 0$ and $0\le k\le n/2$
\[
f^{(n-k,k)}={n \choose k}-{n \choose k-1},
\]
where ${n \choose -1} = 0$ by convention. In particular,
\[
f^{(m,m)} = f^{(m,m-1)} = C_m = \frac{1}{m+1}{2m \choose m},
\]
the $m$-th Catalan number.
\end{proposition}

\begin{proof}
We shall outline two proofs, one by induction and one
combinatorial.

For a proof by induction on $n$ note first that $f^{(0,0)} =
f^{(1,0)} =1$.


If $0 < k< n/2$ then there are two options for the location of the
letter $n$ -- at the end of the first row or at the end of the
second. Hence
\[
f^{(n-k,k)} = f^{(n-k-1,k)} + f^{(n-k,k-1)} \qquad (0 < k < n/2).
\]
Thus, by the induction hypothesis and Pascal's identity,
\[
f^{(n-k,k)} = {n-1 \choose k} - {n-1 \choose k-1} + {n-1 \choose
k-1} - {n-1 \choose k-2} = {n-1 \choose k} - {n-1 \choose k-2} =
{n \choose k} - {n \choose k-1}.
\]
The cases $k = 0$ and $k = n/2$ are left to the reader.

For a combinatorial proof, recall (from
Subsection~\ref{AR_s:lattice_paths}) the lattice path
interpretation of a SYT and use Andr\'{e}'s reflection trick:
A SYT of shape $(n-k,k)$ corresponds to a lattice path from
$(0,0)$ to $(n-k,k)$ which stays within the cone $\{(x,y) \in
\bbr^2 \,|\, x \ge y \ge 0\}$, namely does not touch the line $y =
x+1$. The number of {\em all} lattice paths from $(0,0)$ to
$(n-k,k)$ is ${n \choose k}$. If such a path touches the line $y =
x + 1$, reflect its ``tail'' (starting from the first touch point)
in this line to get a path from $(0,0)$ to the reflected endpoint
$(k-1, n-k+1)$. The reflection defines a bijection between all the
``touching'' paths to $(n-k,k)$ and all the (necessarily
``touching'') paths to $(k-1,n-k+1)$, whose number is clearly ${n
\choose k-1}$.
\end{proof}

\begin{corollary}\label{total-two-rows}
The total number of SYT of size $n$ and at most $2$ rows 
is $n \choose {\lfloor n/2 \rfloor}$.
\end{corollary}

\begin{proof} By Proposition~\ref{two-rows},
\[
\sum_{k=0}^{\lfloor n/2 \rfloor} f^{(n-k,k)} = \sum_{k=0}^{\lfloor
n/2 \rfloor} \left( {n \choose k} - {n \choose k-1} \right) = {n
\choose {\lfloor n/2 \rfloor}}.
\]
\end{proof}

\subsection{Zigzag shapes}\label{AR_s:zigzag_sum}

A {\dem zigzag shape}
is a path-connected skew shape which does not contain a $2\times
2$ square. For example, every hook shape is zigzag. Here is an
example of a zigzag shape of size $11$:
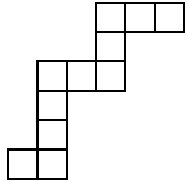
\[
\ydiagram{3 + 3, 3 + 1, 1 + 3, 1 + 1, 1 + 1, 2}
\]

The number of SYT of a specific zigzag shape has an interesting
formula, to be presented in Subsection~\ref{AR_s:zigzag}. The
total number of SYT of given size and various zigzag shapes is
given by the following folklore statement, to be refined later
(Proposition~\ref{t.zigzag-descent class}).

\begin{proposition}\label{total-zigzag}
The total number of zigzag shaped SYT of size $n$ is $n!$.
\end{proposition}

\begin{proof}
Define a map from zigzag shaped SYT of size $n$ to permutations in
$\Sc_n$ by simply listing the entries of the SYT, starting from
the SW corner and moving along the shape. This map is a bijection,
since an obvious inverse map builds a SYT from a permutation
$\sigma = (\sigma_1, \ldots, \sigma_n) \in \Sc_n$ by attaching a
cell containing $\sigma_{i+1}$ to the right of the cell containing
$\sigma_i$ if $\sigma_{i+1} > \sigma_i$, and above this cell
otherwise.
\end{proof}


\section{Jeu de taquin and the RS correspondence}\label{AR_s:JdT}

\subsection{Jeu  de taquin} 

Jeu de taquin is a very powerful combinatorial algorithm,
introduced by Sch\"utzenberger~\cite{Schutzenberger}. It provides
a unified approach to many enumerative results. In general, it
transforms a SYT of skew shape into some other SYT of skew shape,
using a sequence of {\dem slides}. We shall describe here a
version of it, using only {\dem forward slides}, which transforms
a SYT of skew shape into a (unique) SYT of ordinary shape. Our
description follows \cite{Sagan_book}.

\begin{definition}
Let $D$ be a nonempty diagram of skew shape. An {\bf inner corner}
for $D$ is a cell $c \not\in D $ such that
\begin{enumerate}
\item $D \cup \{c\}$ is a skew shape, and \item there exists a
cell $c' \in D$ such that $c \le c'$ (in the natural partial order
of $\bbz^2$, as in Definition~\ref{AR_d:order_on_D}).
\end{enumerate}
\end{definition}

\begin{example}
Here is a skew shape with marked inner corners:
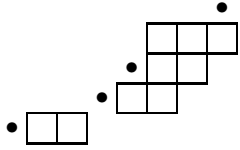
\[
\ytableaushort{\none\none\none\none\none\none\none{\none[\bullet]},
\none, \none\none\none\none{\none[\bullet]},
\none\none\none{\none[\bullet]}, {\none[\bullet]}}
* {0, 5+3, 5+2, 4+2, 1+2}
\]
\end{example}


Here is the main jeu de taquin procedure:

\medskip

\begin{algorithmic}[1]
\Require{$T$, a SYT of arbitrary skew shape.} \Ensure{$T'$, a SYT
of ordinary shape.} \Statex \Procedure{JdT}{$T$} \State $D \gets
\sh(T)$ \State Choose a cell $c_0 = (i_0, j_0)$ such that $D
\subseteq (c_0)_+ := \{(i, j) \in \bbz^2 : i \ge i_0,\, j \ge
j_0\}$ \While{$c_0 \not\in D$}
    \State Choose $c = (i, j) \in (c_0)_+ \setminus D$ which is an inner corner for $D$
    \State $T \gets$ \Call{ForwardSlide}{$T, c$}
    \State $D \gets \sh(T)$
\EndWhile \State \textbf{return} $T$ \Comment Now $c_0 \in D
\subseteq (c_0)_+$, so $D$ has ordinary shape \EndProcedure
\end{algorithmic}

\medskip

and here is the procedure $\textsc{ForwardSlide}$:

\medskip

\begin{algorithmic}[1]
\Require{$(T_{in}, c_{in})$, where $T_{in}$ is a SYT of skew shape
$D_{in}$ and $c_{in}$ is an inner corner for $D_{in}$.}
\Ensure{$T_{out}$, a SYT of skew shape $D_{out} = D_{in} \cup
\{c_{in}\} \setminus \{c'\}$ for some $c' \in D_{in}$.} \Statex
\Procedure{ForwardSlide}{$T_{in}, c_{in}$} \State $T \gets
T_{in}$, $c \gets c_{in}$ \State $D \gets \sh(T)$ \If{$c = (i,
j)$}
    \State $c_1 \gets (i+1, j)$
    \State $c_2 \gets (i, j+1)$
\EndIf \While{at least one of $c_1$ and $c_2$ is in $D$}
    \State $c' \gets \begin{cases}
    c_1, & \text{if $c_1 \in D$ but $c_2 \not\in D$, or $c_1, c_2 \in D$ and $T(c_1) < T(c_2)$} \\
    c_2, & \text{if $c_2 \in D$ but $c_1 \not\in D$, or $c_1, c_2 \in D$ and $T(c_2) < T(c_1)$}
    \end{cases}$
    \State $D' \gets D \cup \{c\} \setminus \{c'\}$ \Comment{$c \not\in D$, $c' \in D$}
    \State Define $T' \in \SYT(D')$ by: $T' = T$ on $D \setminus \{c'\}$ and $T'(c) := T(c')$
    \State $D \gets D'$, $T \gets T'$, $c \gets c'$
    \If{$c = (i, j)$}
        \State $c_1 \gets (i+1, j)$
        \State $c_2 \gets (i, j+1)$
    \EndIf
\EndWhile \State \textbf{return} $T$ \EndProcedure
\end{algorithmic}

\medskip

The $\JdT$ algorithm employs certain random choices, but actually

\begin{proposition}{\rm \cite{Schutzenberger, Thomas1974, Thomas1977}}
For any SYT $T$ of skew shape, the resulting SYT $\JdT(T)$ of
ordinary shape is independent of the choices made during the
computation.
\end{proposition}


\begin{example}
Here is an example of a forward slide, with the initial $c_{in}$
and the intermediate cells $c$ marked:
\[
T_{in} = \, \ytableaushort{\none {\none[\bullet]} 3 6, \none 1 4
7, 2 5 8} \,\to\, \ytableaushort{\none 1 3 6, \none
{\none[\bullet]} 4 7, 2 5 8} \,\to\, \ytableaushort{\none 1 3 6,
\none 4 {\none[\bullet]} 7, 2 5 8} \,\to\, \ytableaushort{\none 1
3 6, \none 4 7 {\none[\bullet]}, 2 5 8} \, = T_{out} \, ,
\]
and here is an example of a full jeu de taquin (where each step is
a forward slide):
\[
T = \, \ytableaushort{\none {\none[\bullet]} 3 6, \none 1 4 7, 2 5
8} \,\to\, \ytableaushort{\none 1 3 6, {\none[\bullet]} 4 7, 2 5
8} \,\to\, \ytableaushort{{\none[\bullet]} 1 3 6, 2 4 7, 5 8}
\,\to\, \ytableaushort{1 3 6, 2 4 7, 5 8} \,= \JdT(T).
\]
\end{example}

\subsection{The Robinson-Schensted  correspondence}\label{AR_s:RSK}

The Robinson-Schensted (RS) correspondence  is a bijection from
permutations in $\Sc_n$ to pairs of SYT of size $n$ and same
ordinary shape. Its original motivation was the study of the
distribution of longest increasing subsequences in a permutation.
For a detailed description see, e.g., the textbooks
~\cite{Stanley_EC2}, \cite{Sagan_book} and \cite{Fulton}. We shall
use the jeu de taquin algorithm to give an alternative
description.


\begin{definition}
Denote $\delta_n := [(n,n-1,n-2,\dots,1)]$. For a permutation
$\pi\in \Sc_n$ let $T_\pi$ the skew SYT of antidiagonal shape
$\delta_n/\delta_{n-1}$ in which the entry in the $i$-th column
from the left is $\pi(i)$.
\end{definition}

\begin{example}
\[
\pi = 53412 \quad \then \quad T_\pi =
\ytableaushort{\none\none\none\none 2, \none\none\none 1,
\none\none 4, \none 3, 5}.
\]
\end{example}

\begin{definition} {\rm (The Robinson-Schensted (RS) correspondence)}
For a permutation $\pi\in \Sc_n$ let
\[
P_\pi := \JdT(T_\pi) \quad \text{and} \quad Q_\pi :=
\JdT(T_{\pi^{-1}}).
\]
\end{definition}

\begin{example}
\[
\pi = 2413 
\quad \then \quad T_\pi =
\ytableaushort{\none\none\none 3, \none\none 1, \none 4, 2} \quad
, \quad T_{\pi^{-1}} =
\ytableaushort{\none\none\none 2, \none\none 4, \none 1, 3}.
\]
Then
\[
T_\pi = \,
\ytableaushort{\none\none\none 3, \none{\none[\bullet]} 1, \none
4, 2} \,\to\, \ytableaushort{\none\none{\none[\bullet]} 3, \none
1, \none 4, 2} \,\to\, \ytableaushort{\none\none 3,\none 1,
{\none[\bullet]} 4, 2} \,\to\, \ytableaushort{\none\none 3,
{\none[\bullet]} 1, 24} \,\to\,
\ytableaushort{\none{\none[\bullet]} 3, 14, 2} \,\to\,
\ytableaushort{{\none[\bullet]} 3, 14, 2} \,\to\,
\ytableaushort{13, 24} \, = P_\pi
\]
and
\[
T_{\pi^{-1}} = \,
\ytableaushort{\none\none{\none[\bullet]} 2, \none\none 4, \none
1, 3} \,\to\, \ytableaushort{\none\none 2, \none\none 4,
{\none[\bullet]} 1, 3} \,\to\, \ytableaushort{\none\none 2,
\none{\none[\bullet]} 4, 1, 3} \,\to\, \ytableaushort{\none\none
2, {\none[\bullet]} 4, 1, 3} \,\to\,
\ytableaushort{\none{\none[\bullet]} 2, 14, 3} \,\to\,
\ytableaushort{{\none[\bullet]} 2, 14, 3} \,\to\,
\ytableaushort{12, 34} \, = Q_\pi.
\]
\end{example}

\begin{theorem}\label{AR_t:RSK_bijection}
The RS correspondence is a bijection from all permutations in
$\Sc_n$ to all pairs of SYT of size $n$ and the same shape.
\end{theorem}



Thus

\begin{claim}\label{AR_t:RSK_properties}
For every permutation $\pi\in \Sc_n$,
\begin{itemize}

\item[(i)] $\sh(P_\pi)=\sh(Q_\pi)$.

\item[(ii)] $\pi \leftrightarrow (P, Q) \,\then\, \pi^{-1}
\leftrightarrow (Q, P)$.


\end{itemize}
\end{claim}

A very fundamental property of the RS correspondence is the
following.

\begin{proposition}\label{AR_t:Schensted}{\rm \cite{Schensted}}
The height of $\sh(\pi)$ is equal to the size of the longest
decreasing subsequence in $\pi$. The width of $\sh(\pi)$ is equal
to the size of the longest increasing subsequence in $\pi$.
\end{proposition}

A version of the RS correspondence for shifted shapes was given,
initially, by Sagan~\cite{Sagan79}. An improved algorithm was
found, independently, by Worley~\cite{Worley} and
Sagan~\cite{Sagan87}. See also~\cite{Haiman1989}.

\subsection{Enumerative applications}\label{AR_s:RSK-involutions}

In this section we list just a few applications of the above
combinatorial algorithms.

\begin{corollary}\label{AR_t:RSK_cor1}\
\begin{itemize}
\item[(1)] The total number of pairs of SYT of the same shape is
$n!$. Thus
\[
\sum_\la (f^\la)^2 = n!
\]

\item[(2)] The total number of SYT of size $n$ is equal to the
number of involutions in $\Sc_n$~{\rm \cite[A000085]{oeis}}. Thus
\[
\sum\limits_{\la\vdash n}f^\la=\sum\limits_{k=0}^{\lfloor
n/2\rfloor} {n\choose 2k} (2k-1)!!,
\]
where $(2k-1)!! := 1 \cdot 3 \cdot \ldots \cdot (2k-1)$.

\item[(3)] Furthermore, for every positive integer $k$, the total
number of SYT of height $< k$ is equal to the number of
$[k,k-1,\ldots,1]$-avoiding involutions in $\Sc_n$.
\end{itemize}
\end{corollary}

\begin{proof}
$(1)$ follows from Claim~\ref{AR_t:RSK_properties}(i), $(2)$ from
Claim~\ref{AR_t:RSK_properties}(ii), and $(3)$ from
Proposition~\ref{AR_t:Schensted}.
\end{proof}

\medskip


A careful examination of the RS correspondence implies the
following refinement of Corollary~\ref{AR_t:RSK_cor1}(2).

\begin{theorem}\label{AR_t:involutions_fixed}
The total number of SYT of size $n$ with $n - 2k$ odd rows is
equal to ${n\choose 2k} (2k-1)!!$, the number of involutions in
$\Sc_n$ with $n - 2k$ fixed points.
\end{theorem}

\begin{corollary}\label{even_parts}
The total number of SYT of size $2n$ and all rows even is equal to
$(2n-1)!!$, the number of fixed point free involutions in
$\Sc_{2n}$.
\end{corollary}




For further refinements see, e.g.,~\cite[Ex.\ 45--46, 85]{Stanley_EC2_Supp}.

\medskip

%






Recalling the simple formula for the number of two-rowed SYT
(Corollary~\ref{total-two-rows}), it is tempting to look for the
total number of SYT of shapes with more rows.


\begin{theorem}\label{height3}{\rm \cite{Regev-height}}
The total number of SYT of size $n$ and at most $3$ rows is
\[
M_n = \sum_{k=0}^{\lfloor n/2 \rfloor} {n\choose 2k} C_k \, ,
\]
the $n$-th Motzkin number~{\rm \cite[A001006]{oeis}}.
\end{theorem}


\begin{proof}
By Observation~\ref{AR_t:obs1} together with
Proposition~\ref{two-rows}, the number of SYT of skew shape
$(n-k,k,k) / (k)$ is equal to
\[
{n\choose 2k} C_k \, ,
\]
where $C_k$ is $k$-th Catalan number. On the other hand, by
careful examination of the jeu de taquin algorithm, one can verify
that it induces a bijection from the set of all SYT of skew shape
$(n-k,k,k) / (k)$ to the set of all SYT of shapes $(n-k-j,k,j)$
for $0 \le j \le \min(k, n-2k)$. Thus
\[
\sum\limits_{\la \vdash n \atop \ell(\la)\le 3} f^\la =
\sum\limits_{k=0}^{\lfloor n/2 \rfloor}\sum\limits_{j}
f^{(n-k-j,k,j)} = \sum_{k=0}^{\lfloor n/2 \rfloor} {n\choose 2k}
C_k \, ,
\]
completing the proof.
\end{proof}

See~\cite{Eu} for a bijective proof of Theorem~\ref{height3} via a
map from SYT of height at most $3$ to Motzkin paths.

The $n$-th Motzkin number also counts non-crossing involutions in
$\Sc_n$. It follows that

\begin{corollary}
The total number of SYT of height at most $3$ is equal to the
number of non-crossing involutions in $\Sc_n$.
\end{corollary}

\medskip

Somewhat more complicated formulas have been found for shapes with
more rows.

\begin{theorem}\label{height4}{\rm \cite{GB}}
\begin{itemize}
%
%
\item[1.]
The total number of SYT of size $n$ and at most $4$ rows 
is equal to $C_{\lfloor (n+1)/2\rfloor} C_{\lceil (n+1)/2\rceil}$.
\item[2.]
The total number of SYT of size $n$ and at most $5$ rows 
is equal to $6\sum_{k=0}^{\lfloor n/2 \rfloor}{n\choose 2k} C_k
\frac{(2k+2)!}{(k+2)!(k+3)!}$.
\end{itemize}
\end{theorem}

%
%
%
%

%
%

\bigskip

The following shifted analogue of Corollary~\ref{AR_t:RSK_cor1}(1)
was proved by Schur~\cite{Schur}, more than a hundred years ago,
in a representation theoretical setting. A combinatorial proof,
using the shifted RS correspondence, was given by
Sagan~\cite{Sagan79}. An improved shifted RS algorithm was found,
independently, by Worley~\cite{Worley} and Sagan~\cite{Sagan87}.
See the end of Subsection~\ref{AR_s:def_classical_shapes}
for the notation $g^\la$.

\begin{theorem}
\[
\sum_{\text{strict } \la} 2^{n - \ell(\la)} (g^\la)^2 = n!
\]
\end{theorem}

%
%
%


\section{Formulas for classical shapes}\label{AR_s:classical_shapes}

There is an explicit formula for the number of SYT of each
classical shape -- ordinary, skew or shifted. In fact, there are
several equivalent fomulas, all unusually elegant. These formulas,
with proofs, will be given in this section. Additional proof
approaches (mostly for ordinary shapes) will be described in
Section~\ref{AR_s:proof_approaches}.

\subsection{Ordinary shapes}\label{AR_s:ordinary_shapes}

In this subsection we consider ordinary shapes $D = [\la]$,
corresponding to partitions $\la$. Recall the notation $f^{\la} :=
|\SYT(\la)|$ for the number of standard Young tableaux of shape
$\la$. Several explicit formulas are known for this number -- a
product formula, a hook length formula and a determinantal
formula.


Historically, ordinary tableaux were introduced by Young in 1900~\cite{Young1900}.
The first explicit formula for the number of SYT of ordinary shape was 
the product formula. 
It was obtained in 1900 by Frobenius~\cite[eqn.\ 6]{Frobenius1900} in an algebraic context,
as the degree of an irreducible character $\chi^{\la}$ of $\Sc_n$. 
Independently, MacMahon~\cite[p.\ 175]{MacMahon1909} in 1909 
(see also~\cite[\S 103]{MacMahon_book}) obtained the same formula 
for the number of ballot sequences (see Definition~\ref{AR_d:ballot_seq} above),
which are equinumerous with SYT. In 1927 Young~\cite[pp.\ 260--261]{Young1927} 
showed that $\deg(\chi^{\la})$ is actually equal to the number of SYT of shape $\la$, 
and also provided his own proof~\cite[Theorem II]{Young1927} of MacMahon's result.

\begin{theorem}\label{t.AR_num_ordinary_prod}%
{\rm\ (Ordinary product formula)}
For a partition $\la = (\la_1, \ldots, \la_t)$, let $\ell_i :=
\la_i + t -i$ $(1 \le i \le t)$. Then
\[
f^{\la} = \frac{|\la|!}{\prod_{i=1}^{t} \ell_i!} 
\cdot \prod_{(i,j):\, i < j} (\ell_i - \ell_j). 
\]
\end{theorem}


The best known and most influential of the explicit formulas is
doubtlessly the Frame-Robinson-Thrall hook length formula,
published in 1954~\cite{FRT}. The story of its discovery is quite
amazing~\cite{Sagan_book}: Frame was led to conjecture the formula
while discussing the work of Staal, one of Robinson's students,
during Robinson's visit to him in May 1953. Robinson could not
believe, at first, that such a simple formula exists, but became
convinced after trying some examples, and together they proved it.
A few days later, Robinson gave a lecture followed by a
presentation of the new result by Frame. Thrall, who was in the
audience, was very surprised because he had just proved the same
result on the same day!

\begin{definition}\label{AR_d:hook}
For a cell $c = (i,j) \in [\la]$ let
\[
H_{c} := [\la] \cap \left( \{(i,j)\} \cup \{(i,j') \,|\, j' > j\}
\cup \{(i',j) \,|\, i' > i\} \right)
\]
be the corresponding {\dem hook}, and let
\[
h_{c} := |H_{c}| = \la_i + \la'_j - i - j + 1.
\]
be the corresponding {\dem hook length}.
\end{definition}
For example, in the following diagram the cells of the hook
$H_{(1,2)}$ are marked:
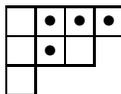
\[
\ydiagram{4, 3, 1} * [\bullet]{1 + 3, 1 + 1}
\]
and in the following diagram each cell is labeled by the
corresponding hook length:
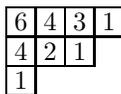
\[
\ytableaushort{6431, 421, 1}
\]

\begin{theorem}\label{t.AR_num_ordinary_hook}%
{\rm\ (Ordinary hook length formula)}
For any partition $\la = (\la_1, \ldots, \la_t)$,
\[
f^{\la} = \frac{|\la|!}{\prod_{c \in [\la]} h_{c}}.
\]
\end{theorem}


Last, but not least, is the determinantal formula. Remarkably, it
also has a generalization to the skew case; see the next
subsection.


\begin{theorem}\label{t.AR_num_ordinary_det}%
{\rm\ (Ordinary determinantal formula)}
For any partition $\la = (\la_1, \ldots, \la_t)$,
\[
f^{\la} = |\la|! \cdot \det \left[\frac{1}{(\la_i - i + j)!}
\right]_{i, j = 1}^{t},
\]
using the convention $1/k! := 0$ for negative integers $k$.
\end{theorem}

We shall now show that all these formulas are equivalent. Their
validity will then follow from a forthcoming proof of
Theorem~\ref{t.AR_num_skew_det}, which is a generalization of
Theorem~\ref{t.AR_num_ordinary_det}. Other proof approaches will
be described in Section~\ref{AR_s:proof_approaches}.

\begin{claim}
The formulas in Theorems~\ref{t.AR_num_ordinary_prod},
\ref{t.AR_num_ordinary_hook} and \ref{t.AR_num_ordinary_det} are
equivalent.
\end{claim}
\begin{proof}
To prove the equivalence of the product formula
(Theorem~\ref{t.AR_num_ordinary_prod}) and the hook length formula
(Theorem~\ref{t.AR_num_ordinary_hook}), it suffices to show that
\[
\prod_{c \in [\la]} h_{c} = \frac{\prod_{i=1}^{t} (\la_i + t -
i)!}{\prod_{(i,j):\, i < j} (\la_i - \la_j - i + j)}.
\]
This follows by induction on the number of columns, once we show
that the product of hook lengths for all the cells in the first
column of $[\la]$ satisfies
\[
\prod_{i=1}^{t} h_{(i,1)} =
\prod_{i=1}^{t} (\la_i + t - i)  \,;
\]
and this readily follows from the obvious
\[
h_{(i,1)} = \la_i + t - i \qquad (\forall i).
\]
Actually, one also needs to show that the ordinary product formula
is valid even when the partition $\la$ has trailing zeros (so that
$t$ in the formula may be larger than the number of nonzero parts
in $\la$). This is not difficult, since adding one zero part
$\la_{t+1} = 0$ (and replacing $t$ by $t+1$) amounts, in the
product formula, to replacing each $\ell_i = \la_i + t - i$ by
$\ell_i + 1$ $(1 \le i \le t)$ and adding $\ell_{t+1} = 0$, which
multiplies the RHS of the formula by
\[
\frac{1}{\prod_{i=1}^{t} (\ell_i + 1) \cdot \ell_{t+1}!} \cdot
\prod_{i=1}^{t} (\ell_i + 1 - \ell_{t+1}) = 1.
\]

To prove equivalence of the product formula
(Theorem~\ref{t.AR_num_ordinary_prod}) and the determinantal
formula (Theorem~\ref{t.AR_num_ordinary_det}), it suffices to show
that
\[
\det \left[\frac{1}{(\ell_i - t + j)!} \right]_{i, j = 1}^{t} =
\frac{1}{\prod_{i=1}^{t} \ell_i!} \cdot {\prod_{(i,j):\, i < j}
(\ell_i - \ell_j)} \quad ,
\]
where
\[
\ell_i := \la_i + t - i \qquad (1 \le i \le t)
\]
as in Theorem~\ref{t.AR_num_ordinary_prod}. Using the {\dem
falling factorial} notation
\[
(a)_n := \prod_{i = 1}^{n} (a + 1 - i) \qquad (n \ge 0),
\]
this claim is equivalent to
\[
\det \left[ (\ell_i)_{t - j} \right]_{i, j = 1}^{t} =
\prod_{(i,j):\, i < j} (\ell_i - \ell_j)
\]
which, in turn, is equivalent (under suitable column operations)
to the well known evaluation of the Vandermonde determinant
\[
\det \left[ \ell_i^{t - j} \right]_{i, j = 1}^{t} =
\prod_{(i,j):\, i < j} (\ell_i - \ell_j).
\]
See~\cite[pp.\ 132--133]{Sagan_book} for an inductive proof
avoiding explicit use of the Vandermonde.

\end{proof}

\subsection{Skew shapes}

The determinantal formula for the number of SYT of an ordinary
shape can be extended to apply to a general skew shape. The
formula is due to Aitken~\cite[p.\ 310]{Aitken}, and was
rediscovered by Feit~\cite{Feit}. No product or hook length
formula is known in this generality (but a product formula for a
staircase minus a rectangle has been found by
DeWitt~\cite{DeWitt}; see also~\cite{Kratt_Schlosser}). Specific
classes of skew shapes, such as zigzags and strips of constant
width, have interesting special formulas; see
Section~\ref{AR_s:formulas_skew_strips}.


\begin{theorem}\label{t.AR_num_skew_det}%
{\rm\ (Skew determinantal
formula)~\cite{Aitken}\cite{Feit}\cite[Corollary
7.16.3]{Stanley_EC2}} The number of SYT of skew shape $\la/\mu$,
for partitions $\la = (\la_1, \ldots, \la_t)$ and $\mu = (\mu_1,
\ldots, \mu_s)$ with $\mu_i \le \la_i$ $(\forall i)$, is
\[
f^{\la/\mu} = |\la/\mu|! \cdot \det \left[\frac{1}{(\la_i - \mu_j
- i + j)!} \right]_{i, j = 1}^{t},
\]
with the conventions $\mu_j := 0$ for $j > s$ and $1/k! := 0$ for
negative integers $k$.
\end{theorem}
The following proof is inductive. There is another approach that
uses the Jacobi-Trudi identity.

\begin{proof} (Adapted from~\cite{Feit})\\
By induction on the size $n := |\la/\mu|$. Denote
\[
a_{ij} := \frac{1}{(\la_i - \mu_j - i + j)!}.
\]

For $n= 0$, $\la_i = \mu_i$ $(\forall i)$. Thus
\[
i = j \,\then\, \la_i - \mu_i - i  + i = 0 \,\then\, a_{ii} = 1
\]
and
\[
i > j \,\then\, \la_i - \mu_j - i  + j < \la_i - \mu_j = \la_i -
\la_j \le 0 \,\then\, a_{ij} = 0.
\]
Hence the matrix $(a_{ij})$ is upper triangular with diagonal
entries $1$, and $f^{\la/\mu} = 1 = 0! \det (a_{ij})$.

For the induction step assume that the claim holds for all skew
shapes of size $n-1$, and consider a shape $\la/\mu$ of size $n$
with $t$ rows. The cell containing $n$ must be the last cell in
its row and column. Therefore
\[
f^{\la/\mu} = \sum_{i'} f^{(\la/\mu)_{i'}}
\]
where $(\la/\mu)_{i'}$ is the shape $\la/\mu$ minus the last cell
in row $i'$, and summation is over all the rows $i'$ which are
nonempty and whose last cell is also last in its column.
Explicitly, summation is over all $i'$ such that $\la_{i'} >
\mu_{i'}$ as well as $\la_{i'} > \la_{i'+1}$. By the induction
hypothesis,
\[
f^{\la/\mu} = (n-1)!\, \sum_{i'} \det\, ( a_{ij}^{(i')} )
\]
where $a_{ij}^{(i')}$ is the analogue of $a_{ij}$ for the shape
$(\la/\mu)_{i'}$ and summation is over the above values of $i'$.
In fact,
\[
a_{ij}^{(i')} = \begin{cases}
a_{ij}, & \hbox{if } i \ne i'; \\
(\la_i - \mu_j - i + j) \cdot a_{ij}, & \hbox{if } i = i'.
\end{cases}
\]
This holds for all values (positive, zero or negative) of $\la_i -
\mu_j - i + j$. The rest of the proof consists of two steps.

\medskip

{\bf Step 1:}
The above formula for $f^{\la/\mu}$ holds with summation extending over all $1 \le i' \le t$.\\
Indeed, it suffices to show that
\[
\la_{i'} = \mu_{i'} \hbox{ or } \la_{i'} = \la_{i'+1} \,\then\,
\det\, ( a_{ij}^{(i')} ) = 0.
\]
If $ \la_{i'} = \la_{i'+1}$ then
\[
\la_{i'+1} - \mu_j - (i'+1) + j = (\la_{i'} - 1) - \mu_j - i' + j
\qquad (\forall j),
\]
so that the matrix $( a_{ij}^{(i')} )$ has two equal rows and
hence its determinant is $0$. If  $\la_{i'} = \mu_{i'}$ then
\[
j \le i' < i \,\then\, \la_i - \mu_j - i + j < \la_i - \mu_j \le
\la_i - \mu_{i'} \le \la_{i'} - \mu_{i'} = 0
\]
and
\[
j \le i' = i \,\then\, (\la_{i'} - 1) - \mu_j - i' + j < \la_{i'}
- \mu_j \le \la_{i'} - \mu_{i'} = 0.
\]
Thus the matrix $( a_{ij}^{(i')} )$ has a zero submatrix
corresponding to $j \le i' \le i$, which again implies that its
determinant is zero -- e.g., by considering the determinant as a
sum over permutations $\sigma \in \Sc_t$ and noting that, by the
pigeon hole principle, there is no permutation satisfying $j =
\sigma(i) > i'$ for all $i \ge i'$.

\medskip

{\bf Step 2:} Let $A_{ij}$ be the $(i,j)$-cofactor of the matrix
$A = (a_{ij})$, so that
\[
\det A = \sum_j a_{ij} A_{ij} \qquad (\forall i)
\]
and also
\[
\det A = \sum_i a_{ij} A_{ij} \qquad (\forall j).
\]
Then, expanding along row $i'$,
\[
\det (a_{ij}^{i'}) = \sum_j a_{i'j}^{(i')} A_{i'j} = \sum_j
(c_{i'} - d_j) a_{i'j} A_{i'j}
\]
where $c_i := \la_i - i$ and $d_j := \mu_j - j$. Thus
\begin{eqnarray*}
\frac{f^{\la/\mu}}{(n-1)!} &=&  \sum_{i' = 1}^{t} \det\,
(a_{ij}^{(i')})
= \sum_{i'} \sum_j (c_{i'} - d_j) a_{i'j} A_{i'j} \\
&=& \sum_{i'} \sum_j c_{i'} a_{i'j} A_{i'j} - \sum_{i'} \sum_j d_j a_{i'j} A_{i'j} \\
&=& \sum_{i'} c_{i'} \det A - \sum_j d_j \det A
= \left( \sum_{i'} c_{i'} - \sum_j d_j \right) \det A \\
&=& \left( \sum_{i'} \la_{i'} - \sum_j \mu_j \right) \det A =
|\la/\mu| \det A = n \det A
\end{eqnarray*}
which completes the proof.

\end{proof}

%
%

\subsection{Shifted shapes}\label{AR_s:shifted_shapes}

For a strict partition $\la$, let $g^{\la} := |\SYT(\la^*)|$ be
the number of standard Young tableaux of shifted shape $\la$. Like
ordinary shapes, shifted shapes have three types of formulas --
product, hook length and determinantal. The product formula was
proved by Schur~\cite{Schur}, using representation theory, and
then by Thrall~\cite{Thrall}, using recursion and combinatorial
arguments.


\begin{theorem}\label{t.AR_num_shifted_prod}%
{\rm\ (Schur's shifted product
formula)~\cite{Schur}\cite{Thrall}\cite[p.\ 267, eq.\ (2)]{Md}}
For any strict partition $\la = (\la_1, \ldots, \la_t)$,
\[
g^{\la} = \frac{|\la|!}{\prod_{i=1}^{t} \la_i!} \cdot
\prod_{(i,j):\, i < j} \frac{\la_i - \la_j}{\la_i + \la_j}.
\]
\end{theorem}


\begin{definition}
For a cell $c = (i,j) \in [\la^*]$ let
\[
H_{c}^* := [\la^*] \cap \left( \{(i,j)\} \cup \{(i,j') \,|\, j' >
j\} \cup \{(i',j) \,|\, i' > i\} \cup \{(j+1,j') \,|\, j' \ge
j+1\} \right)
\]
be the correponding {\dem shifted hook}; note that the last set is
relevant only for $j < t$. Let
\[
h_{c}^* := |H_{c}^*| =
\begin{cases}
\la_i + \la_{j+1}, &\hbox{\rm if } j < t; \\
\la_i - j + |\{i' \,|\, i' \ge i, \, \la_{i'} + i' \ge j + 1\}|,
&\hbox{\rm if } j \ge t.
\end{cases}
\]
be the corresponding {\dem shifted hook length}.
\end{definition}

For example, in the following diagram the cells in the shifted
hook $H_{(1,2)}^*$ are marked
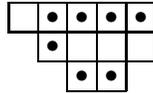
\[
\ydiagram{5, 1 + 4, 2 + 2} * [\bullet]{1 + 4, 1 + 1, 2 + 2}
\]
and in the following diagram each cell is labeled by the
corresponding shifted hook length.
\[
\ytableaushort{97542, \none6431, \none\none21}
\]

\begin{theorem}\label{t.AR_num_shifted_hook}%
{\rm\ (Shifted hook length formula)~\cite[p.\ 267, eq.\ (1)]{Md}}
For any strict partition $\la = (\la_1, \ldots, \la_t)$,
\[
g^{\la} = \frac{|\la|!}{\prod_{c \in [\la^*]} h_{c}^*}.
\]
\end{theorem}



\begin{theorem}\label{t.AR_num_shifted_det}%
{\rm\ (Shifted determinantal formula)} For any strict partition
$\la = (\la_1, \ldots, \la_t)$,
\[
g^{\la} = \frac{|\la|!}{\prod_{(i,j):\, i < j} (\la_i + \la_j)}
\cdot \det \left[\frac{1}{(\la_i - t + j)!} \right]_{i, j =
1}^{t},
\]
using the convention $1/k! := 0$ for negative integers $k$.
\end{theorem}

The formulas in Theorems~\ref{t.AR_num_shifted_prod},
\ref{t.AR_num_shifted_hook} and \ref{t.AR_num_shifted_det} can be
shown to be equivalent in much the same way as was done for
ordinary shapes in Subsection~\ref{AR_s:classical_shapes}. Note
that the factors of the first denominator in the determinantal
formula (Theorem~\ref{t.AR_num_shifted_det}) are precisely the
shifted hook lengths $h_c^*$ for cells $c = (i,j)$ in the region $j < t$.

\section{More proofs of the hook length formula}\label{AR_s:proof_approaches}

\subsection{A probabilistic proof 
}

Probabilistic proofs rely on procedures for a random choice of an
object from a set. The 
key observation is that a uniform distribution implies an exact
evaluation and ``almost uniform" distributions yield good bounds.


A seminal example is the Greene-Nijenhuis-Wilf probabilistic proof
of the ordinary hook length formula, to be described here.
Our outline follows Sagan's description, in the first edition
of~\cite{Sagan_book}, of the original proof of Greene, Nijenhuis
and Wilf~\cite{Greene-etal}.

We start with a procedure that generates a random SYT of a given
ordinary shape $D$. Recall from Definition~\ref{AR_d:hook} the
notions of {\dem hook} $H_c$ and {\dem hook length} $h_c$
corresponding to a cell $c \in D$. A {\dem corner} of $D$ is a
cell which is last in its row and in its column (equivalently, has
hook length $1$).


%
%
%
%

\medskip

\begin{algorithmic}[1]
\Require{$D$, a diagram of ordinary shape.} \Ensure{A random $T
\in \SYT(D)$.} \Statex \Procedure{RandomSYT}{$D$} \While{$D$ is
not empty}
    \State $n \gets |D|$
    \State Choose randomly a cell $c \in D$ \Comment with uniform probability $1/n$
    \While{$c$ is not a corner of $D$}
        \State Choose randomly a cell $c' \in H_c \setminus \{c\}$
        \Comment with uniform probability $1/(h_c - 1)$
        \State $c \gets c'$
    \EndWhile
    \State $T(c) \gets n$
    \State $D \gets D \setminus \{c\}$
\EndWhile \State \textbf{return} $T$ \EndProcedure
\end{algorithmic}

\medskip


We claim that this procedure produces each SYT of shape $D$ with
the same probability. More precisely,
\begin{lemma}\label{AR_t:GNW-lemma1}
The procedure $\textsc{RandomSYT}$ produces each SYT of shape $D$
with probability
\[
p = \frac{1}{|D|!} \prod_{c \in D} h_c \, .
\]
\end{lemma}

\begin{proof}
By induction on $n := |D|$. The claim clearly holds for $n = 0,
1$.


Suppose that the claim holds for all shapes of size $n-1$, and let
$D$ be an ordinary shape of size $n$. Let $T \in \SYT(D)$, and
assume that $T(v) = n$ for some corner $v = (\alpha, \beta)$.
Denote $D' := D \setminus \{v\}$, and let $T' \in \SYT(D')$ be the
restriction of $T$ to $D'$.

In order to produce $T$, the algorithm needs to first produce $v$
(in rows 4--8, given $D$), and then move on to produce $T'$ from
$D'$. By the induction hypothesis, it suffices to show that the
probability that rows 4--8 produce the corner $v = (\alpha,
\beta)$ is
\[
\frac{\prod_{c \in D} h_c / n!}{\prod_{c \in D'} h'_c / (n-1)!} =
\frac{1}{n} \prod_{c \in D'} \frac{h_c}{h'_c} = \frac{1}{n}
\prod_{i = 1}^{\alpha - 1} \frac{h_{i,\beta}}{h_{i,\beta} -1}
\prod_{j = 1}^{\beta - 1} \frac{h_{\alpha,j}}{h_{\alpha,j} -1} \,
,
\]
where $h'_c$ denotes hook length in $D'$.
This is equal to
\[
\frac{1}{n} \prod_{i = 1}^{\alpha - 1} \left( 1 +
\frac{1}{h_{i,\beta}-1} \right)
\prod_{j = 1}^{\beta - 1} \left( 1 + \frac{1}{h_{\alpha,j}-1} \right) 
= \frac{1}{n} \sum\limits_{A \subseteq [\alpha-1] \atop B\subseteq
[\beta-1]} \prod_{i \in A} \frac{1}{h_{i,\beta} -1} \prod_{j \in
B} \frac{1}{h_{\alpha,j} -1} \, .
\]

Following Sagan~\cite{Sagan_book} we call any possible sequence of
cells of $D$ obtained by lines 4--8 of the procedure (starting at
a random $c$ and ending at the given corner $v$) a {\dem trial}.
For each trial $\tau$, let
\[
A(\tau) := \{ i < \alpha \,:\, \exists j \text{ s.t. } (i,j)
\text{ is a cell in the trial } \tau \} \subseteq [1, \alpha - 1]
\]
be its {\dem horizontal projection} and let
\[
B(\tau) := \{ j < \beta \,:\, \exists i \text{ s.t. } (i,j) \text{
is a cell in the trial } \tau \} \subseteq [1, \beta - 1] \, .
\]
be its {\dem vertical projection}.

\smallskip

It then suffices to show that for any given $A \subseteq [1,
\alpha - 1]$ and $B \subseteq [1, \beta - 1]$, the sum of
probabilities of all trials $\tau$ ending at $v = (\alpha, \beta)$
such that $A(\tau) = A$ and $B(\tau) = B$ is
\[
\frac{1}{n} \prod_{i \in A} \frac{1}{h_{i,\beta} -1} \prod_{j \in
B} \frac{1}{h_{\alpha,j} -1} \, .
\]
This may be proved by induction on $|A \cup B|$.

\end{proof}

Lemma~\ref{AR_t:GNW-lemma1} says that the algorithm produces each
$T \in \SYT(D)$ with the same probability $p$. The number of SYT
of shape $D$ is therefore $1/p$, proving the hook length formula
(Theorem~\ref{t.AR_num_ordinary_hook}).



For a fully detailed proof see~\cite{Greene-etal} or the first
edition of~\cite{Sagan_book}.

\medskip

A similar method was applied in~\cite{Sagan80} to prove the hook
length formula for shifted shapes
(Theorem~\ref{t.AR_num_shifted_hook} above).


\subsection{Bijective proofs}\label{bijective}





There are several bijective proofs of the (ordinary) hook length
formula. Franzblau and Zeilberger~\cite{FZ} gave a bijection which
is rather simple to describe, but breaks the row-column symmetry
of hooks. Remmel~\cite{Remmel82} used the Garsia-Milne involution
principle~\cite{GarsiaMilne81} to produce a composition of maps,
``bijectivizing'' recurrence relations.
Zeilberger~\cite{Zeilberger_DM1984} then gave a bijective version
of the probabilistic proof of Greene, Nijenhuis and
Wilf~\cite{Greene-etal} (described in the previous subsection).
Krattenthaler~\cite{Kratt95} combined the Hillman-Grassl
algorithm~\cite{HillmanGrassl} and Stanley's
$(P,\omega)$-partition theorem with the involution principle.
Novelli, Pak and Stoyanovskii~\cite{NovelliPakSto97} gave a
complete proof of a bijective algorithm previously outlined by Pak
and Stoyanovskii~\cite{PakSto92}. A generalization of their method
was given by Krattenthaler~\cite{Kratt98}.

Bijective proofs for the shifted hook length formula were given by
Krattenthaler~\cite{Kratt95} and Fischer~\cite{Fischer}.

A bijective proof of the ordinary determinantal formula was
given by Zeilberger~\cite{Zeilberger_DM1983}; see
also~\cite{Linial} and~\cite{Kratt89}.
%

We shall briefly describe here the bijections of Franzblau-Zeilberger 
and of Novelli-Pak-Stoyanovskii. 
Only the algorithms (for the map in one direction) will be
specified; the interested reader is referred to the original
papers (or to~\cite{Sagan_book}) for more complete descriptions
and proofs.

The basic setting for both bijections is the following.

\begin{definition}
Let $\la$ be a partition of $n$ and $A$ a set of positive integers
such that $|A| = |\la|$. A {\dem Young tableaux} of (ordinary)
shape $\la$ and image $A$ is a bijection $R: [\la] \to A$, not
required to be order-preserving. A {\dem pointer tableau} (or
{\dem hook function}) of shape $\la$ is a function $P: [\la] \to
\bbz$ which assigns to each cell $c \in [\la]$ a pointer $p(c')$
which encodes some cell $c'$ in the hook $H_c$ of $c$ (see
Definition~\ref{AR_d:hook}). The pointer corresponding to $c' \in
H_c$ is defined as follows:
\[
p(c') :=
\begin{cases}
j, & \text{if $c'$ is $j$ columns to the right of $c$, in the same row;} \\
0, & \text{if $c' = c$;} \\
-i, & \text{if $c'$ is $i$ rows below $c$, in the same column.}
\end{cases}
\]
Let $\YT(\la, A)$ denote the set of all Young tableaux of shape
$\la$ and image $A$, $\PT(\la)$ the set of all pointer tableaux of
shape $\la$, and $\SPT(\la, A)$ the set of all pairs $(T, P)$
(``standard and pointer tableaux'') where $T \in \SYT(\la, A)$ and
$P \in \PT(\la)$. $\YT(\la)$ is a shorthand for $\YT(\la, [n])$
where $n = |\la|$, and $\SPT(\la)$ a shorthand for $\SPT(\la,
[n])$.
\end{definition}
\begin{example}
A typical hook, with each cell marked by its pointer:
\[
\ytableaushort{0 1 2 3 4, \mone, \mtwo}
\]
\end{example}

The hook length formula that we want to prove may be written as
\[
n! = f^{\la} \cdot \prod_{c \in [\la]} h_c \, .
\]
The LHS of this formula is the size of $\YT(\la)$, while the RHS
is the size of $\SPT(\la)$. Any explicit bijection $f: \YT(\la)
\to \SPT(\la)$ will prove the hook length formula. As promised, we
shall present algorithms for two such bijections.

\bigskip

\noindent {\dem The Franzblau-Zeilberger algorithm~\cite{FZ}:} The
main procedure, $\textsc{FZ-SortTableau}$, ``sorts'' a YT $R$ of
ordinary shape, column by column from right to left, to produce a
SPT $(T, P)$ of the same shape. The pointer tableau $P$ records
each step of the sorting, keeping just enough information to
enable reversal of the procedure. $\varnothing$ denotes the empty
tableau.

\medskip

\begin{algorithmic}[1]
\Require{$R \in \YT(\la)$.} \Ensure{$(T, P) \in \SPT(\la)$.}
\Statex \Procedure{FZ-SortTableau}{$R$} \State $(T, P) \gets
(\varnothing, \varnothing)$ \Comment{Initialize} \State $m \gets$
number of columns of $R$ \For{$j \gets m$ \textbf{downto} $1$}
\Comment{Add columns from right to left}
   \State $c \gets$ column $j$ of $R$
   \State $(T, P) \gets$ \Call{InsertColumn}{$T, P, c$}
\EndFor \State \textbf{return} $(T, P)$ \EndProcedure
\end{algorithmic}

\medskip

The algorithm makes repeated use of the following procedure
$\textsc{InsertColumn}$:

\medskip

\begin{algorithmic}[1]
\Require{$(T, P, c)$, where $(T, P) \in \SPT(\mu, A)$ for some
ordinary shape $\mu$ and some set $A$ of positive numbers of size
$|A| = |\mu|$ such that all the rows of $T$ are increasing, and $c
= (c_1, \dots, c_m)$ is a vector of distinct positive integers
$c_i \not\in A$ whose length $m \ge \ell(\mu)$.} \Ensure{$(T', P')
\in \SPT(\mu', A')$, where $A' = A \cup \{c_1, \ldots, c_m\}$ and
$\mu'$ is obtained from $\mu$ by attaching a new first column of
length $m$.} \Statex \Procedure{InsertColumn}{$T, P, c$} \For {$i
\gets 1$ \textbf{to} $m$}
    \State $T \gets$ \Call{Insert}{$T, i, c_i$}
    \Comment{Insert $c_i$ into row $i$ of $T$, keeping the row entries increasing}
    \State $d_i \gets$ (new column index of $c_i) - 1$
    \Comment{Initialize the pointer $d_i$}
\EndFor \While{$T$ is not a Standard Young Tableau}
    \State $(k,x) \gets T^{-1}(\min \{T(i,j) \,|\, T(i-1, j) > T(i, j) \})$
    \Comment{The smallest entry out of order}
    \State $y \gets d_{k-1} + 1$
    \Comment{Claim: $y > 0$}
    \State $T \gets$ \Call{Exchange}{$T, (k, x), (k-1, y)$}
    \State $y' \gets$ new column index of the old $T(k-1, y)$
    \Comment{The new row index is $k$}
    \State \Comment{Update the pointers $d_{k-1}$ and $d_k$}
    \State $d_{k-1} \gets \begin{cases}
    v, & \text{if } d_k = v \ge 0,\ v \ne x-1; \\
    -1, & \text{if } d_k = x-1; \\
    -(u+1), & \text{if } d_k = -u < 0.
    \end{cases}$
    \State $d_k \gets y'-1$
\EndWhile \State $P \gets$ \Call{Attach}{$d, P$} \Comment{Attach
$d$ to $P$ as a first column} \State \textbf{return} $(T, P)$
\Comment{$T$ is now a SYT} \EndProcedure
\end{algorithmic}

\medskip

This procedure makes use of some elementary operations, which may
be described as follows:
\begin{itemize}
\item $\textsc{Insert}(T, i, c_i)$ inserts the entry $c_i$ into
row $i$ of $T$, reordering this row to keep it increasing. \item
$\textsc{Exchange}(T, (k, x), (\ell, y))$ exchanges the entries in
cells $(k, x)$ and $(\ell, y)$ of $T$ and then reorders rows $k$
and $\ell$ to keep them increasing. \item $\textsc{Attach}(d, P)$
attaches the vector $d$ to the pointer tableau $P$ as a new first
column.
\end{itemize}

\begin{example}
An instance of $\textsc{InsertColumn}(T, P, c)$ with
\[
T =\, \ytableaushort{1 8, 4, 7} \quad , \quad P  =\,
\ytableaushort{1 0, 0, 0} \quad , \quad c =\, \ytableaushort{{12},
5, 3, 6}
\]
proceeds as follows (with the smallest entry out of order set in boldface):
\[
(T, d) =\, \ytableaushort{1 8 {12}, 4 5, {\bnum{3}} 7, 6} \quad
\ytableaushort{2, 1, 0, 0} \quad \to \quad \ytableaushort{1 8
{12}, 3 {\bnum{4}}, 5 7, 6} \quad \ytableaushort{2, \mone, 0, 0}
\quad \to \quad \ytableaushort{1 4 8, 3 {12}, 5 {\bnum{7}}, 6}
\quad \ytableaushort{\mtwo, 1, 0, 0} \quad \to \quad
\ytableaushort{1 4 8, 3 7, 5 {12}, 6} \quad \ytableaushort{\mtwo,
0, 1, 0}
\]
and yields
\[
T =\, \ytableaushort{1 4 8, 3 7, 5 {12}, 6} \quad , \quad P  =\,
\ytableaushort{\mtwo 1 0, 0 0, 1 0, 0} \quad .
\]
An instance of $\textsc{FZ-SortTableau}(R)$ with
\[
R =
\ytableaushort{9 {12} 8 1, 2 5 4, {11} 3 7, {10} 6}
\]
proceeds as follows (the second step being the instance above):
\[
\ytableaushort{1} \quad \ytableaushort{0} \quad \to \quad
\ytableaushort{1 8, 4, 7} \quad \ytableaushort{1 0, 0, 0} \quad
\to \quad \ytableaushort{1 4 8, 3 7, 5 {12}, 6} \quad
\ytableaushort{\mtwo 1 0, 0 0, 1 0, 0} \quad \to \quad
\ytableaushort{1 3 4 8, 2 7 9, 5 {10} {12}, 6 {11}} \quad
\ytableaushort{0 \mtwo 1 0, 2 0 0, \mone 1 0, 1 0} \quad = (T, P)
\quad .
\]
\end{example}


\bigskip

\noindent {\dem The Novelli-Pak-Stoyanovskii
algorithm~\cite{NovelliPakSto97}:} Again, we prove the hook length
formula
\[
n! = f^{\la} \cdot \prod_{c \in [\la]} h_c
\]
by building an explicit bijection $f: \YT(\la) \to \SPT(\la)$.
However, instead of building the tableaux column by column, we
shall use a modified jeu de taquin to unscramble the entries of $R
\in \YT(\la)$ so that rows and columns increase. Again, a pointer
tableau will keep track of the process so as to make it
invertible. Our description will essentially
follow~\cite{Sagan_book}.

First, define a linear (total) order on the cells of a diagram $D$
of ordinary shape by defining
\[
(i, j) \unlhd (i', j') \iff \text{either } j > j' \text{ or } j =
j' \text{ and } i \ge i'.
\]
For example, the cells of the following diagram are labelled $1$
to $7$ according to this linear order:
\[
\ytableaushort{7 4 2, 6 3 1, 5}
 \]
If $R \in \YT(\la)$ and $c \in D := [\la]$, let $R^{\unlhd c}$
(respectively $R^{\lhd c}$) be the tableau consisting of all cells
$b \in D$ with  $b \unlhd c$ (respectively, $ \lhd c$).

Define a procedure $\textsc{MForwardSlide}$ which is the procedure
$\textsc{ForwardSlide}$ from the description of jeu de taquin,
with the following two modifications:
\begin{enumerate}
\item Its input is $(T, c)$ with $T \in \YT$ rather than $T \in
\SYT$. \item Its output is $(T, c)$ (see there), rather than just
$T$.
\end{enumerate}

\medskip

\begin{algorithmic}[1]
\Require{$R \in \YT(\la)$.} \Ensure{$(T, P) \in SPT(\la)$.}
\Statex \Procedure{NPS}{$R$} \State $T \gets R$ \State $D \gets
\sh(T)$ \State $P \gets 0 \in \PT(D)$ \Comment A pointer tableau
of shape $D$ filled with zeros \While{$T$ is not standard}
    \State $c \gets$ the $\lhd$-maximal cell such that $T^{\lhd c}$ is standard
    \State $(T', c') \gets$ \Call{MForwardSlide}{$T^{\lhd c}, c$}
    \For {$b \in D$}
        \Comment Replace $T^{\unlhd c}$ by $T'$, except that $T(c') \gets$ the old $T(c)$
        \State $T(b) \gets
        \begin{cases}
        T(b), & \text{if } b \rhd c; \\
        T'(b), & \text{if } b \unlhd c \text{ and } b \ne c'; \\
        T(c), & \text{if } b = c'.
        \end{cases}$
    \EndFor
    \State Let $c = (i_0, j_0)$ and $c' = (i_1, j_1)$
    \Comment Necessarily $i_0 \le i_1$ and $j_0 \le j_1$
    \For {$i \textbf{ from } i_0 \textbf{ to } i_1 - 1$}
        \State $P(i, j_0) \gets P(i+1, j_0) - 1$
    \EndFor
    \State $P(i_1, j_0) \gets j_1 - j_0$
\EndWhile \State \textbf{return} $T$ \Comment Now $c_0 \in D
\subseteq (c_0)_+$, so $D$ has ordinary shape \EndProcedure
\end{algorithmic}

\medskip

\begin{example}
For
\[
R = \, \ytableaushort{6 2, 4 3, 5 1} \quad ,
\]
here is the sequence of pairs $(T, P)$ produced during the
computation of $\textsc{NPS}(R)$ (with $c$ in boldface):
\[
\ytableaushort{6 2, 4 {\bnum{3}}, 5 1} \quad \ytableaushort{0 0, 0
0, 0 0} \quad \to \quad \ytableaushort{6 {\bnum{2}}, 4 1, 5 3}
\quad \ytableaushort{0 0, 0 \mone, 0 0} \quad \to \quad
\ytableaushort{6 1, 4 2, {\bnum{5}} 3} \quad \ytableaushort{0
\mtwo, 0 0, 0 0} \quad \to \quad
\]
\[
\quad \to \quad \ytableaushort{6 1, {\bnum{4}} 2, 3 5} \quad
\ytableaushort{0 \mtwo, 0 0, 1 0} \quad \to \quad
\ytableaushort{{\bnum{6}} 1, 2 4, 3 5} \quad \ytableaushort{0
\mtwo, 1 0, 1 0} \quad \to \quad \ytableaushort{1 4, 2 5, 3 6}
\quad \ytableaushort{0 \mtwo, 0 0, 1 0} \quad .
\]
\end{example}


\subsection{Partial difference operators}\label{AR_s:DZ}

MacMahon~\cite{MacMahon_book} has originally used partial
difference equations, also known as recurrence relations, to solve
various enumeration problems -- among them the enumeration of
ballot sequences, or equivalently SYT of an ordinary shape (see
Subsection~\ref{AR_s:lattice_paths}).
Zeilberger~\cite{Zeilberger_DM1980} improved on MacMahon's proof
by extending the domain of definition of the enumerating
functions, thus simplifying the boundary conditions: In PDE
terminology, a Neumann boundary condition (zero normal
derivatives) was replaced by a Dirichlet boundary condition (zero
function values). He also made explicit use of the algebra of
partial difference operators; we shall present here a variant of
his approach.

Consider, for example, the two dimensional ballot problem --
finding the number $F(m_1, m_2)$ of lattice paths from $(0, 0)$ to
$(m_1, m_2)$ which stay in the region $m_1 \ge m_2 \ge 0$.
MacMahon~\cite[p.\ 127]{MacMahon_book} has set the partial
difference equation
\[
F(m_1, m_2) =  F(m_1 - 1, m_2) + F(m_1, m_2 - 1) \qquad (m_1 > m_2
> 0)
\]
with the boundary conditions
\[
F(m_1, m_2) = F(m_1, m_2 - 1) \qquad (m_1 = m_2 > 0)
\]
and
\[
F(m_1, 0) = 1 \qquad (m_1 \ge 0).
\]
By extending $F$ to the region $m_1 \ge m_2 - 1$, the recursion
can be required to hold for almost all $m_1 \ge m_2 \ge 0$:
\[
F(m_1, m_2) =  F(m_1 - 1, m_2) + F(m_1, m_2 - 1) \qquad (m_1 \ge
m_2 \ge 0,\, (m_1, m_2) \ne (0,0))
\]
with
\[
F(m_1, m_2) = 0 \qquad(m_1 = m_2 - 1 \hbox{\ \ or\ \ } m_2 = -1)
\]
and
\[
F(0,0) = 1.
\]
In general,
consider functions $f : \bbz^n \to \bbc$ 
and define the fundamental {\dem shift operators} $X_1, \ldots,
X_n$ by
\[
(X_i f)(m_1, \ldots, m_n) := f(m_1, \ldots, m_i + 1, \ldots, m_n)
\qquad (1 \le i \le n).
 \]
For $\alpha \in \bbz^n$ write $X^{\alpha} = X_1^{\alpha_1} \cdots
X_n^{\alpha_n}$, so that $(X^{\alpha} f)(m) = f(m + \alpha)$. A
typical linear partial difference operator with constant
coefficients has the form
\[
P = p(X_1^{-1}, \ldots, X_n^{-1}) = \sum_{\alpha \ge 0} a_{\alpha}
X^{-\alpha},
\]
for some polynomial $p(z)$ with complex coefficients, so that
$a_\alpha \in \bbc$ for each $\alpha = (\alpha_1, \ldots,
\alpha_n) \in \bbn^n$ and $a_\alpha \ne 0$ for only finitely many
values of $\alpha$. We also assume that $p(0) = a_0 = 1$.

\begin{definition}
Define the {\dem discrete delta function} $\delta: \bbz^n \to
\bbc$ by
\[
\delta(m) = \begin{cases}
1, & \hbox{if } m = 0; \\
0, & \hbox{otherwise}.
\end{cases}
\]
A function $f :\bbz^n \to \bbc$ satisfying $Pf = \delta$ is called
a {\dem fundamental solution} corresponding to the operator $P$.
If $f$ is supported in $\bbn^n$, it is called a {\dem canonical}
fundamental solution.
\end{definition}
It is clear that each operator $P$ as above has a unique canonical
fundamental solution.

In the following theorem we consider a slightly more general type
of operators, which can be written as $X^\alpha p(X_1^{-1},
\ldots, X_n^{-1})$ for a polynomial $p$ and some $\alpha \ge 0$.

\begin{theorem}\label{AR_t:DZ_thm2}%
{\rm (A variation on~\cite[Theorem 2]{Zeilberger_DM1980})} Let
$F_n = F_n(m_1, \ldots, m_n)$ be the canonical fundamental
solution corresponding to an operator $P = p(X_1^{-1}, \ldots,
X_n^{-1})$, where $p(z)$ is a symmetric polynomial with $p(0) =
1$.
Denote
\[
\Delta_n := \prod_{(i,j):\, i<j} (I - X_i X_j^{-1}).
\]
Then $G_n = \Delta_n F_n$ is the unique solution of the equation
$Pg =0$ in the region
\[
\{(m_1, \ldots, m_n) \in \bbz^n \,|\, m_1 \ge \ldots \ge m_n \ge
0\} \setminus \{(0, \ldots, 0)\}
\]
subject to the boundary conditions
\[
(\exists i)\, m_i = m_{i+1} - 1 \,\then\, g(m_1, \ldots, m_n) = 0,
\]
\[
m_n = -1 \,\then\, g(m_1, \ldots, m_n) = 0
\]
and
\[
g(0, \ldots, 0) = 1.
\]
\end{theorem}
\begin{proof} 
Since each $X_i$ commutes with the operator $P$, so does
$\Delta_n$.
Since $m_1 + \ldots + m_n$ is invariant under $X_i X_j^{-1}$ and
$F_n$ is a solution of $Pf = 0$ in the 
complement of the hyperplane $m_1 + \ldots + m_n = 0$, so is
$G_n$. It remains to verify that $G_n$ satisfies the prescribed
boundary conditions. Now, by definition,
\[
G_n(m) = (I - X_1 X_2^{-1}) A_{1,2} F_n(m)
\]
where the operator $A_{1,2}$ is symmetric with respect to $X_1$,
$X_2$. Since $F_n(m)$ is a symmetric function, we can write
\[
G_n(m) = (I - X_1 X_2^{-1}) H(m)
\]
where $H$ is symmetric with respect to $m_1$ and $m_2$.
Suppressing the dependence on $m_3, \ldots, m_n$,
\[
G_n(m_1, m_2) = (I - X_1 X_2^{-1}) H(m_1, m_2) = H(m_1, m_2) -
H(m_1 + 1, m_2 - 1) = 0
\]
whenever $m_1 = m_2 - 1$, by the symmetry of $H$. Similarly, for
$i = 1, \ldots, n-1$, $G_n(m) = 0$ on $m_i = m_{i+1} - 1$. On $m_n
= -1$,
\[
G_n(m_1, \ldots, m_{n-1}, -1) =  \prod_{(i,j): 1 \le i < j \le
n-1} (I - X_i X_j^{-1}) \prod_{1 \le i \le n-1} (I - X_i X_n^{-1})
F_n(m_1, \ldots, m_{n-1}, -1).
\]
Since $F_n(m) = 0$ for $m_n < 0$,
\[
G_n(m_1, \ldots, m_{n-1}, -1) =  0.
\]
Finally, $F_n(m_1, \ldots, m_n) = 0$ on all of the hyperplane $m_1
+ \ldots + m_n = 0$ except the origin $0 = (0, \ldots, 0)$.
Therefore
\[
G_n(0) = \Delta_n F_n(0) = F_n(0) = 1.
\]

\end{proof}

For every function $f: \bbz^n \to \bbc$ whose support is contained
in a translate of $\bbn^n$ (i.e., such that there exists $N \in
\bbz$ such that $f(m_1, \ldots, m_n) = 0$ whenever $m_i < N$ for
some $i$) there is a corresponding generating function (formal
Laurent series)
\[
\gf(f) := \sum_m f(m) z^m \in \bbc((z_1, \ldots, z_n)).
\]
Let $p(z_1, \ldots, z_n)$ be a polynomial with complex
coefficients and $p(0) = 1$, and let  $P = p(X_1^{-1}, \ldots,
X_n^{-1})$ be the corresponding operator. Since $\gf(X^{-\alpha}
f) = z^{\alpha} \gf(f)$ we have
\[
\gf(P f) = p(z_1, \ldots, z_n) \cdot \gf(f),
\]
and therefore $P f = \delta$ implies $\gf(f) = 1/p(z_1, \ldots,
z_n)$.

\begin{definition}{\rm (MacMahon~\cite{MacMahon_book})}
Let $A \subseteq \bbz^n$ and $f : A \to \bbc$. A formal Laurent
series $\sum_m a(m) z^m$ is a {\dem redundant generating function
for $f$ on $A$} if $f(m) = a(m)$ for all $m \in A$.
\end{definition}
\begin{theorem}\label{AR_t:DZ_thm5}%
{\rm (MacMahon~\cite[p.\ 133]{MacMahon_book})} Let $g(m)$ be the
number of lattice paths from $0$ to $m$, where travel is
restricted to the region
\[
A = \{(m_1, \ldots, m_n) \in \bbz^n \,|\, m_1 \ge m_2 \ge \ldots
\ge m_n \ge 0\}.
\]
Then
\[
\prod_{(i,j):\, i < j} \left( 1 - \frac{z_j}{z_i} \right) \cdot
\frac{1}{1 - z_1 - \ldots - z_n}
\]
is a redundant generating function for $g$ on $A$ and therefore
\[
g(m) = \frac{(m_1 + \ldots + m_n)!}{(m_1 + n - 1)! \ldots m_n!}
\prod_{(i,j):\, i <j} (m_i - m_j + j - i).
\]
\end{theorem}
This gives, of course, the ordinary product formula
(Theorem~\ref{t.AR_num_ordinary_prod}).
\begin{proof}
Apply Theorem~\ref{AR_t:DZ_thm2} with $P = I - X_1^{-1} - \ldots -
X_n^{-1}$. The canonical fundamental solution of $P f = \delta$ is
easily seen to be the multinomial coefficient
\[
F_n(m_1, \ldots, m_n) = \begin{cases}
\frac{(m_1 + \ldots + m_n)!}{m_1! \cdots m_n!}, & \hbox{if } m_i \ge 0\, (\forall i); \\
0, & \hbox{otherwise,}
\end{cases}
\]
with generating function $(1 - z_1 - \ldots - z_n)^{-1}$.

The number of lattice paths in the statement of
Theorem~\ref{AR_t:DZ_thm5} clearly satisfies the conditions on $g$
in Theorem~\ref{AR_t:DZ_thm2}, and therefore
\[
g = G_n = \Delta_n F_n \qquad (\hbox{on } A).
\]
This implies the claimed redundant generating function for $g$ on
$A$.

To get an explicit expression for $g(m)$ note that $(m_1 + \ldots
+ m_n)!$ is invariant under $X_i X_j^{-1}$, so that
\begin{eqnarray*}
g(m_1, \ldots, m_n) &=& \prod_{(i,j):\, i<j} (I - X_i X_j^{-1})
\left[ \frac{(m_1 + \ldots + m_n)!}{m_1! \cdots m_n!} \right] \\
&=& (m_1 + \ldots + m_n)! \cdot \prod_{(i,j):\, i<j} (I - X_i
X_j^{-1}) \left[ \frac{1}{m_1! \cdots m_n!} \right].
\end{eqnarray*}
Consider
\begin{eqnarray*}
H(m_1, \ldots, m_n)
&:=& \prod_{(i,j):\, i<j} (I - X_i X_j^{-1}) \left[ \frac{1}{m_1! \cdots m_n!} \right] \\
&=& \prod_{(i,j):\, i<j} (X_i^{-1} - X_j^{-1}) \cdot \prod_{i}
X_i^{n-i}
\left[ \frac{1}{m_1! \cdots m_n!} \right] \\
&=& \prod_{(i,j):\, i<j} (X_i^{-1} - X_j^{-1}) \left[
\frac{1}{\ell_1! \ell_2! \cdots \ell_n!} \right],
\end{eqnarray*}
where $\ell_i := m_i + n - i$ ($1 \le i \le n$). Clearly $H$ is an
alternating (anti-symmetric) function of $\ell_1, \ldots, \ell_n$,
which means that
\[
H(m_1, \ldots, m_n) = \frac{Q(\ell_1, \ell_2, \ldots,
\ell_n)}{\ell_1! \ell_2! \cdots \ell_n!},
\]
where $Q$ is an alternating polynomial of degree $n - 1$ in each
of its variables. $g$, and therefore also $H$ and $Q$, vanish on
each of the hyperplanes $m_i = m_{i+1} - 1$, namely $\ell_i =
\ell_{i+1}$ ($1 \le i \le n-1$). Hence $\ell_i - \ell_{i+1}$, and
by symmetry also $\ell_i - \ell_j$ for each $i \ne j$, divide $Q$.
Hence
\[
Q(\ell) = c \prod_{(i,j):\, i<j} (\ell_i - \ell_j)
\]
and
\[
g(m) = c \frac{(m_1 + \ldots + m_n)!}{(m_1 + n -1)! \cdots m_n!}
\prod_{(i,j):\, i<j} (m_i - m_j - i + j)
\]
for a suitable constant $c$, which is easily found to be $1$ by
evaluating $g(0)$.

\end{proof}

\begin{remark}\label{AR_r:DZ_coeff}
Theorem~\ref{AR_t:DZ_thm5} gives an expression of the number of
SYT of ordinary shape $\la = (\la_1, \ldots, \la_t)$ as the
coefficient of $z^\ell$ (where $\ell_i = \la_i + t - i$) in the
power series
\[
\prod_{(i,j):\, i < j} \left( z_i - z_j \right) \cdot \frac{1}{1 -
z_1 - \ldots - z_t},
 \]
or as the constant term in the Laurent series
\[
\prod_i z_i^{-\ell_i} \prod_{(i,j):\, i < j} \left( z_i - z_j
\right) \cdot \frac{1}{1 - z_1 - \ldots - z_t}.
 \]
\end{remark}


\section{Formulas for skew strips}\label{AR_s:formulas_skew_strips}

We focus our attention now on two important families of skew
shapes, which are of special interest: Zigzag shapes and skew
strips of constant width.

\subsection{Zigzag shapes}\label{AR_s:zigzag}

Recall (from Subsection~\ref{AR_s:zigzag_sum}) that a {\dem zigzag
shape} is a path-connected skew shape which does not contain a
$2\times 2$ square.

\begin{definition}\label{AR_d:zigzag_S}
For any subset $S \subseteq [n-1] := \{1, \ldots, n-1 \}$ define a
zigzag shape $D = \zigzag_n(S)$, with cells labeled $1, \ldots,
n$, as follows: Start with an initial cell labeled $1$. For each
$1 \le i \le n-1$, given the cell labeled $i$, add an adjacent
cell labeled $i+1$ above cell $i$ if $i \in S$, and to the right
of cell $i$ otherwise.
\end{definition}
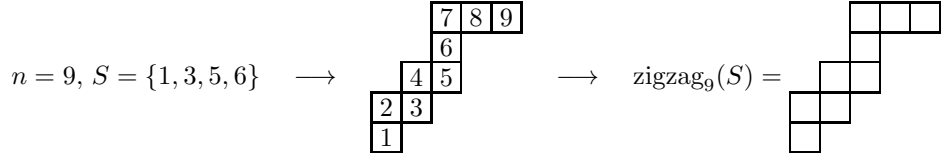
\begin{example}\label{AR_ex:zigzag1}
\[
n = 9,\, S = \{1,3,5,6\} \quad \longrightarrow \quad
\ytableaushort{\none\none789, \none\none6, \none45, 23, 1} \quad
\longrightarrow \quad \zigzag_9(S) =
\ydiagram{2 + 3, 2 + 1, 1+ 2, 2, 1}
\]
\end{example}
This defines a bijection between the set of all subsets of $[n-1]$
and the set of all zigzag shapes of size $n$ (up to translation).
The set $S$ consists of the labels of all the cells in the shape
$\zigzag_n(S)$ such that the cell directly above is also in the
shape. These are exactly the last (rightmost) cells in all the
rows except the top row.

Recording the lengths of all the rows in the zigzag shape, from
bottom up, it follows that zigzag shapes of size $n$ are also in
bijection with all the {\dem compositions} of $n$. In fact, given
a subset $S = \{ s_1, \ldots, s_k \} \subseteq [n-1]$ (with $s_1 <
\ldots < s_k$), the composition corresponding to $\zigzag_n(S)$ is
simply $(s_1, s_2 - s_1, \ldots, s_k - s_{k-1}, n - s_k)$.

\begin{theorem}\label{AR.t.zigzag}%
{\rm \cite[Vol.\ 1, p.\ 190]{MacMahon_book}\cite[Example
2.2.4]{Stanley_EC1}} Let $S=\{s_1,\dots,s_k\}\subseteq [n-1]$
$(s_1 < \ldots < s_k)$ and set $s_0:=0$ and $s_{k+1}:=n$. Then
\[
f^{\zigzag_n(S)} = n! \cdot \det \left[\frac{1}{(s_{j+1} - s_i)!}
\right]_{i, j = 0}^{k} = \det \left[{n - s_i \choose s_{j+1} -
s_i} \right]_{i, j = 0}^{k} = \det \left[{s_{j+1} \choose s_{j+1}
- s_i} \right]_{i, j = 0}^{k}.
\]
\end{theorem}


\medskip

For example, the zigzag shape
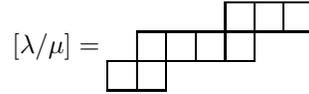
\[
[\la/\mu] =
\ydiagram{4 + 3, 1 + 4, 2}
\]
corresponds to $n = 9$ and $S = \{2,6\}$, and therefore
{\setstretch{1.5}  
\[
f^{\la/\mu} =
9! \cdot \det \begin{bmatrix} 
\frac{1}{2!} & \frac{1}{6!}  & \frac{1}{9!} \\ 
1 & \frac{1}{4!}  & \frac{1}{7!} \\ 
0 & 1 & \frac{1}{3!} \\ 
\end{bmatrix} 
=
\det \begin{bmatrix} 
{9 \choose 2} & {9 \choose 6}  & {9 \choose 9} \\
{7 \choose 0} & {7 \choose 4}  & {7 \choose 7} \\
0 & {3 \choose 0} & {3 \choose 3} \\
\end{bmatrix} 
=
\det \begin{bmatrix} 
{2 \choose 2} & {6 \choose 6}  & {9 \choose 9} \\
{2 \choose 0} & {6 \choose 4}  & {9 \choose 7} \\
0 & {6 \choose 0} & {9 \choose 3} \\
\end{bmatrix} 
\]
}%
Theorem~\ref{AR.t.zigzag} is a special case of the determinantal
formula for skew shapes (Theorem~\ref{t.AR_num_skew_det}). We
shall now consider a specific family of examples.



\begin{example}\label{AR_ex:Andre}
Consider, for each nonnegative integer $n$, one special zigzag
shape $D_n$. For $n$ even it has all rows of length $2$:
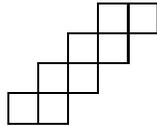
\[
\ydiagram{3 + 2, 2 + 2, 1 + 2, 2}
\]
and for $n$ odd it has a top row of length $1$ and all others of
length $2$:
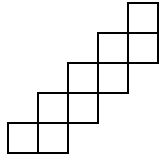
\[
\ydiagram{4 + 1, 3 + 2, 2 + 2, 1 + 2, 2}
\]
Clearly, by Definition~\ref{AR_d:zigzag_S}, $D_n = \zigzag_n(S)$
for $S = \{ 2, 4, 6, \ldots \} \subseteq [n-1]$.
\end{example}
\begin{definition}\label{AR_d:up-down_permutations}
A permutation $\sigma \in S_n$ is {\dem up-down} (or {\dem
alternating},
or {\dem zigzag}) if 
\[
\sigma(1) < \sigma(2) > \sigma(3) < \sigma(4) > \ldots \quad .
\]
\end{definition}
\begin{observation}\label{AR_t:up_down}
If we label the cells of $D_n$ as in
Definition~\ref{AR_d:zigzag_S} (see Example~\ref{AR_ex:zigzag1}),
then clearly each standard Young tableau $T: D_n \to [n]$ becomes
an up-down permutation, and vice versa. (For an extension of this
phenomenon see Proposition~\ref{t.zigzag-descent class}.)
\end{observation}
Up-down permutations were already studied by
Andr{\'e}~\cite{Andre1, Andre2} in the nineteenth century. He
showed that their number $A_n$~\cite[A000111]{oeis} satisfies
\begin{proposition}\label{AR_t:andre}
\[
\sum_{n = 0}^{\infty} \frac{A_n x^n}{n!} = \sec x + \tan x.
\]
\end{proposition}
They are therefore directly related to the {\dem secant} (or {\dem
zig}, or {\dem Euler}) {\dem numbers}
$E_{n}$~\cite[A000364]{oeis}, the {\dem tangent} (or {\dem zag})
{\dem numbers} $T_n$~\cite[A000182]{oeis} and the {\dem Bernoulli
numbers} $B_n$~\cite[A000367 and A002445]{oeis} by
\[
A_{2n} = (-1)^n E_{2n} \qquad (n \ge 0)
\]
and
\[
A_{2n-1} = T_n = \frac{(-1)^{n-1} 2^{2n} (2^{2n} - 1)}{2n} B_{2n}
\qquad (n \ge 1).
\]
Note that there is an alternative convention for denoting Euler
numbers, by which $E_n = A_n$ for all $n$.

\begin{proposition}\label{AR_t:zigzag_recur}
\[
2A_{n+1}=\sum_{k=0}^{n} {n \choose k} A_k A_{n-k} \qquad (n \ge 1)
\]
with $A_0 = A_1 = 1$.
\end{proposition}
\begin{proof}
In a SYT of the required shape and size $n+1$, the cell containing
$n+1$ must be the last in its row and column. Removing this cell
leaves at most two path-connected components, with the
western/southern one necessarily of odd size (for $n \ge 1$). It
follows that
\[
A_{n+1}=\sum_{k=0 \atop k \hbox{\scriptsize\ odd}}^{n} {n \choose
k} A_k A_{n-k} \qquad (n \ge 1).
\]
Applying a similar argument to the cell containing $1$ gives
\[
A_{n+1}=\sum_{k=0 \atop k \hbox{\scriptsize\ even}}^{n} {n \choose
k} A_k A_{n-k} \qquad (n \ge 0),
\]
and adding the two formulas gives the required recursion.

\end{proof}

Indeed, the recursion for $A_n$
(Proposition~\ref{AR_t:zigzag_recur}) can be seen to be equivalent
to the generating function (Proposition~\ref{AR_t:andre}) since
$f(x) = \sec x + \tan x$ satisfies the differential equation
\[
2 f'(x) = 1 + f(x)^2
\]
with $f(0) = 1$.

Proposition~\ref{AR.t.zigzag} thus gives the determinantal
formulas
\[
(-1)^n E_{2n} = (2n)! \cdot \det \left( \begin{array}{cccccccc}
\frac{1}{2!} & \frac{1}{4!}  & \frac{1}{6!} & . & . & \frac{1}{(2n-4)!} & \frac{1}{(2n-2)!}  & \frac{1}{(2n)!}  \\
1 & \frac{1}{2!}  & \frac{1}{4!} & . & . & . & \frac{1}{(2n-4)!} & \frac{1}{(2n-2)!} \\
0& 1 & \frac{1}{2!}& . & . & . & . & \frac{1}{(2n-4)!} \\
\vdots & \vdots &\vdots & . & . & . & \vdots & \vdots \\
0 & 0& 0& . & . & . &  1 & \frac{1}{2!}
\end{array}\right)
\]
and
\[
T_{n} = (2n-1)! \cdot \det \left( \begin{array}{cccccccc}
\frac{1}{2!} & \frac{1}{4!}  & \frac{1}{6!} & . & . & \frac{1}{(2n-4)!} & \frac{1}{(2n-2)!}  & \frac{1}{(2n-1)!}  \\
1 & \frac{1}{2!}  & \frac{1}{4!} & . & . & . & \frac{1}{(2n-4)!} & \frac{1}{(2n-3)!} \\
0& 1 & \frac{1}{2!}& . & . & . & . & \frac{1}{(2n-5)!} \\
\vdots & \vdots &\vdots & . & . & . & \vdots & \vdots \\
0 & 0& 0& . & . & . &  1 & \frac{1}{1!}
\end{array}\right)
\]

\subsection{Skew strips of constant width}

The {\dem basic skew strip of width $m$ and height $n$} is the
skew diagram
\[
D_{m,n} = [(n+m-1, n+m-2, \ldots, m+1, m)/(n-1, n-2, \ldots, 1,
0)].
\]
It has $n$ rows, of length $m$ each, with each row shifted one
cell to the left with respect to the row just above it. For
example, $D_{2,n}$ is $D_{2n}$ from Example~\ref{AR_ex:Andre}
above while
\[
D_{3,5} =
\ydiagram{4 + 3, 3 + 3, 2 + 3, 1 + 3, 3}
\]
The {\dem general skew strip of width $m$ and height $n$} ({\dem
$m$-strip}, for short), $D_{m,n,\la,\mu}$, has arbitrary
partitions $\la$ and $\mu$, each of height at most $k := \lfloor
m/2 \rfloor$, as ``head'' (northeastern tip) and ``tail''
(southwestern tip), respectively, instead of the basic partitions
$\la = \mu = (k, k-1, \ldots, 1)$. For example,
\[
D_{6,7,(4,2,1),(3,3,1)} =
\ytableausetup{baseline} \ydiagram{6 + 7, 5 + 6, 4 + 6, 3 + 6, 2 +
6, 7, 6}
* [\bullet]{9 + 4, 9 + 2, 9 + 1, 0, 2 + 1, 3, 3}
\ytableausetup{nobaseline}
\]
where $m = 6$, $n = 7$, $k = 3$ and the head and tail have marked
cells.

%

The determinantal formula for skew shapes
(Theorem~\ref{t.AR_num_skew_det}) expresses $f^D$ as an explicit
determinant of order $n$, the number of rows. Baryshnikov and
Romik~\cite{B-Romik}, developing an idea of Elkies~\cite{Elkies},
gave an alternative determinant of order $k$, half the length of a
typical row. This is a considerable improvement if $m$, $k$, $\la$
and $\mu$ are fixed while $n$ is allowed to grow.

The general statement needs a bit  of notation. Denote, for a
non-negative integer $n$,
\[
A'_n := \frac{A_n}{n!}, \qquad A''_n := \frac{A'_n}{2^{n+1} - 1},
\qquad A'''_n := \frac{(2^n - 1) A''_n}{2^n},
\]
where $A_n$ are Andr{\' e}'s alternating numbers (as in the
previous subsection); and let
\[
\epsilon(n) := \begin{cases}
(-1)^{n/2}, & \hbox{if } n \hbox{ is even,} \\
0, & \hbox{if } n \hbox{ is odd.}
 \end{cases}
 \]
Define, for nonnegative integers $N$, $p$ and $q$,
\begin{eqnarray*}
X_N^{(0)}(p, q) &:=& \sum_{i = 0}^{\lfloor p/2 \rfloor} \sum_{j =
0}^{\lfloor q/2 \rfloor} \frac{(-1)^{i + j} A'_{N + 2i + 2j +
1}}{(p - 2i)!(q - 2j)!} + \epsilon(p+1) \sum_{j = 0}^{\lfloor q/2
\rfloor}
\frac{(-1)^{j} A'_{N + p + 2j + 1}}{(q - 2j)!} \\
& & +  \epsilon(q+1) \sum_{i = 0}^{\lfloor p/2 \rfloor}
\frac{(-1)^{i} A'_{N + 2i + q + 1}}{(p - 2i)!} + \epsilon(p+1)
\epsilon(q+1) A'_{N + p + q + 1}
\end{eqnarray*}
and
\begin{eqnarray*}
X_N^{(1)}(p, q) &:=& \sum_{i = 0}^{\lfloor p/2 \rfloor} \sum_{j =
0}^{\lfloor q/2 \rfloor} \frac{(-1)^{i + j} A'''_{N + 2i + 2j +
1}}{(p - 2i)!(q - 2j)!} + \epsilon(p) \sum_{j = 0}^{\lfloor q/2
\rfloor}
\frac{(-1)^{j} A''_{N + p + 2j + 1}}{(q - 2j)!} \\
& & +  \epsilon(q) \sum_{i = 0}^{\lfloor p/2 \rfloor}
\frac{(-1)^{i} A''_{N + 2i + q + 1}}{(p - 2i)!} + \epsilon(p)
\epsilon(q) A'''_{N + p + q + 1} \quad .
\end{eqnarray*}

\begin{theorem}\label{AR_t:Romik_4}{\rm \cite[Theorem 4]{B-Romik}}
Let $D = D_{m,n,\la,\mu}$ be an $m$-strip with head and tail
partitions $\la = (\la_1, \ldots, \la_k)$ and $\mu = (\mu_1,
\ldots, \mu_k)$, where $k := \lfloor m/2 \rfloor$. For $(1 \le i
\le k)$ define $L_i := \la_i + k - i$ and $M_i := \mu_i + k - i$,
and denote
\[
m\%2 := \begin{cases}
0, & \hbox{if } m \hbox{ is even;} \\
1, & \hbox{if } m \hbox{ is odd.}
\end{cases}
\]
Then
\[
f^D = (-1)^{k \choose 2} |D|! \cdot \det \left[
X_{2n-m+1}^{(m\%2)}(L_i,M_j) \right]_{i,j=1}^{k}.
\]
\end{theorem}

Note that $X_N^{(\epsilon)}(p,q)$ are linear combinations of
$A_{N+1}, \ldots, A_{N+p+q+1}$, so that $f^D$ is expressed as a
polynomial in the $A_i$ whose complexity depends on the row length
$m$ but not on the number of rows $n$.

The impressive formal definitions of $X_N^{(0)}$ have simple
geometric interpretations:
\[
X_{2n-1}^{(0)}(p,q) = \frac{f^D}{|D|!}
\]
where
\[
D = \zigzag_{2n+p+q}(\{p+2, p+4, \ldots, p+2n-2\}) = \quad
\ydiagram{5 + 5, 4 + 2, 3 + 2, 4} * [\bullet]{7 + 3, 0, 0, 2}
\]
($p$ marked southwestern cells in a row, $2n$ unmarked cells, and
$q$ marked northeastern cells in a row), and
\[
X_{2n}^{(0)}(p,q) = \frac{f^D}{|D|!}
\]
where
\[
D = \zigzag_{2n+p+q+1}(\{p+2, p+4, \ldots, p+2n, p+2n+1, \ldots,
p+2n+q\}) = \quad
\ydiagram{6 + 1, 6 + 1, 6 + 1, 6 + 1, 5 + 2, 4 + 2, 3 + 2, 4}
* [\bullet]{6 + 1, 6 + 1, 6 + 1, 0, 0, 0, 0, 2}
\]
($p$ marked southwestern cells in a row, $2n+1$ unmarked cells,
and $q$ marked northeastern cells in a column). These are
$2$-strips, i.e. zigzag shapes.
It is possible to define $X_{N}^{(1)}(p,q)$ similarly in terms of
$3$-strips, a task left as an exercise to the
reader~\cite{B-Romik}.

Here are some interesting special cases. 

\begin{corollary}\label{AR_t:Romik_3}{\rm \cite[Theorem 1]{B-Romik}}
$3$-strips: 
\[
f^{D_{3,n,(),()}} = \frac{(3n-2)!\, T_n}{(2n-1)!\, 2^{2n-2}},
\]
\[
f^{D_{3,n,(1),()}} = \frac{(3n-1)!\, T_n}{(2n-1)!\, 2^{2n-1}},
\]
\[
f^{D_{3,n}} = f^{D_{3,n,(1),(1)}} = \frac{(3n)!\, (2^{2n-1} - 1)
T_n}{(2n-1)!\, 2^{2n-1} (2^{2n}-1)}.
\]
\end{corollary}
\begin{corollary}{\rm \cite[Theorem 2]{B-Romik}}
$4$-strips: 
\[
f^{D_{4,n,(),()}} = (4n-2)! \left( \frac{T_{n}^2}{(2n-1)!^2} +
\frac{E_{2n-2} E_{2n}}{(2n-2)!(2n)!} \right),
\]
\[
f^{D_{4,n}} = f^{D_{4,n,(1),(1)}} = (4n)! \left(
\frac{E_{2n}^2}{(2n)!^2} - \frac{E_{2n-2}
E_{2n+2}}{(2n-2)!(2n+2)!} \right).
\]
\end{corollary}
\begin{corollary}{\rm \cite[Theorem 3]{B-Romik}}
$5$-strip: 
\[
f^{D_{5,n,(),()}} = \frac{(5n-6)!\, T_{n-1}^2}{(2n-3)!)^2 2^{4n-6}
(2^{2n-2}-1)}.
\]
\end{corollary}

\begin{proof}[Proof of Theorem~\ref{AR_t:Romik_4} (sketch)]
The proof uses transfer operators, following Elkies~\cite{Elkies}.
Elkies considered, essentially, the zigzag shapes ($2$-strips)
$D_n$ from Example~\ref{AR_ex:Andre}, whose SYT correspond to
alternating (up-down) permutations. Recall from
Subsection~\ref{AR_s:order_polytope} the definition of the order
polytope $P(D_n)$, whose volume is, by
Observation~\ref{AR_t:vol_order_polytope},
\[
\vol P(D_n) = \frac{f^{D_n}}{n!}.
\]
This polytope can be written as
\[
P(D_n) = \{ (x_1, \ldots, x_n) \in [0,1]^n \,:\, x_1 \le x_2 \ge
x_3 \le x_4 \ge \ldots \},
\]
and therefore its volume can also be computed by an iterated
integral:
\[
\vol P(D_n) = \int_0^1 dx_1 \int_{x_1}^1 dx_2 \int_0^{x_2} dx_3
\int_{x_3}^1 dx_4 \cdots .
\]
Some manipulations now lead to the expression
\[
\vol P(D_n) = \langle T^{n-1}({\bnum{1}}), {\bnum{1}} \rangle
\]
where ${\bnum{1}} \in L^2[0,1]$ is the function with constant
value $1$, $\langle \cdot, \cdot \rangle$ is the standard inner
product on $L^2[0,1]$, and $T : L^2[0,1] \to L^2[0,1]$ is the
compact self-adjoint operator defined by
\[
(Tf)(x) := \int_0^{1-x} f(y) dy \qquad (\forall f \in L^2[0,1]).
\]
The eigenvalues $\la_k$ and corresponding orthonormal
eigenfunctions $\phi_k$ of $T$ can be computed explicitly, leading
to the explicit formula
\[
\vol P(D_n) = \sum_{k} \la_k^{n-1} \langle {\bnum{1}}, \phi_k
\rangle^2 = \frac{2^{n+2}}{\pi^{n+1}} \sum_{k = - \infty}^{\infty}
\frac{1}{(4k+1)^{n+1}} \qquad (n \ge 1)
\]
which gives a corresponding expression for
\[
A_n = f^{D_n} = n! \vol P(D_n).
\]
Baryshnikov and Romik extended this treatment of a $2$-strip to
general $m$-strips, using additional ingredients. For instance,
the iterated integral for a $3$-strip
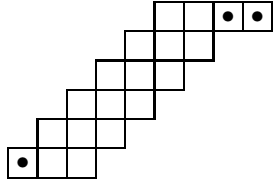
\[
\ydiagram{5 + 4, 4 + 3, 3 + 3, 2 + 3, 1 + 3, 3}
* [\bullet]{7 + 2, 0, 0, 0, 0, 1}
\]
gives
\[
\vol P(D_{3,n,\la,\mu}) = \langle (BA)^{n-1} T_{\mu}({\bnum{1}}),
T_{\la}({\bnum{1}}) \rangle
\]
where $\Omega := \{ (u,v) \in [0,1]^2 \,:\, u \le v\}$, the
transfer operators $A : L^2[0,1] \to L^2(\Omega)$ and $B :
L^2(\Omega) \to L^2[0,1]$ are defined by
\[
(Af)(u,v) := \int_u^v f(x) \,dx
\]
and
\[
(Bg)(x) := \int_0^x \int_x^1 g(u,v) \,dv \,du,
\]
and $T_\la$, $T_\mu$ are operators corresponding to the ``head''
and ``tail'' partitions $\la$ and $\mu$.

\end{proof}

In a slightly different direction, Stanley~\cite{Stanley-skew}
defines $\hD_{a,b,c,n}$ to be the skew shape whose diagram has $n$
rows, with $a$ cells in the top row and $b$ cells in each other
row, and such that each row is shifted $c-1$ columns to the left
with respect to the preceding (higher) row. For example,
\[
\hD_{5,4,3,4} =
\ydiagram{6 + 5, 4 + 4, 2 + 4, 4}
\]

\begin{theorem}{\rm \cite[Corollary 2.5]{Stanley-skew}}
For $a$, $b$ and $c$ with $c \le b < 2c$,
\[
\sum_{n \ge 0} f^{\hD_{a,b,c,n+1}} \frac{x^{n+1}}{(a + nb)!} =
\frac{x \sum_{n \ge 0} \frac{(-x)^n}{(b+nc)!}}%
{(b-c)! - x \sum_{n \ge 0} \frac{(-x)^n}{(a+nc)!}}.
\]
\end{theorem}

Two special cases deserve special attention: $a = b$ and $b = c$.

For $a = b$ all the rows are of the same length.
\begin{corollary}
For $a$ and $c$ with $c \le a < 2c$,
\[
1 + \sum_{n \ge 1} f^{\hD_{a,a,c,n}} \frac{x^n}{(na)!} = \left({1
- \frac{x}{(a-c)!} \sum_{n \ge 0} \frac{(-x)^n}{(a+nc)!}}
\right)^{-1}.
\]
\end{corollary}

In particular, for $a = b = 3$ and $c = 2$, $\hD_{3,3,2,n} =
D_{3,n}$ as in Theorem~\ref{AR_t:Romik_3}:
\[
\sum_{n \ge 0} f^{D_{3,n}} \frac{x^{2n}}{(3n)!} = \left( {\sum_{n
\ge 0} \frac{(-x^2)^n}{(2n+1)!}} \right)^{-1} = \frac{x}{\sin x}.
\]
This result was already known to Gessel and
Viennot~\cite{GesselViennot1989}.

For $b = c$ one obtains a zigzag shape: $\hD_{a,c,c,n+1} =
\zigzag_{a+nc}(S)$ for $S = \{c, 2c, \ldots, nc\}$.
\begin{corollary}
For any positive $a$ and $c$,
\[
\sum_{n \ge 0} f^{\zigzag_{a+nc}( \{c, 2c, \ldots, nc\})}
\frac{x^{n+1}}{(a + nc)!} =
\frac{x \sum_{n \ge 0} \frac{(-x)^n}{(c+nc)!}}%
{1 - x \sum_{n \ge 0} \frac{(-x)^n}{(a+nc)!}}.
\]
\end{corollary}

\section{Truncated and other non-classical shapes}\label{AR_s:formulas_non_classical}

\begin{definition}
A diagram of {\dem truncated shape} is a line-convex diagram
obtained from a diagram of ordinary or shifted shape by deleting
cells from the NE corner (in the English notation, where row
lengths decrease from top to bottom).
\end{definition}

For example, here are diagrams of a truncated ordinary shape
\[
[(4,4,2,1)\setminus (1)] =
\ydiagram{3, 4, 2, 1}
\]
and a truncated shifted shape:
\[
[(4,3,2,1)^*\setminus (1,1)] =
\ydiagram{3, 1 + 2, 2 + 2, 3 + 1}
\]

Modules associated to truncated shapes were introduced and studied
in~\cite{James-Peel, Reiner-Shimozono}. Interest in the
enumeration of SYT of truncated shapes was
recently enhanced by a new interpretation~\cite{AR_tft2}: 
The number of geodesics between distinguished pairs of antipodes
in the flip graph of inner-triangle-free triangulations is twice
the number of SYT of a corresponding truncated shifted staircase
shape. Motivated by this result, extensive computations were
carried out for the number of SYT of these and other truncated
shapes. It was found that, in some cases, these numbers are
unusually ``smooth'', i.e., all their prime factors are relatively
very small. This makes it reasonable to expect a product formula.
Subsequently, such formulas were conjectured and proved for
rectangular and shifted staircase shapes truncated by a square, or
nearly a square, and for rectangular shapes truncated by a
staircase; see~\cite{AKR, Panova, Sun1, Sun2}.

\subsection{Truncated shifted staircase shape}%
\label{sec:staircase}

In this subsection, $\la = (\la_1, \ldots, \la_t)$ (with $\la_1 >
\ldots > \la_t > 0$) will be a strict partition,
with $g^{\la}$ denoting the number of SYT of shifted shape $\la$.

For any nonnegative integer $n$, let $\delta_n := (n, n-1, \ldots,
1)$ be the corresponding shifted staircase shape. By Schur's
formula for shifted shapes (Theorem~\ref{t.AR_num_shifted_prod}),

\begin{corollary}\label{t.shifted_staircase}
The number of SYT of shifted staircase shape $\delta_n$ is
\[
g^{\delta_n} = N! \cdot \prod_{i=0}^{n-1} \frac{i!}{(2i+1)!},
\]
where $N := |\delta_n| = {n+1 \choose 2}$.
\end{corollary}

\bigskip

The following enumeration problem was actually the original
motivation for the study of truncated shapes, because of its
combinatorial interpretation, as explained in~\cite{AR_tft2}.

\begin{theorem}\label{tr_thm1}{\rm \cite[Corollary 4.8]{AKR}\cite[Theorem 1]{Panova}}
The number of SYT of truncated shifted staircase shape
$\delta_{n}\setminus (1)$ is equal to
\[
g^{\delta_{n}}\frac{C_{n} C_{n - 2}}{2\, C_{2n - 3}},
\]
where $C_n=\frac{1}{n+1}{2n\choose n}$ is the $n$-th Catalan
number.
\end{theorem}

\begin{example}\label{AR_ex:SYT_truncated}{
There are $g^{\delta_4} = 12$ SYT of shape $\delta_4$, but only
$4$ SYT of truncated shape $\delta_4\setminus (1)$:
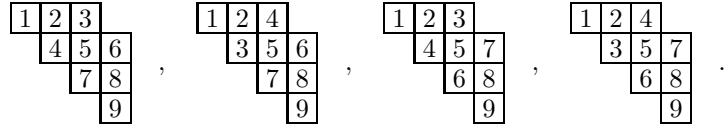
\[
\ytableaushort{123, \none456, \none\none78, \none\none\none9}
\quad , \quad
\ytableaushort{124, \none356, \none\none78, \none\none\none9}
\quad , \quad
\ytableaushort{123, \none457, \none\none68, \none\none\none9}
\quad , \quad
\ytableaushort{124, \none357, \none\none68, \none\none\none9}
\quad .
\]
}
\end{example}


Theorem~\ref{tr_thm1} may be generalized to a truncation of a
$(k-1) \times (k-1)$ square from the NE corner of a shifted
staircase shape $\delta_{m+2k}$.


\begin{example} For $m=1$ and $k=3$, the truncated shape is
\[
[\delta_5 \setminus (2^2)] =
\ydiagram{3, 1 + 2, 2 + 3, 3 + 2, 4 + 1}
\]
\end{example}

\medskip

\begin{theorem}\label{t.stair_minus_sq}{\rm \cite[Corollary 4.8]{AKR}}
The number of SYT of truncated shifted staircase shape
$\delta_{m+2k} \setminus ((k-1)^{k-1})$ is
$$
g^{(m+k+1,\ldots,m+3,m+1,\ldots,1)} g^{(m+k+1,\ldots,m+3,m+1)}
\cdot \frac{N! M!}{(N - M - 1)!(2M + 1)!},
$$
where $N = {m+2k+1 \choose 2} - (k-1)^2$ is the size of the shape
and $M = k(2m+k+3)/2 - 1$.
\end{theorem}

\medskip

Similar results were obtained in~\cite{AKR} for truncation by
``almost squares'', namely by $\kappa=(k^{k-1},k-1)$.

\subsection{Truncated rectangular shapes}\label{sec:rectangular}

In this section, $\la = (\la_1, \ldots, \la_m)$ (with $\la_1 \ge
\ldots \ge \la_m \ge 0$) will be a partition with (at most) $m$
parts.
Two partitions which differ only in trailing zeros will be
considered equal.

For any nonnegative integers $m$ and $n$, let $(n^m) :=
(n,\ldots,n)$ ($m$ times) be the corresponding rectangular shape.
The Frobenius-Young formula (Theorem~\ref{t.AR_num_ordinary_prod})
implies the following.

\begin{observation}\label{t.rectangle}
The number of SYT of rectangular shape $(n^m)$ is
\[
f^{(n^m)} = (mn)! \cdot \frac{F_m F_n}{F_{m+n}},
\]
where
\[
F_m := \prod_{i=0}^{m-1} i!.
\]
\end{observation}

Consider truncating a $(k-1) \times (k-1)$ square from the NE
corner of a rectangular shape $((n+k-1)^{m+k-1})$.

\begin{example}
Let $n=3$, $m=2$ and $k=3$. Then
\[
[(5^4) \setminus (2^2)] =
\ydiagram{3, 3, 5, 5}
\]
\end{example}

\begin{theorem}\label{t.rect_minus_sq}{\rm \cite[Corollary 5.7]{AKR}}
The number of SYT of truncated rectangular shape
$((n+k-1)^{m+k-1}) \setminus ((k-1)^{k-1})$ is
\[
\frac{N!(mk-1)!(nk-1)!(m+n-1)!k}{(mk+nk-1)!} \cdot \frac{F_{m-1}
F_{n-1} F_{k-1}}{F_{m+n+k-1}},
\]
where $N$
is the size of the shape and $F_n$ is as in
Observation~\ref{t.rectangle}.
\end{theorem}

In particular,

\begin{corollary}\label{cor7}
The number of SYT of truncated rectangular shape $((n+1)^{m+1})
\setminus (1)$ is
\[
\frac{N!(2m-1)!(2n-1)! \cdot 2}{(2m+2n-1)!(m+n+2)} \cdot
\frac{F_{m-1} F_{n-1}}{F_{m+n+1}},
\]
where $N = (m+1)(n+1) - 1$ is the size of the shape and $F_n$ is
as in Observation~\ref{t.rectangle}.
\end{corollary}


Similar results were obtained in~\cite{AKR, Panova} for truncation
by almost squares $\kappa=(k^{k-1},k-1)$.

\medskip

Not much is known for truncation by rectangles.  The following
formula was conjectured in~\cite{AKR} and proved by
Sun~\cite{Sun_conjecture} using complex integrals.

\begin{proposition}{\rm \cite{Sun_conjecture}}  For $n\ge 2$
\[
f^{(n^n)\setminus (2)} = 
\frac{ (n^2-2)! (3n-4)!^2 \cdot 6 }{ (6n-8)! (2n-2)! (n-2)!^2 } 
\cdot \frac{F_{n-2}^2}{F_{2n-4}},
\]
where $F_n$ is as in Observation~\ref{t.rectangle}.
\end{proposition}

The following result was proved by Snow~\cite{Snow}.

\begin{proposition}{\rm \cite{Snow}}  For $n\ge 2$ and $k \ge 0$
\[
f^{(n^{k+1}) \setminus (n-2)} = 
\frac{ (kn-k)! (kn+n)! }{ (kn+n-k)! } 
\cdot \frac{F_k F_n}{F_{n+k}},
\]
where $F_n$ is as in Observation~\ref{t.rectangle}.
\end{proposition}

\medskip

A different method to derive product formulas, for other families
of truncated shapes, has been developed by Panova~\cite{Panova}.
Consider a rectangular shape truncated by a staircase shape.

\begin{example}
\[
[(4^5) \setminus \delta_2] =
\ydiagram{2, 3, 4, 4, 4}
\]
\end{example}

\begin{theorem}{\rm \cite[Theorem 2]{Panova}}\
Let $m\ge n\ge k$ be positive integers. The number of SYT of
truncated shape $(n^m)\setminus \delta_k$ is
\[
{N\choose m(n-k-1)}f^{(n-k-1)^m}g^{(m, m-1, \ldots, m-k)}
\frac{E(k+1,m,n-k-1)}{E(k+1,m,0)},
\]
where $N=mn - {k+1\choose 2}$ is the size of the shape and
\[
E(r,p,s)=\begin{cases} \prod\limits_{r<l<2p-r+2}
\frac{1}{(l+2s)^{r/2}} \prod\limits_{2\le l\le r}
\frac{1}{((l+2s)(2p-l+2s+2))^{\lfloor l/2\rfloor}}
, & \text{if $r$ is even}; \\
\frac{((r-1)/2+s)!}{(p-(r-1)/2 +s)!}E(r-1,p,s) , & \text{if $r$ is
odd}.
\end{cases}
\]
\end{theorem}


\subsection{Other truncated shapes}

The following elegant  
result regarding {\dem shifted strips} was recently proved by Sun.

\begin{theorem}~{\rm \cite[\S 4.2]{Sun3}}\label{Sun-lozenge}
The number of SYT of truncated shifted shape with $n$ rows and $4$
cells in each row
\[
\ydiagram{4, 1 + 4, 2 + 4, 3 + 4, 4 + 4}
\]
is the $(2n-1)$-st Pell number~{\rm \cite[A000129]{oeis}}
\[
\frac{1}{2\sqrt 2} \left( (1+\sqrt 2)^{2n-1}-(1-\sqrt 2)^{2n-1}
\right).
\]
\end{theorem}






Sun applied a probabilistic version of computations of volumes of
order polytopes to enumerate SYT of truncated and other exotic
shapes. In~\cite{Sun2} he obtained product formulas for the number
of SYT of certain truncated skew shapes. This includes the shape
$((n+k)^{r+1}, n^{m-1})/(n-1)^r$ truncated by a rectangle or an
``almost rectangle'', the truncated shape $((n+1)^3, n^{m-2}) /
(n-2) \setminus (2^2)$, and the truncated shape $(n + 1)^2 /
n^{m-2} \setminus (2)$.

%
%


\medskip

Modules associated with non-line-convex shapes were considered
in~\cite{James-Peel}. The enumeration of SYT of such shapes is a
very recent subject of study. Special non-line-convex shapes with
one box removed at the end or middle of a row were considered
in~\cite{Sun1}. For example,

\begin{proposition}{\rm \cite[Theorem 5.2]{Sun1}}
For $m\ge 0$, the number of SYT of shape $(m + 3, 3, 3)$ with
middle box in the second row removed, is
\[
\frac{m+5}{10}{m+2\choose 2}{m+9\choose 2}.
\]
\end{proposition}

There are very few known results in this direction; problems in
this area are wide open.

\subsection{Proof approaches for truncated shapes}

Different approaches were applied to prove the above product
formulas. We will sketch one method and remark on another.


\bigskip

The {\dem pivoting approach} of~\cite{AKR} is based on a
combination of two different bijections from SYT to pairs of
smaller SYT:

\begin{itemize}
\item[(i)] Choose a pivot cell $P$ in the NE boundary of a
truncated shape $\zeta$ and subdivide the entries of a given SYT
$T$ into those that are less than the entry of $P$ and those that
are greater. \item[(ii)] Choose a letter $t$ and subdivide the
entries in a SYT $T$ into those that are less than or equal to $t$
and those that are greater than $t$.
\end{itemize}

Proofs are obtained by combining applications of the first
bijection to truncated shapes and the second to corresponding
non-truncated ones. Here is a typical example.

\medskip

\begin{proof}[Proof of Theorem~\ref{tr_thm1} (sketch)]

First, apply the second bijection to a SYT of a shifted staircase
shape.

Let $n$ and $t$ be nonnegative integers, with $t \le {n+1 \choose
2}$. Let $T$ be a SYT of shifted staircase shape $\delta_n$, let
$T_1$ be the set of all cells in $T$ with values at most $t$, and
let $T_2$ be obtained from $T \setminus T_1$ by transposing the
shape (reflecting in an anti-diagonal) and replacing each entry
$i$ by $N-i+1$, where $N = |\delta_n| = {n+1 \choose 2}$. Clearly
$T_1$ and $T_2$ are shifted SYT.

Here is an example with $n=4$ and $t=5$.
\[
\ytableaushort{1246, \none358, \none\none79, \none\none\none{10}}
\, \to \, \left(\,
\ytableaushort{124, \none35} \;,\;
\ytableaushort{\none6, \none8, 79, \none{10}} \,\right)\, \to \,
\left(\,
\ytableaushort{124, \none35} \;,\;
\ytableaushort{{10}986, \none7} \,\right)\, \to \, \left(\,
\ytableaushort{124, \none35} \;,\;
\ytableaushort{1235, \none4} \,\right).
\]
Notice that, treating strict partitions as sets,
$\delta_4=(4,3,2,1)$ is the disjoint union of $\sh(T_1)=(3,2)$ and
$\sh(T_2)=(4,1)$. This is not a coincidence.

\medskip

\noindent{\bf Claim.} {\em Treating strict partitions as sets,
$\delta_n$ is the disjoint union of the shape of $T_1$ and the
shape of $T_2$. }
\smallskip

In order to prove the claim notice that the borderline between
$T_1$ and $T \setminus T_1$ is a lattice path of length exactly
$n$, starting at the NE corner of the staircase shape $\delta_n$,
and using only S and W steps, and ending at the SW boundary.
If the first step is S then the first part of $\sh(T_1)$ is $n$,
and the rest corresponds to a lattice path in $\delta_{n-1}$.
Similarly, if the first step is W then the first part of
$\sh(T_2)$ is $n$, and the rest corresponds to a lattice path in
$\delta_{n-1}$. Thus exactly one of the shapes of $T_1$ and $T_2$
has a part equal to $n$. The claim follows by induction on $n$.

\medskip

We deduce that, for any nonnegative integers $n$ and $t$ with $t
\le {n+1 \choose 2}$,
\begin{equation}\label{tr_eq2}
\sum_{\la \subseteq \delta_n \atop |\la|=t} g^{\la}g^{\la^c} =
g^{\delta_n}.
\end{equation}
Here summation is over all strict partitions $\la$ with the
prescribed restrictions, and $\la^c$ is the complement of $\la$ in
$\delta_n = \{1, \ldots, n\}$ (where strict partitions are treated
as sets). In particular, the LHS is independent of $t$.

\bigskip

Next apply the first bijection on SYT of truncated staircase shape
$\delta_n \setminus (1)$. Choose as a pivot the cell $c = (2,
n-1)$, just SW of the missing NE corner. The entry $t = T(c)$
satisfies $2n - 3 \le t \le {n \choose 2} - 2n + 2$. Let $T$ be a
SYT of shape $\delta_n \setminus (1)$ with entry $t$ in $P$. One
subdivides the other entries of $T$ into those that are (strictly)
less than $t$ and those that are greater than $t$. The entries
less than $t$ constitute $T_1$. To obtain $T_2$, replace each
entry $i > t$ of $T$ by $N - i + 1$, where $N$ is the total number
of entries in $T$, and suitably transpose the resulting array. It
is easy to see that both $T_1$ and $T_2$ are shifted SYT.

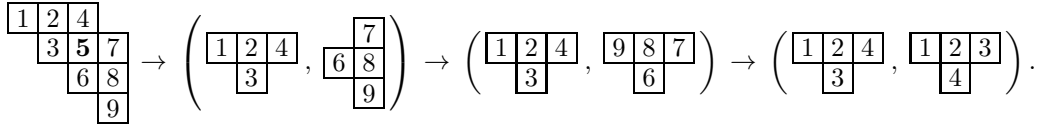
\begin{example}
\[
\ytableaushort{124, \none3{\bnum{5}}7, \none\none68,
\none\none\none9} \, \to \, \left(\,
\ytableaushort{124, \none3} \;,\;
\ytableaushort{\none7, 68, \none9} \,\right)\, \to \, \left(\,
\ytableaushort{124, \none3} \;,\;
\ytableaushort{987, \none6} \,\right)\, \to \, \left(\,
\ytableaushort{124, \none3} \;,\;
\ytableaushort{123, \none4} \,\right).
\]
\end{example}



Next notice that the shape of $T_1$ is $(m-1,m-3)\cup \lambda$
while the shape of $T_2$ is  $(m-1,m-3)\cup \lambda^c$, where
$\lambda$ is a strict partition contained in $\delta_{n-2}$ and
$\lambda^c$ is its complement in $\delta_{n-2}$.



We deduce that

\begin{equation}\label{tr_eq1}
g^{\delta_n \setminus (1)}= \sum_t \sum_{\la \subseteq
\delta_{n-2} \atop |\la|=t} g^{(n-1, n-3) \cup \la} g^{(n-1, n-3)
\cup \la^c}.
\end{equation}

Here summation is over all strict partitions $\la$ with the
prescribed restrictions.

\bigskip

Finally, by Schur's formula (Theorem~\ref{t.AR_num_shifted_prod}),
for any strict partitions $\la$ and $\mu = (\mu_1, \ldots, \mu_k)$
with $\mu_1 > \ldots > \mu_k > m$ and $\la\cup \la^c=\delta_m$,
the following holds.

\begin{equation}\label{tr_eq3}
g^{\mu \cup \la} g^{\mu \cup \la^c} = c(\mu,|\la|,|\la^c|) \cdot
g^{\la} g^{\la^c},
\end{equation}
where
\[
c(\mu,|\la|,|\la^c|) =
\frac{g^{\mu \cup \delta_m} g^{\mu}}{g^{\delta_m}} \cdot
\frac{|\delta_m|!(|\mu|+|\la|)!(|\mu|+|\la^c|)!}%
{(|\mu|+|\delta_m|)!|\mu|!|\la|!|\la^c|!}
\]
depends only on the sizes $|\la|$ and $|\la^c|$ and not on the
actual partitions $\la$ and $\la^c$.

\medskip

Combining Equations (\ref{tr_eq1}), (\ref{tr_eq2}) and
(\ref{tr_eq3}) together with some binomial identities completes
the proof.

\end{proof}

For a detailed proof and applications of the method to other
truncated shapes see~\cite{AKR}.

\bigskip


A different proof was presented by Panova~\cite{Panova}. Panova's
approach is sophisticated and involved
and will just be outlined. The proof relies on a bijection from
SYT of the truncated shape to semi-standard Young tableaux of skew
shapes. This bijection translates the enumeration problem to
evaluations of sums of Schur functions at certain specializations.
These evaluations are then reduced to computations of complex
integrals, which are carried out by a comparison to another
translation of the original enumerative problem to a volume of the
associated order polytope.

\section{Rim hook and domino tableaux}

\subsection{Definitions}

The following concept generalizes the notion of SYT. Recall  from
Subsection~\ref{AR_s:zigzag_sum} the definition of a zigzag shape.

\begin{definition}
Let $r$ and $n$ be positive integers and let $\la \vdash rn$. An
{\dem $r$-rim hook tableau} of shape $\la$ is a filling of the
cells of the diagram $[\la]$ by the letters $1, \ldots, n$ such
that
\begin{enumerate}
\item each letter $i$ fills exactly $r$ cells, which form a zigzag
shape
called the $i$-th {\dem rim hook} (or {\dem border strip}); 
and \item for each $1 \le k \le n$, the union of the $i$-th rim
hooks for $1 \le i \le k$ is a diagram of ordinary shape.
\end{enumerate}
Denote by $f^\la_r$ the number of $r$-rim hook tableaux of shape
$\la \vdash rn$.
\end{definition}

The $n$-th rim hook forms a path-connected subset of the {\dem
rim} (SE boundary) of the diagram $[\la]$, and removing it leads
inductively to a similar description for the other rim hooks.

$1$-rim hook tableaux are ordinary SYT; $2$-rim hook tableaux are
also called {\dem domino tableaux}.
\begin{example}
Here is a domino tableau of shape $(5,5,4)$:
\[
\ytableaushort{11336,24556,2477}
\]
and here is a $3$-rim hook tableau of shape $(5,4,3)$:
\[
\ytableaushort{11333,1244,224}
\]
\end{example}

\medskip


\begin{definition}
An {\dem $r$-partition} of $n$ is a sequence $\la = (\la^0,
\ldots, \la^{r-1})$ of partitions of total size $|\la^0| + \ldots
+ |\la^{r-1}| = n$. The corresponding {\dem $r$-diagram}
$[\la^0,\ldots, \la^{r-1}]$ is the sequence $([\la^0], \ldots,
[\la^{r-1}])$ of ordinary diagrams. It is sometimes drawn as a
skew diagram, with $[\la^{i}]$ lying directly southwest of
$[\la^{i-1}]$ for every $1 \le i \le r-1$.
\end{definition}
\begin{example}
The $2$-diagram of shape $(\la^0, \la^1) = ((3,1), (2,2))$ is
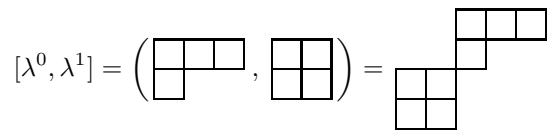
\[
[\la^0, \la^1] = \left(\,
\ydiagram{3,1} \;,\;
\ydiagram{2,2} \,\right) = \,
\ydiagram{2+3,2+1,2,2}
\]
\end{example}

\begin{definition}
A {\dem standard Young $r$-tableau} ({\dem $r$-SYT}) $T =
(T^0,\dots, T^{r-1})$ of shape $\la = (\la^0,\dots, \la^{r-1})$
and total size $n$ is obtained by inserting the integers $1, 2,
\ldots, n$ as entries into the cells of the diagram $[\la]$ such
that the entries increase along rows and columns.
\end{definition}

\subsection{The $r$-quotient and $r$-core}

\begin{definition}
Let $\la$ be a partition, and $D = [\la]$ the corresponding
ordinary diagram. The {\dem boundary sequence} of $\la$ is the
$0/1$ sequence $\partial(\la)$ constructed as follows: Start at
the SW corner of the diagram and proceed along the edges of the SE
boundary up to the NE corner. Each horizontal (east-bound) step is
encoded by $1$, and each vertical (north-bound) step by $0$.
\end{definition}

\begin{example}\label{AR_ex:boundary_map}
\[
\la = (3, 1) \quad\to\quad
D = [\la] = \ydiagram{3, 1} \quad\to\quad
\partial(\la) = (1,0,1,1,0)
\]
\end{example}

The boundary sequences starts with $1$ and ends with $0$ -- unless
$\la$ is the empty partition, for which $\partial(\la)$ is the
empty sequence.

\begin{definition}
The {\dem extended boundary sequence} $\partial_*(\la)$ of $\la$
is the doubly-infinite sequence obtained from $\partial(\la)$ by
prepending to it the sequence $(\ldots, 0, 0)$ and appending to it
the sequence $(1, 1, \ldots)$.
\end{definition}

Geometrically, these additions represent a vertical ray and a
horizontal ray, respectively, so that the tour of the boundary of
$[\la]$ actually ``starts'' at the far south and ``ends'' at the
far east.

\begin{example}
If $\la = (3, 1)$ then $\partial_*(\la) =
(\ldots,0,0,1,0,1,1,0,1,1,\ldots)$, and if $\la$ is empty then
$\partial_*(\la) = (\ldots,0,0,1,1,\ldots)$.
\end{example}

$\partial_*$ is clearly a bijection from the set of all partitions
to the set of all doubly-infinite $0/1$ sequences with initially
only $0$-s and eventually only $1$-s.

\begin{definition}
There is a {\dem natural indexing} of any (extended) boundary
sequence, as follows: The index $k$ of an element is equal to the
number of $1$-s weakly to its left minus the number of $0$-s
strictly to its right.
\end{definition}

\begin{example}
\[
\begin{array}{l}
\text{\rm $0/1$ sequence:} \\
\\
\text{\rm Indexing:}
\end{array}
\begin{array}{rrrrrrrrrrr}
\ldots &0 &0 &1 &0 &1 &1 &0 &1 &1 &\ldots \\
&\uparrow &\uparrow &\uparrow &\uparrow &\uparrow
&\uparrow &\uparrow &\uparrow &\uparrow & \\
\ldots &-3 &-2 &-1 &\,0 &\,1 &\,2 &\,3 &\,4 &\,5 &\ldots
\end{array}
\]
\end{example}

\begin{definition} Let $\la$ be a partition, $r$ a positive integer, and $s := \partial_*(\la)$.
\begin{enumerate}
\item The {\dem $r$-quotient} $q_r(\la)$ is a sequence of $r$
partitions obtained as follows: For each $0 \le i \le r-1$ let
$s^i$ be the subsequence of $s$ corresponding to the indices which
are congruent to $i \pmod r$, and let $\la^i :=
\partial_*^{-1}(s^i)$. Then $q_r(\la) := (\la^0, \ldots,
\la^{r-1})$. \item The {\dem $r$-core} (or {\dem $r$-residue})
$c_r(\la)$ is the partition $\la' = \partial_*^{-1}(s')$, where
$s'$ is obtained from $s$ by a sequence of moves which interchange
a $1$ in position $i$ with a $0$ in position $i + r$ (for some
$i$), as long as such a move is still possible.
\end{enumerate}
\end{definition}

Denote $|q_r(\la)| := |\la^0| + \ldots + |\la^{r-1}|$.
\begin{theorem}
$|\la| = r \cdot |q_r(\la)| + |c_r(\la)|.$
\end{theorem}

\begin{example}
For $\la = (6,4,2,2,2,1)$ and $r = 2$,
\[
s = \partial_*(\la) = (\ldots, 0, 0, 0, 1, 0, 1, 0, 0, \hat{0}, 1,
1, 0, 1, 1, 0, 1, 1, 1, \ldots)
\]
with a hat over the entry indexed $0$. It follows that
\[
s^0 = (\ldots, 0, 0, 0, 0, 0, 1, 1, 0, 1, \ldots) \quad \text{and}
\quad s^1 = (\ldots, 0, 1, 1, 0, 1, 0, 1, 1, 1, \ldots).
\]
The $2$-quotient is therefore $q_2(\la) = ((2), (3, 2))$.
The $2$-core is
\[
c_2(\la) = \partial_*^{-1}(s') = \partial_*^{-1}(\ldots, 0, 0, 0,
0, 0, 0, 0, 1, \hat{0}, 1, 0, 1, 1, 1, 1, 1, 1, 1, \ldots) = (2,
1).
\]
Indeed, $|\la| = 17 = 2 \cdot 7 + 3 = 2 \cdot |q_2(\la)| +
|c_2(\la)|$.
\end{example}
It is easy to see that, in this example, there are no $2$-rim hook
tableaux of shape $\la$.

\begin{theorem}
$f_r^{\la} \ne 0 \iff \text{\rm the } r \text{\rm-core } c_r(\la)
\text{\rm\ is empty.}$
\end{theorem}

\begin{example}
Let $\la=(4,2)$, $n=3$ and $r=2$. Then
\[
s = \partial_*(\la) = (\ldots, 0, 0, 0, 1, \hat{1}, 0, 1, 1, 0, 1,
1, 1, \ldots),
\]
so that the $2$-core
\[
s' = \partial_*^{-1}(\ldots, 0, 0, 0, 0, \hat{0}, 1, 1, 1, 1, 1,
1, 1, \ldots)
\]
is empty and the $2$-quotient is
\[
q_2(\la) =
(\partial_*^{-1}(\ldots,0,0,1,1,0,1,\ldots),\partial_*^{-1}(\ldots,0,1,0,1,1,1,\ldots))
= ((2),(1)).
\]
Of course, here $|\la| = 6 = 2 \cdot 3 = r \cdot |q_2(\la)|$.
\end{example}

In this example there are three domino tableaux of shape $(4,2)$,
and also three $2$-SYT of shape $((2),(1))$:
\[
\ytableaushort{1122,33} \quad , \quad
\ytableaushort{1133,22} \quad ,\quad
\ytableaushort{1233,12} \quad \longleftrightarrow \quad
\ytableaushort{\none12,3} \quad , \quad
\ytableaushort{\none13,2} \quad , \quad
\ytableaushort{\none23,1}
\]
This is not a coincidence, as the following theorem shows.

\begin{theorem}\label{AR_t:r_quotient} 
Let $\la$ be a partition with empty $r$-core,
and let $q_r(\la)$ 
be its $r$-quotient. Then
\[
f_r^{\la} = f^{q_r(\la)}.
\]
\end{theorem}


Theorem~\ref{AR_t:r_quotient} may be combined with the hook length
formula for ordinary shapes (Theorem~\ref{t.AR_num_ordinary_hook})
to obtain the following.

\begin{theorem}\label{r_hook}{\rm \cite[p. 84]{JK}}
If $f^\la_r\ne 0$ then
\[
f^\la_r=\frac{(|\la|/r)!}{\prod_{c \in [\la]:\, r | h_{c}}
{h_{c}/{r}}}.
\]
\end{theorem}

\begin{proof}
$\la$ has an empty $r$-core; let $q_r(\la) = (\la^0, \ldots,
\la^{r-1})$ be its $r$-quotient. By the hook length formula for
ordinary shapes (Theorem~\ref{t.AR_num_ordinary_hook}) together
with Observation~\ref{AR_t:obs1},
\[
f^{(\la^0, \ldots, \la^{r-1})} = {|\la|/r \choose
{|\la^0|,\ldots,|\la^{r-1}|}} \prod_{i=0}^{r-1} f^{\la^i}
=\frac{(|\la|/r)!}{\prod_{c \in [\la^0, \ldots, \la^{r-1}]}
h_{c}}.
\]
A careful examination of the $r$-quotient implies that it induces
a bijection from the cells in $\la$ with hook length divisible by
$r$ to all the cells in $(\la^0, \ldots, \la^{r-1})$, such that
every cell $c \in [\la^0, \ldots, \la^{r-1}]$ with hook length
$h_c$ corresponds to a cell $c' \in [\la]$ with hook length
$h_{c'}= r h_{c}$. This completes the proof.

\end{proof}



Stanton and White~\cite{SW} generalized the RS correspondence to a
bijection from $r$-colored permutations (i.e., elements of the
wreath product $\bbz_r\wr \Sc_n$) to pairs of $r$-rim hook
tableaux of the same shape.
The Stanton-White bijection, together with
Theorem~\ref{AR_t:r_quotient}, implies the following
generalization of Corollary~\ref{AR_t:RSK_cor1}.

\begin{theorem}\label{sum_r}
\[
\sum\limits_{\la\vdash rn} (f_r^\la)^2= r^n n! \leqno(1)
\]
%
%
%
and
\[
\sum\limits_{\la\vdash rn} f_r^\la = \sum\limits_{k=0}^{\lfloor
n/2\rfloor} {n\choose 2k} (2k-1)!! r^{n-k}. \leqno(2)
\]
\end{theorem}


In particular, the total number of domino tableaux of size $2n$ is
equal to number of involutions in the hyperoctahedral group $B_n$.
By similar arguments, the number of SYT of size $n$ and unordered
$2$-partition shape is equal to the number of involutions in the
Weyl group $D_n$.

An important inequality for the number of rim hook tableaux has
been found by Fomin and Lulov.
\begin{theorem}{\rm \cite{Fomin_Lulov}}
For any $\la \vdash rn$,
\[
f_r^\la  \le ≤ r^n n! \left( \frac{f^\la}{(rn)!} \right)^{1/r}.
\]
\end{theorem}
See also~\cite{Lulov_Pak}\cite{Roichman}\cite{Larsen_Shalev}.


\section{$q$-Enumeration}\label{AR_s:q}


This section deals primarily with three classical combinatorial
parameters -- inversion number, descent number and major index.
These parameters were originally studied in the context of
permutations (and, more generally, words).
The major index, for example, was introduced by
MacMahon~\cite{MacMahon_book}. These permutation statistics were
studied extensively by Foata and Sch\"utzenberger~\cite{Foata,
FS}, Garsia and Gessel~\cite{GarsiaGessel}, and others. Only later
were these concepts defined and studied for standard Young
tableaux.


\subsection{Permutation statistics}

We start with definitions of the main permutation statistics.
\begin{definition}\
The {\dem descent set} of a permutation $\pi\in \Sc_n$ is
\[
\Des(\pi):=\{i:\ \pi(i)>\pi(i+1)\},
\]
the {\dem descent number} of $\pi$ is
\[
\des(\pi) := |\Des(\pi)|,
\]
and the {\dem major index} of $\pi$ is
\[
\maj(\pi) := \sum\limits_{i\in \Des(\pi)} i.
\]
The {\dem inversion set} of $\pi$ is
\[
\Inv(\pi):=\{(i,j):\ 1\le i<j \le n,\, \pi(i)>\pi(j)\},
\]
and the {\dem inversion number} of $\pi$ is
\[
\inv(\pi) := |\Inv(\pi)|.
\]
\end{definition}

We also use standard $q$-notation:
For a nonnegative integer $n$ and nonnegative integers $k_1,
\ldots, k_t$ summing up to $n$ denote
\[
[n]_q := \frac{q^n-1}{q-1}, \quad [n]_q! := \prod_{i=1}^n [i]_q
\quad \text{(including } [0]_q! := 1), \quad \left[n \atop {k_1,
\ldots, k_t} \right]_q := \frac{[n]_q!}{\prod_{i=1}^t [k_i]_q!}.
\]

\begin{theorem}\label{AR_t:MM}{\rm \cite{MacMahon_book}
(MacMahon's fundamental equidistribution theorem)} For every
positive integer $n$
\[
\sum\limits_{\pi\in \Sc_n} q^{\maj(\pi)}= \sum\limits_{\pi\in
\Sc_n} q^{\inv(\pi)}=[n]_q!
\]
\end{theorem}

A bijective proof was given in the classical paper of
Foata~\cite{Foata}.
Refinements and generalizations were given by many. In particular,
Foata's bijection was applied to show that the major index and
inversion number are equidistributed over inverse descent
classes~\cite{FS}. A different approach was suggested by Garsia
and Gessel, who proved the following.

\begin{theorem}\label{GG}{\rm \cite{GarsiaGessel}}
For every subset $S=\{k_1,\dots,k_t\}\subseteq [n-1]$
\[
\sum\limits_{\pi\in \Sc_n \atop \Des(\pi^{-1})\subseteq S}
q^{\maj(\pi)} = \sum\limits_{\pi\in \Sc_n \atop
\Des(\pi^{-1})\subseteq S} q^{\inv(\pi)} =\left[n \atop
k_1,k_2-k_1,k_3-k_2,\dots,n-k_t \right]_q.
\]
\end{theorem}

The following determinantal formula~\cite[Example
2.2.5]{Stanley_EC1} follows by the inclusion-exclusion principle.

\begin{eqnarray*}\label{des_class_equi}
\sum_{\pi \in \Sc_n \atop \Des(\pi^{-1}) = S} q^{\maj(\pi)} =
\sum_{\pi \in \Sc_n \atop \Des(\pi^{-1}) = S} q^{\inv(\pi)}
&=& [n]!_q \det \left( \frac{1}{[s_{j+1} - s_i]!_q} \right)_{i, j = 0}^{k} \\
&=& \det \left( \left[ n - s_i \atop s_{j+1} - s_i \right]_q
\right)_{i, j = 0}^{k} .
\end{eqnarray*}

\subsection{Statistics on tableaux}


We start with definitions of descent statistics for SYT. Let $T$
be a standard Young tableau of shape $D$ and size $n = |D|$. For
each entry $1 \le t \le n$ let $\row(T^{-1}(t))$ denote the index
of the row containing the cell $T^{-1}(t)$.

\begin{definition}\
The {\dem descent set} of $T$ is
\[
\Des(T) := \{1 \le i  \le n-1\,|\, \row(T^{-1}(i)) <
\row(T^{-1}(i+1))
\},
\]
the {\dem descent number} of $T$ is
\[
\des(T) := |\Des(T)|,
\]
and the {\dem major index} of $T$ is
\[
\maj(T) := \sum\limits_{i\in \Des(T)} i.
\]
\end{definition}

\medskip

\begin{example}\label{AR_ex:des_maj_T}
Let
\[
T = \,
\ytableaushort{1258, 346, 7} \quad .
\]
Then $\Des(T)= \{2,5,6\}$, $\des(T)= 3$ and $\maj(T)=2+5+6=13$.
\end{example}

For a permutation $\pi\in \Sc_n$, recall from
Section~\ref{AR_s:JdT} the notation $T_\pi$ for the skew
(anti-diagonal) SYT which corresponds to $\pi$ and the notation
$(P_\pi,Q_\pi)$ for the pair of ordinary SYT which corresponds to
$\pi$ under the Robinson-Schensted correspondence. By definition,
\[
\Des(T_\pi)=\Des(\pi^{-1}).
\]
The jeu de taquin algorithm preserves the descent set of a SYT,
and therefore
\begin{proposition}\label{AR_t:RSK_Des}
For every  permutation $\pi\in \Sc_n$,
\[
\Des(P_\pi)=\Des(\pi^{-1}) \qquad \hbox{and} \qquad
\Des(Q_\pi)=\Des(\pi).
\]
\end{proposition}



When it comes to inversion number, there is more than one possible
definition for SYT.


\begin{definition}
An {\dem inversion} in $T$ is a pair $(i, j)$ such that $1 \le i <
j \le n$ and the entry for $j$ appears strictly south and strictly
west of the entry for $i$:
\[
\row(T^{-1}(i)) < \row(T^{-1}(j)) \qquad \hbox{and} \qquad
\col(T^{-1}(i)) > \col(T^{-1}(j)).
\]
The {\dem inversion set} of $T$, $\Inv(T)$, consists of all the
inversions in $T$, and the {\dem inversion number} of $T$ is
\[
\inv(T) := |\Inv(T)|.
\]
The {\dem sign} of $T$ is
\[
\sgn(T) := (-1)^{\inv(T)}.
\]
\end{definition}

\begin{definition}
A {\dem weak inversion} in $T$ is a pair $(i, j)$ such that $1 \le
i < j \le n$ and the entry for $j$ appears strictly south and
weakly west of the entry for $i$:
\[
\row(T^{-1}(i)) < \row(T^{-1}(j)) \qquad \hbox{and} \qquad
\col(T^{-1}(i)) \ge \col(T^{-1}(j)).
\]
The {\dem weak inversion set} of $T$, $\Winv(T)$, consists of all
the weak inversions in $T$, and the {\dem weak inversion number}
of $T$ is
\[
\winv(T) := |\Winv(T)|.
\]
\end{definition}

\begin{observation}
For every standard Young tableaux $T$ of ordinary shape $\lambda$,
\[
\winv(T)=\inv(T)+\sum\limits_j \binom{\lambda_j'}{2}.
\]
Here $\lambda_j'$ is the length of the $j$-th column of the
diagram $[\la]$.
\end{observation}

\begin{example}
For $T$ as in Example~\ref{AR_ex:des_maj_T},
$\Inv(T) = \{(2,3),(2,7),(4,7),(5,7),(6,7)\}$, $\inv(T) = 5$
and $\sgn(T) = -1$. The weak inversion set consist of the
inversion set plus all pairs of entries in the same column. Thus
$\winv(T) = \inv(T) + {3 \choose 2} + {2 \choose 2} + {2 \choose
2} + {1 \choose 2} = 4+3+1+1 = 9$.
\end{example}

\bigskip

For another (more complicated) inversion number on SYT
see~\cite{Haglund}.

\subsection{Thin shapes}

We begin with refinements and $q$-analogues of results from
Section~\ref{AR_s:thin}.



\subsubsection{Hook shapes} 

It is easy to verify that

\begin{observation}
For any $1 \le k \le n-1$
\[
\sum\limits_{\sh(T)= (n-k,1^k)}{\bf x}^{\Des(T)}=e_k,
\]
where
\[
e_k:=\sum\limits_{1\le i_1<i_2<\ldots<i_k< n} x_{i_1}\cdots
x_{i_k}
\]
are the elementary symmetric functions.
\end{observation}

\begin{proof}
Let $T$ be a SYT of hook shape. Then for every $1\le i< n$, $i$ is
a descent in $T$ if and only if the letter $i+1$ lies in the first
column.
\end{proof}

Thus
\begin{equation}\label{Des_hook}
\sum\limits_{\sh(T)= \text{hook of size $n$}}{\bf
x}^{\Des(T)}=\prod\limits_{i=1}^{n-1}(1+x_i).
\end{equation}
It follows that

\begin{equation}\label{des_maj_hook}
\sum\limits_{\sh(T)= \text{hook of size
$n$}}t^{\des(T)}q^{\maj(T)} =\prod\limits_{i=1}^{n-1}(1+tq^i).
\end{equation}


Notice that for a $T$ of hook shape and size $n$,
$\sh(T)=(n-k,1^k)$ if and only if $\des(T)=k$. Combining this
observation with Equation (\ref{des_maj_hook}), the $q$-binomial
theorem implies that
\[
\sum\limits_{\sh(T)= (n-k,1^k)}q^{\maj(T)}= q^{{k\choose
2}}\left[n-1\atop k\right]_q.
\]

Finally, 
the statistics $\winv$ and $\maj$ are equal on all SYT of hook
shape.
For most Young tableaux of non-hook zigzag shapes these statistics
are not equal.  However, the equidistribution phenomenon may be
generalized to all zigzags. This will be shown below.

\subsubsection{Zigzag shapes}

Recall from Subsection~\ref{AR_s:zigzag} that
each subset $S \subseteq [n-1]$ defines a zigzag shape
$\zigzag_n(S)$ of size $n$. The following statement refines
Proposition~\ref{total-zigzag}.

\begin{proposition}\label{t.zigzag-descent class}
For any $S \subseteq [n-1]$,
\[
f^{\zigzag_n(S)}=\#\{\pi\in \Sc_n:\ \Des(\pi)=S\}.
\]
\end{proposition}

\begin{proof} Standard Young tableaux of the zigzag shape encoded by $S$
are in bijection with permutations in $S_n$ whose descent set is
exactly $S$. The bijection converts such a tableau into a
permutation by reading the cell entries starting from the
southwestern corner. For example,
\[
T = \,
\ytableaushort{\none\none268, \none\none5, \none37, 14, 9}
\,\mapsto\, \pi=[914375268].
\]
\end{proof}

\medskip

Notice that in this example, $\pi^{-1}=[274368591]$ and
$\Des(T)=\Des(\pi^{-1})=\{2,3,6,8\}$. Also, $\Winv(T)=\Inv(\pi)$.
This is a general phenomenon. Indeed,

\begin{observation}\label{zigzag_Des_Winv}
Let $T$ be a SYT of zigzag shape and let $\pi$ be its image under
the bijection described in the proof of
Proposition~\ref{t.zigzag-descent class}. Then
\[
\Des(T)=\Des(\pi^{-1}),\qquad \Winv(T)=\Inv(\pi).
\]
\end{observation}

%


\medskip

By Observation~\ref{zigzag_Des_Winv}, there is a maj-winv
preserving bijection from SYT of given zigzag shape to
permutations in the correponding descent class.
Combining 
this with Theorem~\ref{GG} one obtains

\begin{proposition}\label{AR.t.q_zigzag}
For every zigzag $z$ encoded by $S=\{s_1, \ldots, s_k\}\subseteq
[n-1]$ set $s_0:=0$ and $s_{k+1}:=n$. Then
\[
\sum\limits_{\sh(T)=z}q^{\maj(T)}=\sum\limits_{\sh(T)=z}q^{\winv(T)}=[n]_q!
\cdot \det \left(\frac{1}{[s_{j+1} - s_i]_q!} \right)_{i, j =
0}^{k}.
\]
\end{proposition}








\subsubsection{Two-rowed shapes}

The major index and (weak) inversion number are not
equidistributed over  SYT of two-rowed shapes. However,
$q$-enumeration of both is nice. 
Two different $q$-Catalan numbers appear in the scene. First,
consider, enumeration by major index.


\begin{proposition}\label{q-two-rows} 
For every $n\in\bbn$ and $0\le k\le n/2$
\[
\sum\limits_{\sh(T) = (n-k,k)} q^{\maj(T)} = \left[ n \atop k
\right]_q - \left[ n \atop k-1 \right]_q.
\]
In particular,
\[
\sum\limits_{\sh(T)=(m,m)} q^{\maj(T)} = q^m C_m(q)
\]
where
\[
C_m(q) = \frac{1}{[m+1]_q}\left[2m \atop m\right]_q
\]
is the $m$-th F\"urlinger-Hofbauer $q$-Catalan
number~\cite{Furlinger_Hofbauer}.
\end{proposition}


Hence

\begin{corollary}\label{q-total-two-rows}
\[
\sum\limits_{\height(T)\le 2}q^{\maj(T)}=\left[n\atop \lfloor
n/2\rfloor\right]_q.
\]
\end{corollary}

For a bijective proof and refinements see~\cite{BBES}.

\bigskip

The descent set is invariant under jeu de taquin. Hence the proof
of Theorem~\ref{height3} may be lifted to a $q$-analogue. Here is
a $q$-analogue of Theorem~\ref{height3}.

\begin{theorem}
The major index generating function over SYT of size $n$ and
height $\le 3$ is equal to
\[
m_n(q) = \sum_{k=0}^{\lfloor n/2 \rfloor} q^k \left[ n \atop 2k
\right]_q C_k(q).
\]
\end{theorem}




\bigskip

Furthermore, the following strengthened version of
Corollary~\ref{AR_t:RSK_cor1}(3) holds.

\begin{corollary} For every positive integer $k$
$$
\sum\limits_{\{T \in \SYT_n:\ \height(T)< k\}} {\bf
x}^{\Des(T)}=\sum\limits_{\{\pi \in \Avoid_n(\sigma_k):\
\pi^2=id\}} {\bf x}^{\Des(T)},
$$
where $\Avoid_n(\sigma_k)$ is the subset of all permutations in
$\Sc_n$ which avoid $\sigma_k:=[k,k-1,\dots,1]$.
\end{corollary}

\begin{proof}
Combine Theorem~\ref{AR_t:Schensted} with
Proposition~\ref{AR_t:RSK_Des}.
\end{proof}

\bigskip


Counting by inversions is associated with another $q$-Catalan
number.

\begin{definition}{\rm \cite{CarlitzRiordan}}
Define the {\dem Carlitz-Riordan $q$-Catalan number} $\tC_n(q)$ by
the recursion
\[
\tC_{n+1}(q):=\sum\limits_{k=0}^{n} q^k \tC_k(q) \tC_{n-k}(q)
\]
with $\tC_0(q):=1$.
\end{definition}

These polynomials are, essentially, generating functions for the
area under Dyck paths of order $n$.

\begin{proposition}{\rm \cite{Shynar}}
\[
\sum\limits_{\sh(T)=(n,n)}q^{\inv(T)} = \tC_n(q).
\]
\end{proposition}

\begin{proposition}{\rm \cite{Shynar}}
For $0 \le k\le n/2$ denote
$G_k(q):=\sum\limits_{\sh(T)=(n-k,k)}q^{\inv(T)}$. Then
\[
\sum\limits_{k=0}^{\lfloor n/2 \rfloor}q^{\binom{n-2k}{2}}
G_k(q)^2 = \tC_n(q).
\]
\end{proposition}



\medskip

%


Enumeration of two-rowed SYT by descent number was studied by
Barahovski.

\begin{proposition}{\rm \cite{Barahovski}}
For $m \ge k \ge 1$,
\[
\sum_{\sh(T)=(m,k)} t^{\des(T)} = \sum_{d =1}^{k}
\frac{m-k+1}{k}{k \choose d}{m \choose d-1} t^d
\]
\end{proposition}



%

\subsection{The general case}

\subsubsection{Counting by descents}

There is a nice formula, due to Gessel, for the
number of SYT of a given shape $\la$ with a given descent set.
Since it involves as a scalar product of symmetric functions,
which are out of the scope of the current survey, we refer the
reader to the original paper~\cite[Theorem 7]{Gessel1984}.

There is also a rather complicated formula of
Kreweras~\cite{Kreweras66, Kreweras67} for the generaing function
of descent number on SYT of a given shape. However, the
first moments of the distribution of this statistic may be
calculated quite easily.

\begin{proposition}\label{one-des} For every partition $\lambda\vdash n$ and
$1\le k \le n-1$
\[
\#\{T\in \SYT(\lambda) :\, k \in \Des(T)\} =
\left( \frac{1}{2} - \frac{\sum_{i} \binom{\lambda_i}{2} - \sum_j
\binom{\lambda_j'}{2}}{n(n-1)} \right) f^\lambda.
\]
Here $\lambda_i$ is the length of the $i$-th row in the diagram $[\la]$,
and $\lambda'_j$ is the length of the $j$-th column.
\end{proposition}

For proofs see~\cite{Hasto, AR-Random}.

One deduces that $\#\{T\in \SYT(\la) :\, k \in \Des(T)\}$ is independent of $k$.
This phenomenon may be generalized as follows:
For any composition $\mu=(\mu_1, \ldots, \mu_t)$ of $n$ let
\[
S_\mu := \{\mu_1, \mu_1+\mu_2, \ldots, \mu_1+\ldots+\mu_{t-1}\}
\subseteq [1, n-1].
\]
The underlying partition of $\mu$ is obtained by
reordering the parts in a weakly decreasing order.

\begin{theorem}
For every partition $\la \vdash n$ and any two compositions
$\mu$ and $\nu$ of $n$ with same underlying partition,
\[
\sum_{\sh(T)=\la} {\bf x}^{\Des(T)\setminus S_\mu}=
\sum_{\sh(T)=\la} {\bf x}^{\Des(T)\setminus S_\nu}.
\]
\end{theorem}



%


%


\medskip

Proposition~\ref{one-des} implies that

\begin{corollary}
The expected descent number of a random SYT of shape
$\lambda\vdash n$ is equal to
\[
\frac{n-1}{2} - \frac{1}{n}\left(\sum\limits_{i}
\binom{\lambda_i}{2}-\sum\limits_j \binom{\lambda_j'}{2}\right).
\]
\end{corollary}

The variance of descent number was computed in~\cite{AR-Random, Hasto}, implying
a concentration around the mean phenomenon. The proofs
in~\cite{AR-Random} involve character theory, while those in~\cite{Hasto}
follow from a careful examination of the hook length bijection of Novelli, Pak and
Stoyanovskii~\cite{NovelliPakSto97}, described in Subsection~\ref{bijective} above.

\subsubsection{Counting by major index}

Counting SYT of general ordinary shape by descents is difficult.
Surprisingly, it was discovered by Stanley that counting by major index
leads to a natural and beautiful $q$-analogue of the
ordinary hook length formula (Proposition~\ref{t.AR_num_ordinary_hook}).

\begin{theorem}\label{t.AR_q_ordinary_hook}%
{\rm\ ($q$-Hook Length Formula)~\cite[Corollary
7.21.5]{Stanley_EC2}} For every partition $\la \vdash n$
\[
\sum_{T \in \SYT(\la)} q^{\maj(T)} = q^{\sum_i {{\lambda_i'\choose
2}}}\frac{[n]_q!}{\prod_{c \in [\la]} [h_{c}]_q}.
\]
\end{theorem}

This result follows from a more general identity, showing that the
major index generating function for SYT of a skew shape is essentially
the corresponding skew Schur function~\cite[Proposition 7.19.11]{Stanley_EC2}.
%
If $|\la/\mu| = n$ then
\[
s_{\la/\mu}(1, q, q^2, \ldots) = \frac{\sum_{T \in \SYT(\la/\mu)}
q^{\maj(T)}}{(1-q)(1-q^2) \cdots (1-q^n)}.
\]

\medskip

An elegant $q$-analogue of Schur's shifted product formula
(Proposition~\ref{t.AR_num_shifted_prod}) was found by Stembridge.

\begin{theorem}{\rm \cite[Corollary 5.2]{Stembridge}}
For every strict partition $\la=(\la_1,\dots,\la_t) \vdash n$
\[
\sum_{T \in \SYT(\la^*)} q^{n\cdot\des(T)-\maj(T)} =
\frac{[n]_q!}{\prod_{i=1}^t [\la_i]_q!}\cdot
\prod_{(i,j):i<j}\frac{q^{\la_j}-q^{\la_i}}{1-q^{\la_i+\la_j}}.
\]
\end{theorem}

Theorem~\ref{t.AR_q_ordinary_hook} may be easily generalized to
$r$-tableaux.

\begin{corollary}
For every $r$-partition $\la = (\la^1, \ldots, \la^r)$ of total size $n$
\[
\sum_{T \in \SYT(\la)} q^{\maj(T)} =
q^{\sum_i {{\la_i'\choose 2}}} \cdot \frac{[n]_{q}!}{\prod_{c \in [\la]} [h_{c}]_{q}}.
\]
\end{corollary}

The proof relies on a combination of Theorem~\ref{GG} with
the Stanton-White bijection for colored permutations.



\subsubsection{Counting by inversions}

Unlike descent statistics, not much is known about enumeration by
inversion statistics in the general case.
The following result was conjectured by Stanley~\cite{Stanley-sign} and proved,
independently, by Lam~\cite{Lam}, Reifegerste~\cite{Reifegerste}
and Sj\"ostrand~\cite{Sjostrand}.

\begin{theorem}
\[
\sum_{\la \vdash n} \sum_{T \in \SYT(\la)} \sgn(T) = 2^{\lfloor
n/2 \rfloor}.
\]
\end{theorem}

A generalization of Foata's bijective proof of MacMahon's
fundamental equidistribution theorem (Theorem~\ref{AR_t:MM}) to
SYT of any given shape was described in~\cite{Haglund}, using a
more involved concept of inversion number for SYT.

\section{Counting reduced words}\label{AR_s:words}

An interpretation of SYT as reduced words is presented in this
section. This interpretation is based on Stanley's seminal
paper~\cite{Stanley-words} and follow ups. For further reading
see~\cite[\S 7.4-7.5]{Bjorner-Brenti} and~\cite[\S
7]{Bjorner-Stanley}.

\subsection{Coxeter generators and reduced words}

Recall that the symmetric group $\Sc_n$ is generated by the set
$S:=\{s_i:\ 1\le i <n\}$ subject to the defining Coxeter
relations:
 \[
 s_i^2=1\ \ (1\le i<n);\ \ \ \ \ \ s_is_j=s_js_i\ \ (|j-i|>1); \ \ \ \ \ \ s_is_{i+1}s_i=s_{i+1}s_is_{i+1}\ \ (1\le i<n-1).
\]
The elements in $S$ are called simple reflections and may be
identified as adjacent transposition in  $\Sc_n$, where
$s_i=(i,i+1)$.

\bigskip

The Coxeter length of a permutation $\pi\in \Sc_n$ is
\[
\ell(\pi):=\min \{t:\ s_{i_1}\cdots s_{i_t}=\pi\}
\]
the minimal length of an expression of $\pi$ as a product of
simple reflections.

\begin{claim}
For every $\pi\in \Sc_n$
\[
\ell(\pi)=\inv(\pi).
\]
\end{claim}

A series of Coxeter generators $(s_{i_1},\dots, s_{i_t})$ is a
reduced word of $\pi\in \Sc_n$ if the resulting product is a
factorization of minimal length of $\pi$, that is
 $s_{i_1}\cdots s_{i_t}=\pi$ and $t=\ell(\pi)$.   In this section, the enumeration of
reduced words will be reduced to enumeration of SYT.


\subsection{
Ordinary and skew shapes}

A {\dem shuffle} of the two sequences $(1, 2, \ldots, k)$ and $(k
+ 1, k + 2, \ldots, \ell)$ is a permutation $\pi\in \Sc_{\ell}$ in
which both sequences appear as subsequences.
For example ($k = 3$, $\ell = 7$): $4516237$.



\begin{proposition}\label{t.AR_num_shuffles}{\rm ~\cite{Elizalde-R}}
There exists a bijection $\la \mapsto \pi_\la$ from the set of all
partitions to the set of all fixed point free shuffles such that:
\begin{enumerate}
\item If $\la$ has height $k$, width (length of first row) $\ell -
k$ and size $n$ then $\pi_\la$ is a fixed point free shuffle of
$(1, 2, \ldots, k)$ and $(k+1, k+2, \ldots, \ell)$ with
$\inv(\pi_\la) = n$. \item The number of SYT of shape $\lambda$ is
equal to the number of reduced words of $\pi_\la$.
\end{enumerate}
\end{proposition}

\begin{proof}[Proof sketch]
For the first claim, read the permutation from the shape as
follows: Encode the rows by $1, 2, \ldots, k$ from bottom to top,
and the columns by $k+1, k+2, \ldots, \ell$ from left to right.
Then walk along the SE boundary from bottom to top. If the $i$-th
step is horizontal set $\pi_\la(i)$ to be its column encoding;
otherwise set $\pi_\la(i)$ to be its row encoding.

\begin{example} The shape
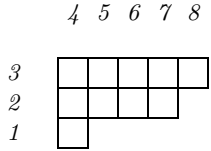
\[
\ytableausetup{baseline}
\ytableaushort{\none\none{\none[{\ts4}]}{\none[\ts5]}{\none[\ts6]}{\none[\ts7]}{\none[\ts8]},
\none, {\none[\ts3]}\none{}{}{}{}{}, {\none[\ts2]}\none{}{}{}{},
{\none[\ts1]}\none{}} \ytableausetup{nobaseline}
\]
corresponds to the shuffle permutation
$$\pi =41567283.$$
\end{example}

For the second claim, read the reduced word from the SYT as
follows:
If the letter $1 \le j \le n$ lies on the $i$-th diagonal (from
the left), set the $j$-th letter in the word (from right to left)
to be $s_i$.

\begin{example} The SYT
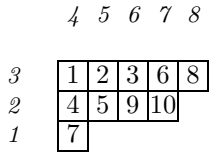
\[
\ytableausetup{baseline}
\ytableaushort{\none\none{\none[{\ts4}]}{\none[\ts5]}{\none[\ts6]}{\none[\ts7]}{\none[\ts8]},
\none, {\none[\ts3]}\none12368, {\none[\ts2]}\none459{10},
{\none[\ts1]}\none7} \ytableausetup{nobaseline}
\]
corresponds to the reduced word (in adjacent transpositions)
\[
s_5 s_4 s_7 s_1 s_6 s_3 s_2 s_5 s_4 s_3 = 41567283.
\]
\end{example}

The proof that this map is a bijection from all SYT of shape
$\lambda$ to all reduced words of $\pi_\lambda$ is obtained by
induction on the size of $\lambda$.

\end{proof}

\begin{corollary}\label{t.AR_num_shuffles_rectangle}
For every pair of positive integers $1 \le k \le \ell$, the number
of reduced words of the permutation $[k+1, k+2, \ldots, \ell, 1,
2, \ldots, k]$ is equal to the number of SYT of rectangular shape
$(k^{\ell - k})$.
\end{corollary}


\begin{proposition}\label{t.AR_num_skew_321}
There exists an injection from the set of all $321$-avoiding
permutations to the set of all skew shapes, which satisfies the
following property: For every $321$-avoiding permutation $\pi$
there exists a skew shape $\lambda/\mu$ such that the number of
reduced words of $\pi$ is equal to the number of SYT of shape
$\lambda/\mu$.
\end{proposition}

\medskip




The following theorem was conjectured and first proved by Stanley
using symmetric functions~\cite{Stanley-words}. A bijective proof
was given later by Edelman and Greene~\cite{Edelman-Greene}.

\begin{theorem}\label{AR_reduced_words}{\rm \cite[Corollary 4.3]{Stanley-words}}
The number of reduced words (in adjacent transpositions) of the
longest permutation $w_0:=[n,n-1,...,1]$ is equal to the number of
SYT of staircase shape $\delta_{n-1} = (n-1,n-2,...,1)$.
\end{theorem}


Corollary~\ref{t.AR_num_shuffles_rectangle} and
Theorem~\ref{AR_reduced_words} are special instances of the
following remarkable result.

\begin{theorem}\label{2143}{\rm \cite{Stanley-words, Edelman-Greene}}
\begin{enumerate}
\item For every permutation $\pi\in \Sc_n$, the number of reduced
words of $\pi$ can be expressed as
\[
\sum_{\la \vdash \inv(\pi)} m_\la f^\la
\]
where $m_\la$ are nonnegative integers canonically determined by
$\pi$. \item The above sum is a unique $f^{\la_0}$ (i.e., $m_\la =
\delta_{\la, \la_0}$) if and only if  $\pi$ is $2143$-avoiding.
\end{enumerate}
\end{theorem}




\medskip

Reiner~\cite{Reiner-note} applied Theorem~\ref{AR_reduced_words}
to show that the expected number of subwords of type $s_i
s_{i+1}s_i$ and $s_{i+1}s_i s_{i+1}$ in a random reduced word of
the longest permutation is exactly one. He conjectured that the
distribution of their number is Poisson. For some recent progress
see~\cite{Tenner}.



\bigskip

A generalization of Theorem~\ref{AR_reduced_words} to type $B$
involves square shapes. The following theorem was conjectured by
Stanley and proved by Haiman.

\begin{theorem}\label{AR_reduced_words_B}{\rm \cite{Haiman1992}}
The number of reduced words (in the alphabet of Coxeter
generators) of the longest signed permutation
$w_0:=[-1,-2,...,-n]$ in $B_n$ is equal to the number of SYT of
square shape.
\end{theorem}

For a recent application see~\cite{Petersen-Serrano}.

\subsection{
Shifted shapes}


An interpretation of the number of SYT of a shifted shape was
given by Edelman. Recall the left weak order from
Section~\ref{AR_s:appetizer}, and recall that $\sigma$ covers
$\pi$ in this order if $\sigma=s_i \pi$ and
$\ell(\sigma)=\ell(\pi)+1$. Edelman considered a modification of
this order in which we further require that the letter moved to
the left be larger than all letters that precede it.


\begin{theorem}{\rm \cite[Theorem 3.2]{Edelman}}
The number of maximal chains in the modified weak order
is equal to the number of SYT of shifted staircase shape.
\end{theorem}

\medskip

A related interpretation of SYT of shifted shapes was given
in~\cite{Elizalde-R}. A permutation $\pi\in \Sc_n$ is {\dem
unimodal} if $\Des(\pi^{-1})=\{1,\dots,j\}$ for some $0\le j \le
n-1$. Consider $U_n$, the set of all unimodal permutations in
$\Sc_n$, as a poset under the left weak order induced from
$\Sc_n$.

\begin{proposition}\label{t.AR_num_shifted_unimodal}{\rm \cite{Elizalde-R}}\
There exists a bijection $\la \mapsto \pi_\la$ from the set of all
shifted shapes contained in the shifted staircase $\delta_{n-1} =
(n-1,n-2,\ldots,1)$ to the set $U_n$ of all unimodal permutations
in $\Sc_n$ such that:
\begin{enumerate}
\item $|\la| = \inv(\pi_\la)$. \item The number of SYT of shifted
shape $\la$ is equal to the number of maximal chains in the
interval $[id,\pi_\la]$ in $U_n$.
\end{enumerate}
\end{proposition}

\begin{proof}[Proof sketch]
Construct the permutation $\pi_\la$ from the shape $\la$ as
follows: Encode the rows of $\la$ by $1, 2, \ldots$ from top to
bottom, and the columns by $2, 3, \ldots$ from left to right. Then
walk along the SE boundary from bottom to top. If the $i$-th step
is horizontal, set $\pi_\la(i)$ to be its column encoding;
otherwise set $\pi_\la(i)$ to be its row encoding.

\begin{example} The shifted shape
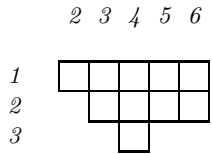
\[
\ytableausetup{baseline}
\ytableaushort{\none\none{\none[{\ts2}]}{\none[\ts3]}{\none[\ts4]}{\none[\ts5]}{\none[\ts6]},
\none, {\none[\ts1]}\none{}{}{}{}{},
{\none[\ts2]}\none\none{}{}{}{}, {\none[\ts3]}\none\none\none{}}
\ytableausetup{nobaseline}
\]
corresponds to the permutation
\[
\pi =4356217.
\]
\end{example}

Now construct the reduced word from the SYT $T$ as follows:
If the letter $j$ lies in the $i$-th diagonal (from left to right)
of $T$ then set the $j$-th letter in the word (from right to left)
to be $s_i$.

\begin{example} The SYT
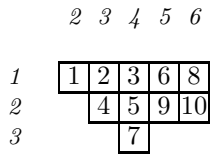
\[
\ytableausetup{baseline}
\ytableaushort{\none\none{\none[{\ts2}]}{\none[\ts3]}{\none[\ts4]}{\none[\ts5]}{\none[\ts6]},
\none, {\none[\ts1]}\none12368, {\none[\ts2]}\none\none459{10},
{\none[\ts3]}\none\none\none7} \ytableausetup{nobaseline}
\]
corresponds to the reduced word (in adjacent transpositions)
\[
s_4 s_3 s_5 s_1 s_4 s_2 s_1s_3 s_2 s_1 =4356217.
\]

\end{example}

\end{proof}





\section{Appendix 1: Representation theoretic aspects}

Representation theory may be considered as the birthplace of SYT;
in fact, one cannot imagine group representations without the
presence of SYT. Representation theory has been intimately related
to combinatorics since its early days. The pioneering work of
Frobenius, Schur and Young made essential use of integer
partitions and tableaux. In particular, formulas for restriction,
induction and decomposition of representations, as well as many
character formulas, involve SYT. On the other hand, it is well
known that many enumerative problems may be solved using
representations. In this survey we restricted the discussion to
combinatorial approaches. It should be noted that most
results have representation theoretic proofs, and in many cases
the discovery of the enumerative results was motivated by
representation theoretic problems.

In this section we briefly point on several connections, assuming
basic knowledge in non-commutative algebra, and give a very short
sample of applications.




\subsection{Degrees and enumeration}

A SYT $T$ of shape $\la$ has an associated group algebra element
$y_T\in \bbc[\Sc_n]$, called the {\dem Young symmetrizer}. 
$y_T$ has a key role: It is an idempotent, and its principal right
ideal
\[
y_T\bbc[\Sc_n]
\]
is an irreducible module of $\Sc_n$. All irreducible modules are
generated, up to isomorphism, by Young symmetrizers and two
modules, which are generated by the Young symmetrizers of two SYT
are isomorphic if and only if these SYT have same shape. The
irreducible characters of the symmetric group $\Sc_n$ over $\bbc$
are, thus, parameterized by the integer partitions of $n$.

\begin{proposition}\label{RT1}
The degree of the character indexed by $\la\vdash n$ is equal to
$f^\la$, the number of SYT of the ordinary shape $\la$.
\end{proposition}

This phenomenon extends to skew and shifted shapes. The number of
SYT of skew shape $\la/\mu$, $f^{\la/\mu}$, is equal to the degree
of the decomposable module generated by a Young symmetrizer of a
SYT of shape $\la/\mu$. Projective representations are indexed by
shifted shapes; the number of SYT of shifted shape $\la$, $g^\la$,
is equal to the degree of the associated projective
representation.


\bigskip


Most of the results in this survey have representation theoretic
proofs or interpretations. A few examples will be given here.

\begin{proof}[Proof sketch of Proposition~\ref{two-rows}]
The symmetric group $\Sc_n$ acts naturally on subsets of size $k$.
The associated character, $\mu^{(n-k,k)}$, is multiplicity free;
its decomposition into irreducibles is
\begin{equation}\label{RT_eq1}
\mu^{(n-k,k)}=\sum\limits_{i=0}^k \chi^{(n-i,i)}.
\end{equation}
Hence
\[
\chi^{(n-k,k)}=\mu^{(n-k,k)}-\mu^{(n-k+1,k-1)}.
\]
The degrees thus satisfy
\[
f^{(n-k,k)}=\chi^{(n-k,k)}(1)=\mu^{(n-k,k)}(1)-\mu^{(n-k+1,k-1)}(1)={n\choose
k}-{n\choose k-1}.
\]
\end{proof}

This argumentation may be generalized to prove
Theorem~\ref{t.AR_num_ordinary_det}. First, notice that
(\ref{RT_eq1}) is a special case of the Young rule for decomposing
permutation modules. The Young rule implies the determinantal
Jacobi-Trudi formula for expressing an irreducible module as an
alternating sum of permutation modules, see e.g.~\cite{JK}.
Evaluation of the characters at the identity permutation implies
Theorem~\ref{t.AR_num_ordinary_det}.

\medskip

Next proceed to identities which involve sums of $f^\la$-s.

\medskip

\begin{proof}[Proof of Corollary~\ref{AR_t:RSK_cor1}(1)]
Recall that for every finite group, the sum of squares of the
degrees of the irreducibles is equal to the size of the group.
This fact together with the interpretation of the $f^\la$-s as
degrees of the irreducibles of the symmetric group $\Sc_n$
(Proposition~\ref{RT1}) completes the proof.
\end{proof}

The same proof yields Theorem~\ref{sum_r}(1).

\medskip

The Frobenius-Schur indicator theorem implies that for every
finite group, which may be represented over $\bbr$, the sum of
degrees of the irreducibles is equal to the number of involutions
in the group, implying  Corollary~\ref{AR_t:RSK_cor1}(2). The
proof of Theorem~\ref{sum_r}(2) is similar, see e.g.~\cite{BG,
APR_Gelfand_wreath}.

\medskip



\begin{proof}[Proof sketch of Corollary~\ref{even_parts}]
The permutation module, defined by the action of $\Sc_{2n}$ on the
cosets of $B_n=\bbz_2\wr \Sc_n$ is isomorphic to a multiplicity
free sum of all $\Sc_{2n}$-irreducible modules indexed by
partitions with all parts even~\cite[\S VII (2.4)]{Md}. Comparison
of the dimensions completes the proof.
\end{proof}





\subsection{Characters and $q$-enumeration}

The Murnaghan-Nakayama rule is a formula for computing values of
irreducible $\Sc_n$-characters as signed enumerations of rim hook
tableaux. Here is an example of special interest.


\begin{proposition}\label{AR_t:RT2}
For every $\la\vdash rn$, the value of the irreducible character
$\chi^\la$ at a conjugacy class of cycle type $r^n$ is equal to
the number of $r$-rim hook tableaux; namely,
\[
\chi^\la_{(r,\ldots,r)}=f^\la_r.
\]
\end{proposition}

Another interpretation of $f^\la_r$ is as the degree of an
irreducible module of the wreath product $\bbz_r\wr \Sc_n$.

\medskip

An equivalent formula for the irreducible character values is by
weighted counts of all SYT of a given shape by their descents; 
$q$-enumeration then amounts to a computation of the
corresponding Hecke algebra characters. 

\bigskip


These character formulas may be applied to counting SYT by
descents. Here is a simple example.

\smallskip

\begin{proof}[Proof sketch of Proposition~\ref{one-des}]
By the Murnaghan-Nakayama rule, the character of $\chi^\la$ at a
transposition $s_i=(i,i+1)$ is equal to
\[
|\{T\in \SYT(\la):\ i\not\in\Des(T)\}|- |\{T\in \SYT(\la):\
i\in\Des(T)\}|=f^\la-2|\{T\in \SYT(\la):\ i\in\Des(T)\}|.
\]
Combining this with the explicit formula for this
character~\cite{I}
\[
\chi^\la_{(2,1^{n-2})} = \frac{\sum_i {\la_i \choose 2} - \sum_j
{\la'_j \choose 2}}{{n \choose 2}} f^\la.
\]
completes the proof.
\end{proof}


\bigskip


Finally, we quote two classical results, which apply enumeration
by major index.

\begin{theorem}{\rm (Kra{\'s}kiewicz-Weyman, in a widely circulated manuscript
finally published as~\cite{Kraskiewicz_Weyman})} Let $\omega$ be a
primitive $1$-dimensional character on the cyclic group $C_n$ of
order $n$.
Then, for any partition $\la$ of $n$, the multiplicity of
$\chi^{\la}$ in the induced character $Ind_{C_n}^{S_n} \omega$,
which is also the character of the $S_n$ action on the multilinear
part of the free Lie algebra on $n$ generators (and of many other
actions on combinatorial objects) is equal to the number of $T \in
\SYT(\la)$ with $\maj(T) \equiv 1 \pmod n$.
\end{theorem}

This result may actually be deduced from the following one.


\begin{theorem}{\rm (Lusztig-Stanley)}
For any partition $\la$ of $n$ and any $0 \le k \le {n \choose
2}$, the multiplicity of
$\chi^{\la}$ in the character of the $S_n$ action on the $k$-th
homogeneous component of the coinvariant algebra is equal to the
number of $T \in \SYT(\la)$ with $\maj(T) = k$.
\end{theorem}





A parallel powerful language is that of symmetric functions. The
interested reader is referred to the excellent textbooks~\cite[Ch.
7]{Stanley_EC2} and~\cite{Md}.



\section{Appendix 2: Asymptotics and probabilistic aspects}

An asymptotic formula is sometimes available when a simple
explicit formula is not known. Sometimes, such formulas do lead to
the discovery of surprising explicit formulas. A number of
important asymptotic results will be given in this appendix.


\medskip

Recall the exact formulas
(Corollary~\ref{total-two-rows}, Theorem~\ref{height3} and
Theorem~\ref{height4}) for the total number of SYT of ordinary
shapes with small height. An asymptotic formula for the total
number of SYT of bounded height was given by
Regev~\cite{Regev-height}; see also~\cite{Berele_Regev}%
\cite{Stanley_ICM_2007}\cite{Regev-height-new}.

\begin{theorem}\label{AR_t:regev_sum}{\rm \cite{Regev-height}}
Fix a positive integer $k$ and a positive real number $\alpha$.
Then, asymptotically as $n \to \infty$,
\[
F_{k,\alpha}(n) := \sum_{\la \vdash n\atop \ell(\la)\le k}
(f^\la)^{2 \alpha} \,\sim\, k^{2 \alpha n} \cdot n^{-\frac{1}{2}
(k - 1)(\alpha k + 2 \alpha - 1)} \cdot c(k, \alpha),
\]
where
\[
c(k, \alpha) := k^{\frac{1}{2} k (\alpha k + \alpha - 1)} (2
\alpha)^{-\frac{1}{2} (k - 1)(\alpha k + 1)} (2 \pi)^{-\frac{1}{2}
(k - 1)(2 \alpha - 1)} \prod_{i = 1}^{k} \frac{\Gamma(i
\alpha)}{\Gamma(\alpha)}.
\]
\end{theorem}
In particular, $\lim_{n \to \infty} {F_{k,\alpha}(n)}^{1/n} = k^{2
\alpha}$. Important special cases are: $\alpha = 1$, which gives
(by Theorem~\ref{AR_t:RSK_bijection} and
Proposition~\ref{AR_t:Schensted}) the asymptotics for the number
of permutations in $\Sc_n$ which do not contain a decreasing
subsequence of length $k+1$; and $\alpha = 1/2$, which gives (by
Corollary~\ref{AR_t:RSK_cor1}(3)) the asymptotics for the number
of involutions in $\Sc_n$ with the same property.
See~\cite{Regev-height} for many other applications.

The proof of Theorem~\ref{AR_t:regev_sum} uses the hook length
formula for $f^\la$, factoring out the dominant terms from the sum
and interpreting what remains (in the limit $n \to \infty$) as a
$k$-dimensional integral. An explicit evaluation of this integral,
conjectured by Mehta and Dyson~\cite{Mehta-Dyson, Mehta},
has been proved 
using Selberg's integral formula~\cite{Selberg}.


\medskip


Okounkov and Olshanski~\cite{OO1} introduced and studied a
non-homogeneous analogue of Schur functions, the {\dem shifted
Schur function}. As a combinatorial application, they gave an
explicit formula for the number of SYT of skew shape $\la/\mu$.
Stanley~\cite{Stanley_FPSAC} proved a formula for $f^{\la/\mu}$ in
terms of values of symmetric group characters. By applying the
Vershik-Kerov $\Sc_\infty$-theory together with the
Okounkov-Olshanski theory of shifted Schur functions, he used that
formula to deduce the asymptotics of $f^{\la/\mu}$. See
also~\cite{Stanley-skew}

\medskip

Asymptotic methods were applied to show that certain distinct
ordinary shapes have the same multiset of hook
lengths~\cite{Regev-Vershik}. Bijective and other purely
combinatorial proofs were given later~\cite{RZ, Bes, Kratt_hooks,
GY}.



\medskip



In two seminal papers, Logan and Shepp~\cite{Logan}, and
independently Vershik and Kerov~\cite{Kerov}, studied the problem
of the {\dem limit shape} of the pair of SYT which correspond,
under the RS correspondence, to a permutation chosen uniformly at
random from $\Sc_n$. In other words, choose each partition $\la$
of $n$ with probability $\mu_n(\la) = (f^{\la})^2 / n!$. This
probability measure on the set of all partitions of $n$ is called
{\dem Plancherel measure}.


It was shown in~\cite{Logan, Kerov} that, under Plancherel
measure, probability concentrates near one asymptotic shape. See
also~\cite{Borodin}.

\begin{theorem}{\rm \cite{Logan, Kerov}}
Draw a random ordinary diagram of size $n$ in Russian notation
(see Subsection~\ref{AR_s:def_classical_shapes}) and scale it down
by a factor of $n^{1/2}$. Then, as $n$ tends to infinity, the
shape converges in probability, under Plancherel measure, to the
following limit shape:
\[
f(x)=
\begin{cases}
\frac{2}{\pi}(x \arcsin\frac{x}{2}+\sqrt{4-x^2}), & \hbox{if } |x|\le 2; \\
|x|, & \hbox{if } |x|> 2.
\end{cases}
\]
\end{theorem}

This deep result had significant impact on mathematics in recent
decades~\cite{Romik_book}.


\medskip

A closely related problem is to find the shape which maximizes
$f^\la$.
First, notice that 
Corollary~\ref{AR_t:RSK_cor1} implies that 
\[
\sqrt{\frac{n!}{p(n)}}\le \max \{f^\la :\, \la \vdash n\} \le
\sqrt {n!},
\]
where $p(n)$ is the number of partitions of $n$.


\begin{theorem}{\rm \cite{VK2}}
\begin{itemize}
\item[(1)] There exist constants $c_1 > c_0 > 0$ such that
\[
e^{-c_1 \sqrt n} \sqrt {n!}\le \max \{f^\la :\, \la \vdash n\} \le
e^{-c_0 \sqrt n} \sqrt {n!}.
\]
\item[(2)] There exists constants $c'_1 > c'_0 > 0$ such that
\[
\lim_{n \to \infty} \mu_n \left\{ \la \vdash n \,:\, c'_0 <
-\frac{1}{\sqrt n} \ln \frac{f^\lambda}{\sqrt{n !}} < c'_1 \right\}
= 1.
\]
\end{itemize}
\end{theorem}


Similar phenomena occur when Plancherel measure is replaced by
other measures. For the uniform measure see, e.g., \cite{Pittel, Pittel_02}.


\medskip

Motivated by the limit shape result, Pittel and Romik proved that
there exists a limit shape to the two-dimensional surface defined
by a uniform random SYT of rectangular shape~\cite{Pittel-Romik}.

\bigskip

Consider a fixed $(i,j)\in \bbz_+^2$ and a SYT $T$
chosen according to some probability distribution on the SYT of size $n$.
A natural task is to estimate the probability that $T(i, j)$ has a prescribed value.
Regev was the first to give an asymptotic answer to this problem,
for some probability measures,
using $\Sc_\infty$-theory~\cite{Regev-probability};
see also~\cite{Olshanski-Regev}.
A combinatorial approach was suggested later by McKay, Morse and
Wilf~\cite{Wilf}.
Here is an interesting special case.

\begin{proposition}\label{Regev-Wilf}{\rm \cite{Regev-probability}\cite{Wilf}}
For a random SYT $T$ of order $n$ and a positive integer $k>1$
\[
\Prob(T(2,1) = k) \sim \frac{k-1}{k!}+O(n^{-3/2}).
\]
\end{proposition}

In~\cite{Wilf}, Proposition~\ref{Regev-Wilf} was deduced from the
following theorem.

\begin{theorem}\label{MW}{\rm \cite{Wilf}}
Let $\mu\vdash k$ be a fixed partition and let $T$ be a fixed SYT
of shape $\mu$. Let $n\ge k$ and let $N(n;T)$ denote the number of
SYT with $n$ cells that contain $T$. Then
\[
N(n;T)\sim \frac{t_n f^\la}{k!},
\]
where $t_n$ denotes the number of involutions in the symmetric
group $\Sc_n$.
\end{theorem}

It follows that
\[
\sum\limits_{\la\vdash n}f^{\la/\mu}\sim \frac{t_n f^\la}{k!}.
\]
Stanley~\cite{Stanley_FPSAC}, applying techniques of symmetric
functions, deduced precise formulas for $N(n;T)$ in the form of
finite linear combinations of the $t_n$-s




%









\bibliographystyle{amsplain}
\bibliography{SYT_enumeration_bibtex_file} 

\end{document}